\documentclass[graybox]{svmult}

\usepackage{mathptmx}       
\usepackage{helvet}         
\usepackage{courier}        
\usepackage{type1cm}        
\usepackage{amssymb}                            
\usepackage{makeidx}         
\usepackage{graphicx}        
\usepackage{multicol}        
\usepackage[bottom]{footmisc}
\newcommand{\be}{\begin{equation}}
\newcommand{\ee}{\end{equation}}
\newcommand{\beq}{\begin{eqnarray}}
\newcommand{\eeq}{\end{eqnarray}}
\newcommand{\la}{\langle}
\newcommand{\ra}{\rangle}
\newcommand{\rt}{\rightarrow}
\newcommand{\lrt}{\longrightarrow}

\newcommand{\q}{\quad}

\newcommand{\text}{\mbox}

\newtheorem{assumption}{Assumption}
\newtheorem{algorithm}{Algorithm}

\makeindex             

\begin{document}

\title*{New Trends in General Variational Inequalities  }
 \titlerunning{General variational inequalities  }
\author{Muhammad Aslam Noor, Khalida Inayat Noor and  Michael Th. Rassias}
\authorrunning{M. A. Noor, K. I. Noor, M. T. Rassias}
\institute{Muhammad Aslam Noor \at COMSATS University Islamabad, Park Road, Islamabad, Pakistan\\
\email{noormaslam@gmail.com}
\and Khalida Inayat Noor \at COMSATS University Islamabad, Park Road, Islamabad, Pakistan\\
\email{khalidan@gmail.com}
\and  Michael Th. Rassias \at  Institute of Mathematics, University of Zurich, CH-8057, Zurich, Switzerland\\
  Moscow Institute of Physics and Technology, 141700 Dolgoprudny, Institutskiy per, d. 9, Russia
 Institute for Advanced Study, Program in Interdisciplinary Studies, 1 Einstein Dr, Princeton, NJ 08540, USA.\\
\email{michail.rassias@math.uzh.ch}}
%
%
\maketitle

\abstract*{ It is well known that general variational inequalities provide
us with a unified, natural, novel and simple framework to study a wide
class of unrelated problems, which arise  in pure and applied sciences. In
this paper, we present a number of new and known numerical techniques
for solving general variational inequalities and equilibrium problems  using various techniques
including projection, Wiener-Hopf equations, dynamical systems,
auxiliary principle and  penalty function.  General variational-like inequalities are introduced and investigated.
Properties of higher order strongly general convex functions have been discussed.
The auxiliary principle technique is used to suggest and analyze some iterative methods for solving
higher order general variational inequalities. Some new classes of strongly exponentially general convex functions
are introduced and discussed. Our proofs of convergence are very simple as compared with
other methods. Our results present a significant improvement of previously
known methods for solving variational inequalities and related
optimization problems. Since the general variational inequalities include (quasi) variational
inequalities and (quasi) implicit complementarity problems as special
cases, these results  continue to hold for these problems. Some numerical results are included
to illustrate the efficiency of the proposed methods.
Several open problems have been suggested for further research in these
areas.
}

\abstract{ It is well known that general variational inequalities provide
us with a unified, natural, novel and simple framework to study a wide
class of unrelated problems, which arise  in pure and applied sciences. In
this paper, we present a number of new and known numerical techniques
for solving general variational inequalities and equilibrium problems  using various techniques
including projection, Wiener-Hopf equations, dynamical systems,
auxiliary principle and  penalty function.  General variational-like inequalities are introduced and investigated.
Properties of higher order strongly general convex functions have been discussed.
The auxiliary principle technique is used to suggest and analyze some iterative methods for solving
higher order general variational inequalities. Some new classes of strongly exponentially general convex functions
are introduced and discussed.  Our proofs of convergence are very simple as compared with
other methods. Our results present a significant improvement of previously
known methods for solving variational inequalities and related
optimization problems. Since the general variational inequalities include (quasi) variational
inequalities and (quasi) implicit complementarity problems as special
cases, these results   continue to hold for these problems. Some numerical results are included
to illustrate the efficiency of the proposed methods.
Several open problems have been suggested for further research in these
areas.
}

\section{Introduction}

Variational inequalities theory, which was introduced Stampacchia \cite{172} and Ficchera \cite{38} indecently  has
emerged as an interesting and fascinating branch of applicable mathematics
with a wide range of applications in industry, finance, economics, social,
pure and applied sciences. Variational inequalities may be viewed as novel generalization of the
variational principles, the origin of which can be traced back to Euler, Lagrange and Bernoulli brothers.
The variational principles have placed a crucial and important in the development of various fields
of sciences and have appeared as a unifying force.  The ideas
and techniques of variational inequalities are being applied in a variety
of diverse areas of sciences and prove to be productive and innovative.
 Variational inequalities have been extended and generalized in several directions using novel and new
techniques. The minimum of a differentiable convex function $F$ on the convex set $K$
is equivalent to finding $u \in K $ such that
\begin{eqnarray}\label{1}
\langle F^{\prime}(u), v-u \rangle \geq 0, \quad \forall v\in K,
\end{eqnarray}
which is known as the variational inequality (\ref{1}). Here $F^{\prime}(u) $ is the Frechet differential.
Stampacchia \cite{172} proved that potential problems associated with elliptic equations can be studied by
 the variational inequality. This simple fact inspired a great interest in variational inequalities.
Lions and Stampacchia \cite{54} studied the existence of a solution of variational inequalities using essentially
the auxiliary principle technique coupled with projection idea.\\
Lemke \cite{62} considered the problem of finding $u\in R^{n}_{+} $ such that
\begin{eqnarray}\label{2}
u \geq 0, \quad Au \geq 0,  \quad \langle Au, u\rangle =0,
\end{eqnarray}
is called the linear complementarity problem. Here $A $ is a linear operator. Lemke \cite{62} proved that the two person game theory problems
can be studied in the framework of linear complementarity problem (\ref{2}).  See also  Lemke and Howson, Jr.\cite{63} and Cottle al et \cite{27} for the nonlinear complementarity problems.\\
It is worth mentioning that both the problems (\ref{1}) and (\ref{2}) are  different and have been studied in infinite dimensional
spaces and finite dimensional spaces independently using quite different techniques. However. Karmardian \cite{57} established that both problems
(\ref{1}) and (\ref{2}) are equivalent, if the underlying set $K $ is a convex cone. This equivalent formulation  played an important role in developing several techniques for solving these problems.\\

If the convex set $K$ depends upon the solution explicitly  or implicitly, then
variational inequality  is called the quasi variational inequality. Quasi variational inequalities
were introduced and investigated by Bensoussan and Lions \cite{15} in control theory. In fact, for a given operator
$T: H \lrt H,$ and a point-to-set mapping $K : u \lrt K(u), $ which associates a closed convex-valued set
$K$ with any element $u $ of $H,$ we consider the problem of finding $u
\in K(u)$ such that
\begin{eqnarray}\label{3}
\la Tu, v- u \ra \q \geq 0,   \q \forall v \in K(u),
\end{eqnarray}
which is known as quasi variational inequality. Chan and Pang \cite{22} considered
 the generalized quasi variational inequalities for set-valued operators.
Noor \cite{84} established the equivalence between the quasi variational inequalities and the fixed point formulation
and used this equivalence to suggest some iterative methods for solving (\ref{3}). This equivalence was used to study
 the existence of a solution of quasi variational inequalities and develop numerical methods.\\

Related to the quasi variational inequality, we have the problem of
finding $u \in H $ such that
\begin{eqnarray}\label{4}
u\geq m(u), \quad Tu \geq 0, \quad \langle Tu, u-m(u) \rangle =0,
\end{eqnarray}
is called the implicit(quasi) complementarity problem. where $m $ is a point-to-point mapping.
Using the technique of Karamardian \cite{57}, Pang \cite{154} and Noor \cite{84} established the equivalence between the problems (\ref{3}) and (\ref{4}).
Noor \cite{85,86}  has used the change of variables technique to prove that the implicit complementarity problems are equivalent to the
fixed point problem. This alternative formulation played an important part in the development of iterative methods
for solving various type of complementarity problems and related optimization problems.
It is an interesting problem to extend this technique for solving variational inequalities.\\

Motivated and inspired by the ongoing research in these fields, Noor \cite{87} introduced and investigated a new class of
variational inequalities involving two operators. For given nonlinear operators $T,g, $ consider the problem of finding
$u \in H: g(u) \in K, $ such that
\begin{eqnarray}\label{5}
\langle Tu, g(v)-g(u) \rangle \geq 0, \quad \forall v\in H: g(v) \in K,
\end{eqnarray}
 which is known as the {\bf general (Noor) variational inequalities. } It turned out that
odd-order and nonsymmetric obstacle, free, unilateral and moving
boundary value problems arising in  pure and applied sciences can be
studied via the general variational inequalities, see \cite{87, 88,89,90,91}.\\
If $K$ is a convex cone, then the  implicit complementarity problem (\ref{5}) is equivalent to finding $u \in H $ such that
\begin{eqnarray}\label{6}
g(u)\geq 0, \quad Tu \in K^*, \quad \langle Tu, g(u) \rangle =0,
\end{eqnarray}
which is known as the general complementarity problem, where $K^* $ is the dual (polar) cone.
We would like to point out that for appropriate and suitable choice of the operators $T, g $ and the convex sets $K, $ one can obtain
several known and new classes of variational inequalities and complementarity problems as special cases of the problem (\ref{5}).

During the  years which have been elapsed since its discovery, a
number of numerical methods including projection method and its variant
forms, Wiener-Hopf equations, auxiliary principle,
dynamical systems have been developed for solving the  variational
inequalities and related optimization problems.  Projection method
and its variants forms including the Wiener-Hopf equations represent
important tools for finding the approximate solution of variational
inequalities, the origin of which can be traced back to Lions and
Stampacchia [66]. The main idea in this technique is to establish the
equivalence between the variational inequalities and the fixed-point
problem by using the concept of projection. This alternative formulation
has played a significant part in developing various projection-type
methods for solving variational inequalities. It is well known that the
convergence of the projection methods requires that the operator must be
strongly monotone and Lipschitz continuous. Unfortunately these strict
conditions rule out many applications of this method. This fact motivated
to  modify the projection method or to  develop other methods.  The
extragradient-type methods overcome this difficulty by performing an additional forward step and a
projection at each iteration according to the double projection. These
methods can be viewed as predictor-corrector methods. Their convergence
requires only that a solution exists and the monotone operator is
Lipschitz continuous.  When the operator is not Lipschitz continuous or
when the Lipschitz continuous constant is not known, the extragradient
method and its variant forms require an Armijo-like line search procedure
to compute the step size with a new projection needed for each trial,
which leads to expansive computation. To overcomes these difficulties,
several modified projection and extragradient-type methods have been
suggested and developed for solving variational inequalities.
See \cite{6,7,8,9,10,17,18,24,36,37,39,42,43,44,49,55,63,68,71,79,82,83,90,92,95,100,101,
102,103,104,106,107,110,111,112,113,114,116,117,119,120,133,134,135,137,138,142,143,144,145,149,151,152,159}
and the references therein.\\

In Section 4,  we  have the concept of the general Wiener-Hopf equations,
which was introduced by Noor \cite{90}.
As a special case, we obtain the original Wiener-Hopf equations, which
were considered and studied by Shi \cite{167} and Robinson [165]  in
conjunction with variational inequalities from different point
of views. Using the projection technique, one usually establishes the
equivalence between the variational inequalities and the Wiener-Hopf
equations. It turned out that the Wiener-Hopf equations are more
general and flexible. This approach has played  not only an important
part in developing various efficient projection-type methods, but  also
in studying the sensitivity analysis, dynamical systems as well as other concepts of
variational inequalities. Noor, Wang and Xiu \cite{151,152} and Noor and
Rassias \cite{145} have suggested and analyzed some predictor-corrector type
projection methods by modifying the Wiener-Hopf equations. These methods are also
 known as Forward-Backward methods, see Tseng \cite{178,179}.
 It has been shown  that these predictor-corrector-type methods are
efficient and robust. Some numerical examples are given to illustrate the efficiency and
implementation of the proposed methods. Consequently, our results represent a refinement and
improvement of the known results.\\

Section 5 is devoted to the concept of projected dynamical system in the context
of variational inequalities, which  was introduced by Dupuis and Nagurney \cite{35} by
using the fixed-point formulation of the variational inequalities.
For the recent development and applications of the dynamical systems,
see \cite{13,34,35,40,41,56,75,79,109,115,116,122}. In this technique, we
reformulate the variational inequality problem as an initial value
problem. Using the discretizing of the dynamical systems, we suggest some new iterative methods
 for solving the general variational inequalities.\\

It is well known fact that to implement the projection-type methods,
one has to evaluate the projection, which is itself a difficult problem.
Secondly, the projection and Wiener-Hopf equations techniques can't be
extended and generalized for some classes of variational inequalities
involving the nonlinear (non)differentiable functions, see [92,94,108].
These facts motivated to use the auxiliary principle technique.  This
technique deals with finding the auxiliary variational inequality and
proving that the solution of the auxiliary problem is the solution of the
original problem by using the fixed-point approach. It turned out that
this technique can be used to find the equivalent differentiable
optimization problems, which enables us to construct gap (merit)
functions. Glowinski et al. \cite{47} used this technique to study the
existence of a solution of mixed variational inequalities. Noor \cite{93,94,95,100,101,114,121,122}
has used  this technique to suggest some predictor-corrector
methods for solving various  classes of variational inequalities. It is
well known that a substantial number of numerical methods can be obtained
as special cases from this technique. We use this technique to suggest and analyze
some explicit predictor-corrector  methods for general
variational inequalities. In this paper, we give the basic idea of the
inertial proximal methods and show that the auxiliary principle technique can be used  to
construct gap (merit) functions.  We use the gap function to consider an optimal control problem
governed by the general variational inequalities. The control
problem as an optimization problem is also referred as  a generalized
bilevel programming problem or mathematical programming with equilibrium
constraints. These results are mainly due to Deitrich \cite{32,33}. It is
an open problem to compare the efficiency of the inertial methods  with other methods
and this is another direction for future research.\\
In Section 7, we  discuss the application of the penalty function method,
which was introduced by Lewy and Stampacchia \cite{64} to study the regularity of
the solutions of the variational inequalities.
It is known that the finite difference and similar  numerical methods cannot be
applied to find the approximate solutions of the obstacle, free and moving
value problems due to the presence of the obstacle and other constraint
conditions. However, it is known that if the obstacle is known then these
obstacle and unilateral problems can be  characterized by a system
of differential equations in conjunction with the general
variational inequalities using the penalty function technique.
 Al-Said \cite{3}, Noor and Al-Said \cite{112},
Noor and Tirmizi \cite{147} and Al-Said et al.\cite{4,5} used this
technique to develop some numerical methods for solving these system of
differential equations.  The main advantage of  this technique is its
simple applicability for solving system of differential equations. We here
give only the main idea of this technique for solving odd-order obstacle
and unilateral problems.\\

In recent years, much attention has been given to study the equilibrium
problems, which were considered and studied by Blum and Oettli \cite{19}  and Noor and
Oettli \cite{141}. It is known that equilibrium problems include variational
inequalities and complementarity problems  as special cases. It is
remarked that there are very few iterative methods for solving equilibrium
problems, since the projection method and its variant forms including the
Wiener-Hopf equations cannot be extended for these problems. We use the auxiliary principle technique to suggest and
analyze some iterative type methods for solving general equilibrium problems, which is considered in Section 8.\\

Hanson \cite{50} introduced the concept of the invex functions to study the mathematical programming problems, which appeared to be
significant generalization of the convex functions. Ben-Israel and Mond \cite{14} considered the concepts of invex sets and preinvex functions.
They proved that the differentiable preinvex functions are invex function. Mohan and Neogy \cite{74} proved that the converse is
also true under certain conditions. Noor \cite{93} proved that the optimality conditions can be characterized by a class of variational inequalities,
which is called variational-like inequality. Due to the inherent nonlinearity, one can not use the projection type iterative methods for
considering the existence results and numerical methods for variational-like inequalities. However, one uses the auxiliary principle
technique to study the existence  and numerical methods for variational-like inequalities. Fulga and previda \cite{43} and Awan et al\cite{11}
considered the general invex sets and general preinvex functions involving an arbitrary function and studied their basic properties.
We  show that the minimum of the differentiable general preinvex functions is characterized  by a class of variational-like inequality.
This fact motivated us to introduce a the general variational-like inequalities and study its properties. We have used the auxiliary principle technique
to analyze some iterative methods for solving the general variational-like inequalities. Several special cases are discussed as applications of the general
variational-like inequalities. These aspects are discussed in Section 9.\\

In section 10, we consider the concept of higher order strongly general convex functions involving an arbitrary function, which can viewed as a novel and innovative extension of the strongly
convex functions. Polyak \cite{158} in 1966 introduced the strongly convex functions to study the optimization problems. Zu and Marcotte \cite{200} discussed
 the role of the strongly convex functions in the analysis of the iterative methods for solving variational inequalities.  Mohsen et al\cite{75} introduced the higher order strongly  convex functions involving bifunction, which can be viewed as significant refinement of the higher order strongly convex function, which were considered by Lin and Fukushima \cite{65} in mathematical programming with equilibrium constraints. They have shown that parallelogram laws for Banach spaces can be obtained as applications of the higher order strongly convex functions.
 Parallelogram laws for Banach spaces were analyzed by Bynum \cite{21}  and Chen at al \cite{23,24,25}, which have applications in prediction theory and information technology.  We have investigated some basic properties of the higher 0rder strongly general convex functions and have shown that the optimality conditions of the differentiable higher order strongly general convex functions can be expressed as higher order general variational inequalities.\\
Higher order general variational inequalities are introduced in Section 11. Some iterative methods are suggested and analyzed for solving higher order general variational inequalities. It is shown that general variational inequalities and related optimization problems  can be obtained as applications.\\

Related to the convex functions, we have the concept of
exponentially convex(concave) functions, which have important applications in information
theory, big data analysis, machine learning and statistic. Exponentially convex functions
have appeared significant generalization of the convex functions, the origin of which can be
traced back to Bernstein \cite{16}. Avriel\cite{9,10} introduced the concept of $r$-convex functions, from
which one can deduce the exponentially convex functions. Antczak \cite{2} considered the $(r, p)$
convex functions and discussed their applications in mathematical programming and optimization theory.
Alirazaie and Mahar \cite{1} investigated the impact of exponentially concave functions in information theory.
Zhao et al\cite{199} discussed some characterizations of $r$-convex functions.
Awan et al \cite{5} also investigated some classes of exponentially convex functions. Noor and Noor \cite{132,133,134,135,136,137,138}
discussed the characterization of several classes of exponentially convex functions.  In Section 12, we introduce the concept of strongly
exponentially general convex functions and show that the strongly exponentially enjoy some nice properties, which convex functions have. \\

Theory of general variational inequalities  is quite broad, so
we shall content ourselves here to give the flavour of the ideas and
techniques involved. The techniques used to analysis  the iterative
methods and other results for general variational inequalities  are a
beautiful blend of ideas of pure and applied mathematical sciences.
In this paper, we have presented the some  results
regarding the development of various algorithms, their convergence
analysis and penalty computational technique. Although this paper is expository
in nature, our choice has been rather to consider some interesting aspects
of general variational inequalities. The
framework chosen should be seen as a model setting for more general
results for other classes of variational inequalities and variational
inclusions. One of the main purposes of this expository paper is to
demonstrate the close connection among various classes of algorithms for
the solution of the general variational inequalities and to point out
that researchers in different field of variational inequalities and
optimization have been considering parallel paths.  We would like to emphasize
that the results obtained and discussed in this paper may motivate and bring a large number of novel,
innovate and potential applications, extensions and interesting topics
in these areas. The comparison of the proposed methods with other techniques
needs further efforts and is itself an open interesting problem
We have given only a brief introduction of the general variational inequalities
in the real Hilbert spaces. For some other aspects of the general variational inequalities, readers
are referred  to the state-of-art articles of Noor \cite{85,86,87,89,90,91,92,93,94,95,100,101,102,103,104,105,110,118,119,120,121,122,124}
and the references therein. The interested reader is advised to explore this field
further and discover  novel and fascinating applications of this theory in Banach and topological spaces.\\
It is perhaps part of the fascination of the subject that so many branches of pure and applied
sciences are involved in the variational inequality theory. The task of becoming conversant with
a wide spectrum of knowledge is indeed a real challenge. The general theory is quite technical,
so we shall content ourselves here to give the flavour of the main ideas involved. The
techniques used to analyze the existence results and iterative algorithms for variational
inequalities are a beautiful blend of ideas from different areas of pure and applied mathematical sciences.
The framework chosen should be seen as a model setting for more general results. However, by
just relying on these special results, interesting problems arising in the applications can be dealt with easily.
Our main motivation of this paper to give a summary account of the basic theory of variational inequalities set in the
framework of nonlinear operators defined on convex sets in a real Hilbert space. We focus our
attention on the iterative methods for solving variational inequalities. The equivalence between
the variational inequalities and the Wiener-Hopf equations has been used to suggest some new
iterative methods. The auxiliary principle technique is  applied to study the existence of the
solution and to propose a novel and innovative general algorithm for the general variational inequalities,
equilibrium problems and related optimization.

\section{Preliminaries and Basic concepts}

Let $H$ be a real Hilbert space, whose inner product
and norm are denoted by $\la \cdot , \cdot \ra$ and $\| \cdot \|$
respectively.\\

\begin{definition}The set $K$ in $H$ is said to be a convex  set, if
\begin{eqnarray*}
u+t(v-u)\in K,\quad \forall u,v\in K, t\in[0,1].
\end{eqnarray*}
\end{definition}
\begin{definition}A function $F$ is said to be a convex function, if
\begin{eqnarray*}
F((1-t)u+tv) \leq (1-t)F(u)+ tF(v), \quad \forall u,v \in K, \quad t \in [0,1].
\end{eqnarray*}
\end{definition}
It is well known that a function $F $ is a convex functions, if and only if, it satisfies  the inequality
\begin{eqnarray*}
F(\frac{a+b}{2}) \leq \frac{1}{b-a} \int^{b}_{a} F(x) dx \leq \frac{F(a)+F(b)}{2}, \quad \forall a, b \in I=[a,b],
\end{eqnarray*}
which is known as the Hermite-Hadamard type inequality. Such type of the inequalities provide us with the upper and lower bounds for the mean value integral.\\
If the  convex function $F $ is differentiable, then  $u \in K $ is the minimum of the $F, $ if and only if, $u\in K $ satisfies the inequality
\begin{eqnarray*}
\langle F^{\prime}(u), v-u \rangle \geq 0, \quad \forall v\in K,
\end{eqnarray*}
which is   called the variational inequality,  introduced and studied by Stampacchia \cite{172}  in 1964.
For the applications, formulation, sensitivity, dynamical systems, generalizations, and other aspects of the variational inequalities, see[1-200] and references  therein.  \\
It is known that a set may not be a convex set. However, a  set may  be made convex set with respect  to an arbitrary function.
Motivated by this fact, Youness \cite{196} introduced the concept of a general convex set involving an arbitrary function.
\begin{definition}
 The set $K $ in $H$ is said to be a general convex  set, if there exists an arbitrary function $g, $ such that
\begin{eqnarray*}
g(u)+t(g(v)-g(u))\in K_g,\quad\forall u,v\in H:  g(u),g(v) \in K, t\in[0,1].
\end{eqnarray*}
\end{definition}
Note that, if $g =I, $ the identity operator, then general convex set reduces to the classical convex set.
Clearly every convex set is a general convex set, but the converse is not true.\\
For the sake of simplicity, we always assume that $ \forall u,v \in H: g(u), g(v) \in K, $ unless otherwise.
\begin{definition}
 A function $F$ is said to a general convex function, if there exists an arbitrary function $g $ such that
\begin{eqnarray*}
F((1-t)g(u)+tg(v)) \leq (1-t)F(g(u)) +t F(g(v)), \forall v\in H: g(v) \in K, \quad t\in [0,1].
\end{eqnarray*}
\end{definition}
It is known that every convex function is a general convex function, but the converse is not true. For example, the function $ F(x) = e^{x^2} $ is a general convex function, but it is not convex.\\

We now define the  general convex functions on  $I_g= [g(a),g(b)].$
\begin{definition}
Let $I_g =[g(a)g(b)].$ Then $F$ is  a general convex function, if and only if,
\begin{eqnarray*}
\left|
\begin{array}{ccc}
1&1&1\\ g(a)&  g(x)&  g(b)\\
F(g(a))&  F(g(x))&F(g(b))
\end{array} \right|\geq0;\quad g(a)\leq g(x)\leq g(b).
\end{eqnarray*}
\end{definition}
\noindent One can easily show that the following are equivalent:
\begin{enumerate}
\item $F$ is a general convex function.\\
\item $F(g(x))\leq  F(g(a))+\frac{F(g(b))-F(g(a))}{g(b)-g(a)}(g(x)-g(a))$.\\
\item
$\frac{F(g(x))-F(g(a)}{g(x)-g(a)}\leq  \frac{F(g(b))-F(g(a))}{g(b)-g(a)}.$\\
\item
$(g(x)-g(a))F(g(a)) +(g(b)-g(a))F(g(x))+(g(a)-g(x))F(g(b))\geq 0.$\\
\item
$\frac{F(g(a))}{(g(b)-g(a))(g(a)-g(x))}+\frac{F(g(x))}{(g(x)-g(b))(g(a)-g(x))}+\frac{F(g(b))}{(g(b)-g(a))(g(x)-g(b))}\geq0$,
\end{enumerate}
where $g(x)= (1-t)g(a)+tg(b),  \in[0,1].$ \\ \\
We now show that the minimum of a differentiable general convex function on
$K$ in $H$  can be characterized by the general variational inequality. This result is mainly due to Noor \cite{110}.
\begin{theorem}\cite{110}\label{theorem1}  Let $F: K \longrightarrow H$ be a
differentiable  general convex function. Then \\$u \in H: g(u) \in K$ is the minimum of
a differentiable general convex function $F$ on $K,$ if and only if, $u \in : g(u) \in K$
satisfies the inequality
\begin{eqnarray}\label{2.1}
\la F'(g(u)), g(v)-g(u)\ra \geq 0, \quad \forall  g(v) \in K,
\end{eqnarray}
where $F'$ is the differential of $F$ at $g(u)\in K$ in the direction $g(v)-g(u).$
\end{theorem}
\begin{proof} Let $u \in H: g(u) \in K$ be a minimum of general convex function
$F$ on $K.$ Then
\begin{eqnarray}\label{2.2}
F(g(u)) \leq F(g(v)), \quad  \forall  g(v) \in K.
\end{eqnarray}
Since $K$ is a general convex set, so, for all $u,v \in K, t \in [0,1],
g(v_t)=g(u)+t(g(v)-g(u)) \in K.$ Setting $g(v)= g(v_t)$ in (\ref{2.2}), we have
\begin{eqnarray*}
F(g(u)) \leq F(g(u)+ t(g(v)-g(u)) \leq F(g(u))+t(F(g(v)-g(u)).
\end{eqnarray*}
Dividing the above inequality by $t$ and taking $t \longrightarrow 0,$ we
have
\begin{eqnarray*}
\la F'(g(u)), g(v)-g(u) \ra  \geq 0,
\end{eqnarray*}
which is the required result (\ref{2.1}).\\
Conversely, let $u \in H, g(u) \in K$ satisfy the inequality (\ref{2.1}). Since
$F$  is a general convex function, so $ \forall g(u),g(v) \in K, \quad t \in [0,1], \quad
g(u)+t(g(v)-g(u)) \in K $ and
\begin{eqnarray*}
F(g(u)+t(g(v)-g(u))) \leq (1-t)F(g(u))+tF(g(v)),
\end{eqnarray*}
which implies that
\noindent
\begin{eqnarray*}
F(g(v))-F(g(u)) \geq \frac{F(g(u)+t(g(v)-g(u)))-F(g(u))}{t}.
\end{eqnarray*}
Letting $t \longrightarrow 0,$ we have
\begin{eqnarray*}
F(g(v))-F(g(u)) \geq \la F'(g(u)),g(v)-g(u)\ra  \geq 0, \quad \mbox{using (\ref{2.1}),}
\end{eqnarray*}
which implies that
\begin{eqnarray*}
F(g(u)) \leq F(g(v)), \quad \forall g(v) \in K,
\end{eqnarray*}
showing that $u \in H:  g(u) K$ is the minimum of $F$ on $K$ in $H.$ \hfill \qquad
$\Box $
\end{proof}
Theorem{\ref{theorem1} implies that general convex programming problem can be studied via
the general variational inequality (\ref{2.5})  with $\quad Tu = F'(g(u)).$ In a
similar way, one can show that the general variational inequality is the
Fritz-John condition of the inequality constrained  optimization problem.\\
In many applications, the general variational inequalities do not arise as the minimization of the
differentiable general convex functions. Also, it is known that the variational inequality introduced
by Stampacchia \cite{172} can only be used to study the even-order boundary value problems.
 These  facts  motivated Noor \cite{87} to introduce a more general  variational inequality involving
  two distinct operators.  General variational inequalities is a unified framework to study such type problems.\\

Let $K$ be a closed convex set in $H$ and $T,g: H \lrt H$ be
nonlinear operators. We now consider the problem of finding $u \in H,g(u) \in K$ such that
\begin{eqnarray}\label{2.5}
\la Tu, g(v)-g(u)\ra  \geq 0, \q \forall v\in H;  g(v) \in K.
\end{eqnarray}
Problem (\ref{2.5}) is called the {\it general variational inequality, } which
was introduced and studied by Noor \cite{87} in 1988.  It has been
shown that a large class of unrelated odd-order and nonsymmetric obstacle,
unilateral, contact, free, moving, and equilibrium problems arising in regional, physical, mathematical,
engineering and applied sciences can be studied in the unified
and general framework of the general variational inequalities (\ref{2.5}). Luc and Noor \cite{69} have studied the
local uniqueness of solution of the general variational inequality (\ref{2.5})
by using the concept of Frechet approximate Jacobian.\\

We now discus some special cases of the general variational inequality (\ref{2.5}).\\ \\
\noindent{\bf(I).} For $g \equiv I,$ where $I$ is the
identity operator, problem  (\ref{2.5}) is equivalent to finding $u \in K$ such that
\begin{eqnarray}\label{2.6}
\la Tu, v-u \ra \geq 0 \quad \forall  v \in K,
\end{eqnarray}
which is known as the classical variational inequality introduced and
studied by Stampacchia \cite{172} in 1964. For recent state-of-the-art in this
field, see [1-200] and the references therein.\\
From now onward, we assume that $g$ is onto $K$ unless, otherwise
specified.\\
\noindent{\bf(II).} If $ N(u)= \{w \in H: \la w,v-u \ra \leq 0,\quad \forall v \in K$ \} is a
normal cone to the convex set $K$ at $u$, then the general variational
inequality (\ref{2.5}) is equivalent to finding $ u \in H, g(u) \in K$ such that
\begin{eqnarray*}
-T(u) \in N(g(u)),
\end{eqnarray*}
which are known as the generalized nonlinear equations, see  \cite{129,130}.\\
\noindent{\bf(III).}  If $P^{tg}$ is the projection of $-Tu$ at $g(u) \in K,$ then it has been
shown that the general variational inequality problem (\ref{2.5}) is equivalent
to finding $ u \in H, g(u) \in K$ such that
\begin{eqnarray*}
P^{tg}[-Tu] := P^{tg}(u) = 0,
\end{eqnarray*}
which are known as the  tangent projection equations.
This equivalence has been used to discuss the local convergence
analysis of a wide class of iterative methods for solving general
variational inequalities (\ref{2.5}).\\
\noindent{\bf(IV).} If $K^*= \{u \in H: \la u, v \ra \geq 0, \forall v \in K\}$ is a polar (dual) cone
of a convex cone $K$ in $H$, then problem (\ref{2.5}) is equivalent to
finding $u \in H$ such that
 \begin{eqnarray}\label{2.7}
 g(u) \in K, \q Tu \in K^* \q \mbox{and} \q \la Tu, g(u)\ra = 0,
\end{eqnarray}
which is known as the general complementarity problem, see Noor \cite{87}.
For $g(u)= m(u)+K,$ where $ m$ is a point-to-point mapping,
 is called the  implicit (quasi) complementarity problem. If $g
\equiv I,$ then problem (\ref{2.7}) is known as the generalized
complementarity problems. Such problems have been studied extensively
in recent years.\\
\noindent{\bf(V).}  If $ K=H, $ then the general variational inequality (\ref{2.5}) is equivalent to finding $u\in H: g(u) \in H $ such that
\begin{eqnarray*}
\langle Tu, g(v) \rangle = 0, \quad \forall v\in H: g(v) \in H,
\end{eqnarray*}
which is called the weak formulation of the odd-order and nonsymmetric boundary value problems.\\
For suitable and appropriate choice of the operators and  spaces, one can
obtain several classes of variational inequalities and related optimization
problems as special cases of the general variational inequalities (\ref{2.5}). \\

We also need the following result, which plays a key role in the studies of variational inequalities and optimization theory.\\
\begin{lemma}\cite{59}\label{lemma1} For  a given $z \in H$, $u \in K$ satisfies
the inequality
 \begin{eqnarray}\label{2.8}
\la  u - z, v -u \ra  \geq 0, \q \forall  v \in K,
\end{eqnarray}
if and only if
$$u = P_K z, $$
where $P_K$ is the projection of $H$ onto $K$.

\end{lemma}
Also, the projection
operator $P_K$ is nonexpansive, that is,
\begin{eqnarray*}
\|P_K(u)-P_K(v)\| \leq \|u-v\|, \quad  \forall u,v \in H.
\end{eqnarray*}
 and satisfies the inequality
\begin{eqnarray*}
\|P_Kz-u\|^2 \leq \|z-u\|^2 -\|z-P_K z\|^2, \quad \forall z,u \in H.
\end{eqnarray*}

\subsection{Applications }
\vskip .4pc
We now discuss some applications of general variational inequalities (\ref{2.5}).
For this purpose, we consider the functional $I[v], $ defined as
\begin{eqnarray}\label{2.9}
I[g(v)]:= \la Tv,g(v) \ra -2 \la f, g(v) \ra, \quad \forall v \in H,
\end{eqnarray}
which is called the general energy or  potential, virtual work  functional.
We remark that, if $g \equiv I, $ the identity operator, then the
functional $I[v] $ reduces to
\begin{eqnarray*}
J[v] = \la Tv,v \ra - 2 \la f, v \ra , \quad \forall v \in H,
\end{eqnarray*}
which is known as the standard energy function.  \\
It is known that, if the operator $T : H \longrightarrow H $ is linear, symmetric and positive,
then the minimum of the functional $J[v] $  on the closed and  convex set
$K$ in $H$ is equivalent to finding $u \in K $ such that
\begin{eqnarray}\label{2.10}
\la Tu,v-u \ra \geq \la f, v-u \ra, \quad \forall  v \in K.
\end{eqnarray}
Inequality of type (\ref{2.10}) is known as variational inequality, which was
introduced by Stampacchia \cite{172} in the study of potential theory.  See also Fichera \cite{38}. It is
clear that the symmetry and positivity of the operator $ T $ is must. On
the other hand, there are many important problems, which are nonsymmetric
and non-positive. For the nonsymmetric and odd-order problems, many
methods have developed by several authors including Filippov \cite{39} and
Tonti \cite{176} to construct the energy functional of type (\ref{2.9}} by
introducing the concept of $g$-symmetry and $g$-positivity of the operator
$g. $ We now recall the following concepts.\\

\begin{definition}\cite{39,176}  $ \forall u,v \in H, $ the operator $T : H
\longrightarrow H $ is said to be : \\
\noindent{\bf (a). } {\it $g$-symmetric , } if and only if,
\begin{eqnarray*}
\la Tu,g(v) \ra = \la g(u),Tv \ra.
\end{eqnarray*}
\noindent{\bf (b). }{$g$-positive, } if and only if,
\begin{eqnarray*}
\la Tu, g(u) \ra \geq 0.
\end{eqnarray*}
\noindent{\bf (c). }  {\it $g$-coercive ($g$-elliptic, } if there exists a
constant $\alpha > 0 $ such that
\begin{eqnarray*}
\la Tu, g(u) \ra \geq \alpha \|g(u)\|^2.
\end{eqnarray*}
\end{definition}
Note that $g$-coercivity implies $g$-positivity, but the converse is not
true. It is also worth mentioning that there are operators which are not
$g$-symmetric  but $g$-positive. On the other hand, there are $g$-positive
, but not $g$-symmetric operators. Furthermore, it is well-known \cite{39,176}
that if, for a linear operator $T, $ there exists an inverse operator
$T^{-1} $ operator on $R(T) $ with $\overline{R(T)} = H, $ then one can
find an infinite set of auxiliary operators $g$ such that the operator $T
$ is both $g$-symmetric and $g$-positive.\\

We now consider the problem of finding the minimum of the functional
$I[v], $ defined by (\ref{2.9}), on the convex set $K $ in $H $ and this is the
main motivation of our next result.

\begin{theorem}\label{theorem2} Let the operator $T : H \longrightarrow H $
be  linear, $g$-symmetric and $g$-positive. If the operator $g: H
\longrightarrow H $ is either linear or convex, then the function $ u \in
H $ minimizes the functional $ I[v] $ defined by (\ref{2.9}) on the convex set $K $ in $H $ if
and only if $ u \in H, \quad g(u) \in K $ such that
\begin{eqnarray}\label{2.11}
\la Tu, g(v)-g(u) \ra \geq \la f, g(v)-g(u) \ra \quad \forall v\in H: g(v) \in K.
\end{eqnarray}
\end{theorem}
\begin{proof} Let $u \in H, \quad  g(u) \in K  $ satisfy (\ref{2.11}).
Then,  using the $g$-positivity of the operator $T, $ we have
\begin{eqnarray}\label{2.12}
\la Tv, g(v)-g(u) \ra&\geq  & \la f,g(v)-g(u) \ra +\la Tv-Tu,g(v)-g(u) \ra \nonumber \\
& \geq & \la f, g(v)-g(u) \ra, \quad \forall g(v) \in K.
\end{eqnarray}
Since $K$ is a convex set, so for all $t \in[0,1], \quad u, w \in K, \quad
v_t = u +t(w-u) \in K. $ Taking $v = v_t $ in (\ref{2.12}) and using the fact
that $g$ is linear (or convex), we have
\begin{eqnarray}\label{2.13}
\la Tv_t, g(w)-g(u) \ra \geq \la f, g(w)-g(u) \ra .
\end{eqnarray}
We now  define the function
\begin{eqnarray*}
h(t) & = &t \la Tu,g(w)-g(u) \ra +\frac{t^2}{2}\la T(w-u),g(w)-g(u) \ra
\nonumber \\
&& -t\la f, g(w)-g(u) \ra ,
\end{eqnarray*}
such that
\begin{eqnarray*}
h^{\prime}(t) &= &\la Tu,g(w)-g(u) \ra +t \la T(w-u),g(w)-g(u) \ra - \la
f,  g(w)-g(u) \ra \\
& \geq & 0, \quad \mbox{by (\ref{2.13}) . }
\end{eqnarray*}
Thus it follows that $h(t) $ is an increasing function on $[0,1] $ and so
$h(0) \leq h(1) $ gives us
\begin{eqnarray*}
\la Tu,g(u) \ra -2 \la f,g(u) \leq \la Tw,g(w) \ra -2 \la f, g(w) \ra,
\end{eqnarray*}
that is,
\begin{eqnarray*}
I[u] \leq I[w],
\end{eqnarray*}
which shows that $u \in H $ minimizes the functional $I[v], $defined by (\ref{2.9}), on the convex set $K$ in $H. $\\
Conversely, assume that $u \in H $ is the minimum of $I[v] $ on the
convex set $K, $ then
\begin{eqnarray}\label{2.14}
I[u] \leq I[v], \quad \forall v \in H :  g(u) \in K.
\end{eqnarray}
Taking $v = v_t \equiv u +t(w-u) \in K , \forall u, w \in K
$ and $t\in [0,1] $ in (\ref{2.14}), we have
\begin{eqnarray*}
I[u] \leq I[v_t].
\end{eqnarray*}
Using (\ref{2.9}) and the linearity (or convexity ) of $g , $ we obtain
\begin{eqnarray*}
\la Tu,g(w)-g(u) \ra + \frac{t}{2}\la T(w-u),g(w)-g(u) \ra \geq \la
f,g(w)-g(u) \ra ,
\end{eqnarray*}
from which, as $t \longrightarrow 0, $ we have
\begin{eqnarray*}
\la Tu, g(w)-g(u) \ra \geq \la f, g(w)-g(u) \ra, \quad \forall w\in H: g(w) \in K.
\end{eqnarray*}
This completes the proof. \hfill \qquad $\Box $
\end{proof}
\vskip .4pc
 We remark that for $g = I, $ the identity operator, Theorem \ref{theorem2} reduces
to the following well-known result in variational inequalities, which is
due to Stampacchia \cite{172}.\\
\begin{theorem}Let the operator $T$ be linear, symmetric and
positive. Then the minimum of the functional $J[v]$ defined by (\ref{2.9}) on
the convex set $K $ in $H $ can be characterized by the variational
inequality
\begin{eqnarray*}
\la Tu, v-u \ra \geq \la f, v-u \ra, \quad \forall v \in K.
\end{eqnarray*}
\end{theorem}
\begin{proof}Its proofs follows from Theorem \ref{theorem2}.
\end{proof}

\begin{example}
We now show that  a wide class of nonsymmetric and odd-order obstacle,
unilateral, free, moving and general equilibrium problems arising in pure
and applied sciences  can be formulated in terms of (\ref{2.9}). For
simplicity and to illustrate the applications, we consider the obstacle
boundary value of third order of the type: {\it Find $ u $ such that }

\begin{eqnarray}\label{2.15}
\left. \begin{array}{llll}
-u^{\prime \prime \prime } \geq  f, \hspace{6mm}   &\mbox{on}, \quad \Omega
= & [a,b]  \\ \\
\hspace{4mm} u \geq \psi, \hspace{6mm}    &  \mbox{on}\quad \Omega  = &
[a,b] \\ \\
\hspace{4mm}[-u^{\prime \prime \prime }-f][u - \psi ]= 0, \quad \quad
& \mbox{on}\quad \Omega  = & [a,b] \\ \\
\hspace{4mm} u(a) = 0, \quad u^{\prime }(a) = 0, \quad u^{\prime }(b) =0.
\end{array} \right \},
\end{eqnarray}
where $\Omega = [a,b] $ is a domain, $\psi (x) $ and $f(x) $ are the
given
functions. The function $\psi $ is known as the obstacle function. The
region, where $u(x) = \psi (x), \quad \mbox{for } x \in \Omega $ is called
the contact region (set).
\end{example}
We not that problem (\ref{2.15}) is a generalization of the third-order boundary
value problem
\begin{eqnarray*}
-\frac{d^3u(x)}{dx^3}  =  f(x) \quad  \quad x \in \Omega
\end{eqnarray*}
with boundary condition
\begin{eqnarray*}
u(0) \quad = \quad u^{\prime }(a) \quad = \quad u^{\prime }(b) = 0,
\end{eqnarray*}
which arises from a similarity solution of the so-called barotropic
quasi-geostrophic potential vorticity equation for one layer ocean
circulation.  For the formulation of the equation, see \cite{91} and the
references therein.\\

To study the problem (\ref{2.15}) in the general framework of the general
variational inequality, we define
\begin{eqnarray*}
K = \{ u \in H^{2}_{0}(\Omega ): u(x) \geq \psi (x)  \mbox{on } \Omega \},
\end{eqnarray*}
which is closed convex set in $H^{2}_{0}(\Omega ). $ For the definition
and properties of the spaces $H^{m}_{0} (\Omega ) , $ see \cite{50}.\\

Using the technique of \cite{39, 176}, we can easily show that
the energy functions associated with the problem (\ref{2.15}) is
\begin{eqnarray}\label{2.16}
I[v] & = & \int _{a}^{b}
\left(-\frac{d^3v}{dx^3}\right)\left(\frac{dv}{dx}\right)dx-2\int
_{a}^{b}f(x)\left(\frac{dv}{dx}\right)dx, \quad \mbox{for all }
\frac{dv}{dx} \in K \nonumber \\
& = &
\int_{a}^{b}\left(\frac{d^2v}{dx^2}\right)\left(\frac{d^2v}{dx^2}\right)dx-2\int
_{a}^{b}f(x)\left(\frac{dv}{dx}\right)dx \nonumber \\
& = & \la Tv,g(v) \ra -2\la f, g(v) \ra,
\end{eqnarray}
where
\begin{eqnarray}\label{2.17}
\la  Tu, g(v)  \ra  = \int_{a}^{b}\left(\frac{d^2u}{dx^2}\right)\left(\frac{d^2v}{dx^2}\right)dx,
\end{eqnarray}
and
\begin{eqnarray*}
\la f, g(v) \ra = \int _{a}^{b}f(x)\left(\frac{dv}{dx}\right)dx.
\end{eqnarray*}
Here $$g(u) = \frac{du}{dx } \quad \mbox{and}\quad Tu= \frac{-d^3u}{d^3x}$$ are the linear operators.\\
It is clear that the operator $T $ defined by the relation (\ref{2.17}) is
$g$-symmetric, $g$-positive and linear. Also we note that the operator $ g
= \frac{d}{dx} $ is a linear operator. Consequently all the assumptions of
Theorem \ref{theorem2} are satisfied. Thus it follows from Theorem \ref{theorem2} that the
minimum of the functional $I[v], $ defined by (\ref{2.16}) is equivalent to
finding  $u \in H $ such that $g(u) \in K $ and the inequality (\ref{2.5})  holds.\\
In fact, we conclude the problems equivalent to (\ref{2.15}) are:\\
\noindent{\it  The Variational Problem. }
Find $ u \in H^{2}_{0}(\Omega ), $ which gives the minimum value to the
functional
\begin{eqnarray*}
I[v] = \la Tv,g(v) \ra - 2 \la f,g(v) \ra \quad \mbox{on the convex set }
\quad K.
\end{eqnarray*}
\noindent{\it The Variational Inequality (Weak) Problem.  }
Find $u \in H^{2}_{0}(\Omega )$ such that $g(u) \in K $ and
\begin{eqnarray*}
\la Tu, g(v)-g(u) \ra \geq \la f, g(v)-g(u) \ra , \quad\forall\quad g(v) \in K.
\end{eqnarray*}
\subsection{Quasi variational inequalities}
We now show that  quasi variational inequalities are  a special
case of general variational inequalities (\ref{2.5}). If the
convex set $K$ depends upon the solution explicitly  or implicitly, then
variational inequality problem is known as the quasi variational
inequality. For a given operator $T: H \lrt H,$ and a point-to-set
mapping $K : u \lrt K(u), $ which associates a closed convex-valued set
$K(u)$ with any element $u $ of $H,$ we consider the problem of finding $u
\in K(u)$ such that
\begin{eqnarray}\label{2.18a}
\la Tu, v- u \ra \q \geq 0,   \q \forall v \in K(u)
\end{eqnarray}
Inequality of type (\ref{2.18a}) is called the  quasi variational inequality.
For the formulation, applications, numerical methods and sensitivity
analysis of the quasi variational inequalities, see \cite{15,22,84,85,102,139,146}
and the references therein.\\
We can rewrite the equation (\ref{2.18a}), for $\rho > 0, $ as
\begin{eqnarray*}\label{2.18}
0 & \leq &\la \rho Tu +u-u, v- u \ra  \\
&=& \la u-(u-\rho Tu),v-u \ra,  \quad \forall v \in K(u).
\end{eqnarray*}
which is equivalent ( using Lemma \ref{lemma1})  to finding $ u \in K(u) $ such that
\begin{eqnarray}\label{2.19}
u = P_{K(u)}[u-\rho Tu].
\end{eqnarray}
In many important applications, the convex-valued  set $K(u)$ is of the
form
\begin{eqnarray}\label{2.20}
K(u)= m(u) +K,
\end{eqnarray}
where $m $ is a point-to-point mapping and $K$ is a closed convex set.\\ \\
From (\ref{2.19}) and (\ref{2.20}), we see that problem (\ref{2.18a}) is equivalent to
\begin{eqnarray*}
u & = &P_{K(u)}[u-\rho Tu] =P_{m(u)+K}[u-\rho Tu]  \\
& = & m(u)+P_K[u-m(u)-\rho Tu]
\end{eqnarray*}
which implies that
\begin{eqnarray*}
g(u)=P_K[g(u)-\rho Tu ] \q  \mbox{with} \q g(u)= u-m(u),
\end{eqnarray*}
which is equivalent to the general variational inequality (\ref{2.5}) by
an application of Lemma \ref{lemma1}. We have shown that the quasi variational
inequalities (\ref{2.18a}) with the convex-valued set $K(u)$ defined by (\ref{2.20})
are equivalent to the general variational inequalities (\ref{2.5}).\\
We now recall the well known concepts.

\begin{definition}\label{def7} For all $u,v \in H$, the operator $T: H
\longrightarrow H$ is said to be
\noindent{\bf (i).} {\it $g$-monotone}, if
$$
\la Tu-Tv,g(u)-g(v)\ra  \geq 0.
$$
{\bf (ii).}  {\it $g$-pseudomonotone}, if
$$
\la Tu,g(v)-g(u)\ra  \geq 0 \quad  \mbox{implies} \quad \la Tv,
g(v)-g(u)\ra \geq 0.
$$

\end{definition}
For $g \equiv I,$ Definition \ref{def7} reduces to the usual definition of
monotonicity and pseudomonotonicity of the operator $T.$ \ Note that
monotonicity implies pseudomonotonicity  but the converse is not true,
see [35].\\
\begin{definition} \label{def8} A function $ F $ is said to be strongly
general convex on the general convex set $K $ with modulus $\mu > 0, $ if, for all
$g(u),g(v)  \in K_g, t \in [0,1], $
\begin{eqnarray*}
F(g(u)+t(g(v)-g(u)) \leq (1-t)F(g(u)) +tF(g(v))-t(1-t)\mu \|g(v)-g(u)\|^2.
\end{eqnarray*}
\end{definition}
For  differentiable strongly  general convex function $F,$  the following statements are equivalent.
\begin{eqnarray*}
&& 1.\quad F(g(v))-F(g(u)) \geq  \la F^{\prime }(g(u)), g(v)-g(u) \ra + \mu
\|g(v)-g(u)\|^2\\
&& 2. \quad F^{\prime }(g(u))-F^{\prime }(g(v)), g(u)-g(v) \ra \geq  \mu
\|g(v)-g(u)\|^2.
\end{eqnarray*}
 It is well known that the general convex functions are not
convex function, but they have some nice properties which the convex
functions have. Note that, for $g = I, $ the general convex functions are
convex functions and definition (\ref{def8}) is the well known result in convex
analysis.

\section{Projection Methods}
In this section, we use the fixed point formulation to suggest and analyze some new
implicit methods for solving the variational inequalities.
Using Lemma \ref{lemma1}, one can show that the general variational
inequalities are equivalent to the fixed point problems.
\begin{lemma}\label{lemma2}\emph{\cite{87}}
The function $u\in H : g(u) \in K$ is a solution of the general variational inequalities
(\ref{2.5}), if and only if, $u\in H: g(u) \in K$ satisfies the relation
\begin{eqnarray}\label{2.21}
g(u)= P_K[g(u)-\rho Tu],
\end{eqnarray}
where $P_K $ is the projection  operator and $\rho>0$ is a constant.
\end{lemma}
Lemma \ref{lemma2} implies that the general  variational inequality
(\ref{2.5}) is equivalent to the fixed point problem (\ref{2.21}).
This equivalent fixed point formulation  was used to suggest some implicit iterative
methods for solving the general variational inequalities. One uses the equivalent fixed point formulation(\ref{2.21}) to suggest the following iterative
methods for solving general variational inequalities (\ref{2.5}).\\

\begin{algorithm}\label{alg1} For a given $u_{0}\in H$, compute ${u_{n+1}}$ by the iterative scheme
\begin{eqnarray*}
u_{n+1} = u_n- g(u_n)+  P_K[g(u_n)-\rho Tu_n], \quad n=0,1,2,...
\end{eqnarray*}
which is known as the projection method and has been studied extensively.
\end{algorithm}

\begin{algorithm}\label{alg2} For a given $u_{0}\in H$, compute ${u_{n+1}}$ by the iterative
scheme
\begin{eqnarray*}
u_{n+1}  = u_n -g(u_n)+ P_K[g(u_n)-\rho Tg^{-1}P_K[g(u_n)-\rho Tu_{n}]], \quad n=0,1,2,...
\end{eqnarray*}
\end{algorithm}
which can be viewed as the extragradient method, which was suggested and analyzed by Koperlevich\cite{60} for solving the classical variational inequalities.
Noor \cite{114}has proved that the convergence of the extragradient method for pseudomonotone operators.\\

\begin{algorithm}\label{alg3} For a given $u_{0}\in H$, compute ${u_{n+1}}$ by the iterative
scheme
\begin{eqnarray*}
u_{n+1}= u_n -g(u_n)+ P_K[g(u_{n+1})-\rho Tu_{n+1}], \quad n=0,1,2,...
\end{eqnarray*}
which is known as the modified projection method and has been studied extensively, see Noor \cite{103}.
\end{algorithm}
We can rewrite the equation (\ref{2.21})  as:
\begin{eqnarray*}
g(u)= P_K[g(\frac{u+u}{2}) -\rho Tu].
\end{eqnarray*}
This fixed point formulation was used to suggest the following implicit method for solving variational inequalities, which is due to Noor et al\cite{110,116}. We used this equivalent formulation to suggest implicit methods for general variational inequality (\ref{2.5}).\\

\begin{algorithm}\label{alg4} For a given $u_{0}\in H$, compute ${u_{n+1}}$ by the iterative
scheme
\begin{eqnarray*}
 u_{n+1} = u_n- g(u_n)+ P_K[g\big(\frac{u_n+u_{n+1}}{2}\big)-\rho Tu_{n+1}], \quad n=0,1,2,...
\end{eqnarray*}
\end{algorithm}
For the implementation  of this Algorithm, one can  use the predictor-corrector technique
 to suggest the following two-step iterative method for solving general variational inequalities.\\

\begin{algorithm}\label{alg5} For a given $u_{0}\in H$, compute ${u_{n+1}}$ by the iterative
scheme
\begin{eqnarray*}
g(y_n) &= & P_K[g(u_n)-\rho Tu_n] \\
 u_{n+1} &=& u_n-g(u_n)+P_K[g\big (\frac{y_n+u_n}{2}\big)-\rho Ty_{n}],\quad \lambda \in [0,1], \quad n=0,1,2,...
\end{eqnarray*}
which is a two-step iterative  method:
\end{algorithm}
From  the equation (\ref{2.21}), we have
\begin{eqnarray*}
g(u) = P_K[g( u)  -\rho T(\frac{u+u}{2})].
\end{eqnarray*}
This fixed point formulation is used to suggest the implicit method for solving the variational inequalities as\\

\begin{algorithm}\label{alg6} For a given $u_{0}\in H$, compute ${u_{n+1}}$ by the iterative
scheme
\begin{eqnarray*}
 u_{n+1} = u_n-g(u_n)+ P_K[g(u_n)-\rho T(\frac{u_n+u_{n+1}}{2})], \quad n=0,1,2,....
\end{eqnarray*}
which is another implicit method, see Noor et al. \cite{148}.
\end{algorithm}
To implement this implicit method, one can use the predictor-corrector technique to rewrite Algorithm 6 as equivalent two-step iterative method.\\

\begin{algorithm}\label{alg7} For a given $u_{0}\in H$, compute ${u_{n+1}}$ by the iterative
scheme
\begin{eqnarray*}
g(y_n)&= & P_K [g(u_n)-\rho Tu_n], \\
 u_{n+1} &= &u_n-g(u_n)+  P_K[g(u_n)-\rho T(\frac{u_n+y_n}{2})], \quad n=0,1,2,....
\end{eqnarray*}
which  is known as the mid-point implicit method for solving general variational inequalities.
\end{algorithm}
For the convergence analysis and other aspects of Algorithm \ref{alg4}, see Noor et al \cite{148}.\\
It is obvious that Algorithm \ref{alg4} and Algorithm 6 have been suggested using different variant of the fixed point  formulations  (\ref{2.21}). It is natural to combine these fixed point formulation to suggest a hybrid implicit method for solving the general variational inequalities and related optimization problems, which is the main motivation of this paper.\\
One can rewrite the equation (\ref{2.21})  as
\begin{eqnarray*}
g(u)= P_K[g \big(\frac{u+u}{2}\big) -\rho T(\frac{u+u}{2})].
\end{eqnarray*}
This equivalent fixed point formulation enables to suggest the following  method for solving the general variational inequalities.\\
\begin{algorithm} \label{alg8} For a given $u_{0}\in H$, compute ${u_{n+1}}$ by the iterative
scheme
\begin{eqnarray*}
 u_{n+1} =u_n-g(u_n)+ P_K[g\big(\frac{u_n+u_{n+1}}{2}\big)-\rho T(\frac{u_n+ u_{n+1}}{2})], \quad n=0,1,2,....
\end{eqnarray*}
which is an implicit method.
\end{algorithm}

We would like to emphasize that Algorithm \ref{alg8} is an implicit method. To implement the implicit method, one uses the predictor-corrector technique.
 We use Algorithm \ref{alg1} as the predictor and Algorithm 8 as corrector. Thus, we obtain a new two-step method for solving general variational inequalities.\\

\begin{algorithm}\label{alg9} For a given $u_{0}\in H$, compute ${u_{n+1}}$ by the iterative scheme
\begin{eqnarray*}
g(y_n) &= & P_K[g(u_n)-\rho Tu_n] \\
 u_{n+1} & =& u_n-g(u_n)+P_K[g\big(\frac{y_{n}+u_n}{2}\big)-\rho T\big(\frac{y_{n}+u_n}{2}\big)], \quad n=0,1,2,...
 \end{eqnarray*}
which is two step method.
\end{algorithm}
     For constants $\lambda , \xi \in [ 0.1], $ we can rewrite the equation (\ref{2.21})  as:
\begin{eqnarray*}
g(u)= P_K\big[(1-\lambda)g(u)+ \lambda g(u)) -\rho T((1-\xi )u+\xi u)\big].
\end{eqnarray*}
This equivalent fixed point formulation enables to suggest the following  method for solving the general variational inequalities.\\
\begin{algorithm}\label{alg10} For a given $u_{0}\in H$, compute ${u_{n+1}}$ by the iterative
scheme
\begin{eqnarray*}
g(u_{n+1})= P_K\big[(1-\lambda)g(u_n)+ \lambda g(u_{n+1})) -\rho T((1-\xi )u_n+\xi u_{n+1})\big]. \quad n=0,1,2,...
\end{eqnarray*}
which is an implicit method.
\end{algorithm}
Using the prediction-correction technique, Algorithm 10 can be written in the  following form.\\

\begin{algorithm}\label{alg11} For a given $u_{0}\in H$, compute ${u_{n+1}}$ by the iterative
scheme.
\begin{eqnarray*}
g(y_n) &= & P_K[g(u_n)-\rho Tu_n] \\
 g(u_{n+1})&=& P_K\big[(1-\lambda)g(u_n)+ \lambda g(y_{n})) -\rho T((1-\xi )u_n+\xi y_{n})\big], \quad n=0,1,2,..
 \end{eqnarray*}
 \end{algorithm}
which is two step method.
\begin{remark}
It is worth mentioning that Algorithm \ref{alg11} is a unified ones. For suitable and appropriate choice of the constant $\lambda $ and $\xi, $ one can obtain a wide class of iterative methods for solving general variational inequalities and related optimization problems.
\end{remark}

\section{Wiener-Hopf Equations Technique}

In this Section, we consider the problem of the general Wiener-Hopf equations. To be more precise, let
$Q_K= I-P_K,$ where $I$
is the identity operator and $P_K$ is the projection of $H$ onto $K.$ For
given nonlinear operators $T,g: H \rightarrow H,$  consider the problem of
finding $z \in H$ such that
\begin{eqnarray}\label{2.22}
\rho Tg^{-1}P_Kz + Q_Kz = 0,
\end{eqnarray}
provided $g^{-1} $ exists. Equations of the type (\ref{2.22}) are called the {\it
general Wiener-Hopf equations,} which were introduced and studied by
Noor\cite{90,91}. For $g = I,$ we obtain the original Wiener-Hopf equations,
which were introduced and studied by Shi \cite{167} and Robinson \cite{165} in
different settings independently.  Using the
projection operators technique,  one can show that the general variational
inequalities are equivalent to the general  Wiener-Hopf equations. This equivalent
alternative formulation has played a fundamental and important role in
studying various aspects of variational inequalities. It has been
shown that Wiener-Hopf equations are more flexible and provide a unified
framework to develop some efficient and powerful numerical techniques for
solving variational inequalities and related optimization problems.\\
\begin{lemma}\label{lemma3}
The element $u\in H: g(u) \in K$ is a solution of the general variational inequality
(\ref{2.5}), if and only if $z\in H$ satisfies the Wiener-Hopf equation (\ref{2.22}), where
\begin{eqnarray}
g(u)&=&P_{K}z,   \label{eq4.2a}\\
 z&=&g(u)-\rho Tu,\label{eq4.3a}
\end{eqnarray}
where $\rho>0$ is a constant.
\end{lemma}
From Lemma \ref{lemma3}, it follows that the variational
inequalities (\ref{2.5}) and the Wiener–Hopf equations (\ref{2.22})
are equivalent. This alternative equivalent formulation is
used to suggest and analyze a wide class of efficient and robust
iterative methods for solving general variational inequalities and related
optimization problems, see \cite{90,91,99,109} and the references
therein.\\
We use the general Wiener-Hopf equations (\ref{2.22}) to suggest some new
iterative methods for solving the general variational inequalities.
 From (\ref{eq4.2a}) and (\ref{eq4.3a}),
\begin{eqnarray*}
z& = &P_{K}z-\rho TP_{K}z\nonumber \\
&=& P_{K}[g(u)-\rho Tu]-\rho Tg^{-1}P_{K}[g(u)-\rho Tu].
\end{eqnarray*}
Thus, we have
\begin{eqnarray*}
g(u)= \rho Tu+\big[P_{K}[g(u)-\rho Tu]-\rho Tg^{-1}P_{K}[g(u)-\rho Tu +P_{K}[g(u)-\rho Tu]-g(u)].
\end{eqnarray*}
Consequently, for a constant $\alpha >0, $ we have
\begin{eqnarray}\label{eq3.4}
g(u)&=& (1-\alpha)g(u)
+ \alpha \{P_K[P_{K}[g(u)-\rho Tu] +\rho Tu-\rho Tg^{-1}P_{K}[g(u)-\rho Tu]\nonumber \\
&&+P_{K}[g(u)-\rho Tu]-g(u)]\}\nonumber \\
&=& (1-\alpha)g(u)+ \alpha\{ P_K[g(y)-\rho Ty]+P_{K}[g(y)-\rho Tu]-g(u)]\},
\end{eqnarray}
where
\begin{eqnarray}\label{eq3.5}
g(y)= P_K[g(u)-\rho Tu].
\end{eqnarray}
Using (\ref{eq3.4}) and (\ref{eq3.5}), we can suggest the following
new predictor-corrector method for solving variational inequalities.\\
\begin{algorithm}\label{alg12}
For a given $u_{0}\in H$, compute ${u_{n+1}}$ by the iterative
scheme
\begin{eqnarray*}
g(y_n) &= & P_K[g(u_n)-\rho Tu_n] \nonumber \\
g( u_{n+1}) &=&  (1-\alpha_n)g(u_n) + \alpha_n\bigg\{P_K[g(y_n)-\rho Ty_n+g(y_n) -(g(u_n)- \rho Tu_n)]\bigg\}.
\end{eqnarray*}
\end{algorithm}
Algorithm 12 can be rewritten in the following equivalent
form:
\begin{algorithm}\label{alg13}
For a given $u_{0}\in H$, compute ${u_{n+1}}$ by the iterative scheme
\begin{eqnarray*}
 u_{n+1} &=& (1-\alpha_n)u_n \nonumber \\
 &&+ \alpha _n\{P_K[P_K[g(u_n)-\rho Tu_n]-\rho Tg^{-1}P_K[g(u_n)-\rho Tu_n]\nonumber \\
 &&+P_{K}[g(u_n)-\rho Tu_n)-(g(u_n)-\rho Tu_n])\},
\end{eqnarray*}
\end{algorithm}
which is an explicit iterative method and appears to be  a new one.\\
If $\alpha_n =1, $ then Algorithm 13 reduces to\\
\begin{algorithm}\label{alg14} For a given $u_{0}\in H$, compute ${u_{n+1}}$ by the iterative
scheme
\begin{eqnarray*}
g(y_n)&= & P_K[g(u_n)-\rho Tu_n] \\
 u_{n+1} &=& P_K[g(y_n)-\rho Ty_n+ g(y_n)-(g(u_n)-\rho Tu_n])\}, \quad n=0, 1,2, ...,
\end{eqnarray*}
\end{algorithm}
which appears to be a new one.\\

We now consider  another  Algorithm  for solving
general variational inequalities (\ref{2.5}.  We also include some computational experiments
of these special cases. See \cite{122,151,152} for further details.

\begin{algorithm}\label{alg15} For a given $u_0 \in K, $ compute
\begin{eqnarray*}
g(z_n ):=P_K[g(u_n) - Tu_n].
\end{eqnarray*}
If $\|R(u_n))\|=0$,  where $R(u_n)= g(u_n)-P_K[g(u_n) - Tu_n].$
stop; otherwise compute
\begin{eqnarray*}
 g(y^n):=(1-\eta_n)g(u^n)+\eta_n g(z^n),
\end{eqnarray*}
where $\eta_n=\gamma^{m_n}$ with $m_n $ being the smallest nonnegative
integer  satisfying
\begin{eqnarray*}
 \langle T(u^n)-T(u^n-\gamma^mR(u_n)),R(u_n)\rangle \leq \sigma
\|R(u_n)\|^2.
\end{eqnarray*}
 Compute
\begin{eqnarray*}
g(u_{n+1}) :=P_K[g(u_n) +\alpha_n d_n], \quad n =0,1,2, \ldots
\end{eqnarray*}
where
\begin{eqnarray*}
d_n & = & -(\eta_nR(u_n)-\eta_nT(u_n)+T(y_n)) \\
\alpha_n &= & \frac{\eta_n\langle R(u_n),R(u_n)-
T(u_n)+T(y_n)\rangle}{\|d_n\|}.
\end{eqnarray*}
\end{algorithm}

To obtain a larger decrease of the distance from the next iterative point
to the solution set, we consider the following optimization problem
$$ \max \{ \phi_k(\alpha):  \alpha  \geq 0 \}. $$
Following the technique  of Wang et al\cite{181}, one can show that solution to the above optimization
problem is just the root, denoted by $\bar{\alpha_n}. $ \\ \\
If we choose $\overline\alpha_n$ as step size in Algorithm 15,
then we obtain another convergent algorithm. Obviously,
$\overline\alpha_n$ guarantees that the distance between the new
iterative point and the solution set has a larger decrease, so we
call $\alpha_n$ the basic step and $\overline \alpha_n$ the
optimal step. However, in practice, if $ K $ does not possess any
special structure, it is much expensive to compute
$\overline\alpha_n$. That is, we need to find a simple way to
compute the projection $P_K[u_n+\overline\alpha_n d_n]$.
Following the proof of Lemma 4.2 in \cite{187}, we can show that
$u_n(\overline \alpha_n)=P_{K\cap H_n}[u_n+\alpha_n d_n] $, where
$$H_n =\{u \in R^n~|~\eta_n\langle
R(u_n),R(u_n)-T(u_n)+T(y_n)\rangle+\langle u_n-u ,d_n\rangle=0\}.$$

Thus, we can obtain our improved double-projection method
for solving general variational inequalities.

\begin{algorithm}\label{alg16} For a  given $u_0 \in K, $ compute
\begin{eqnarray*}
g(z_n):=P_K[g(u_n)- T(u_n)]
\end{eqnarray*}
If $\|R(u_n)\|=0$, stop; otherwise compute
\begin{eqnarray*}
g(y_n ):=(1-\eta_n)g(u_n)+\eta_n g(z_n),
\end{eqnarray*}
 where $\eta_n=\gamma^{m_n}$ with $m_n$ being the smallest nonnegative
integer  satisfying\begin{eqnarray*}
 \langle T(u_n)-T(u_n-\gamma^mR(u_n)),R(u_n)\rangle \le \sigma
\|R(u_n)\|^2.
\end{eqnarray*}
 Compute
\begin{eqnarray*}
  g(u_{n+1}) =P_{H_n\cap K}[u_n+\alpha_n d_n], \quad n =0, 1,2, \ldots
\end{eqnarray*}
 where
\begin{eqnarray*}
 d_n &  = & -(\eta_nR(u_n)-\eta_nT(u_n)+T(y_n))\\
\alpha_n & = &\frac{\eta_n\langle
R(u_n),R(u_n)-T(u_n)+T(y_n)\rangle}{\|d_n\|}.
\end{eqnarray*}
\end{algorithm}
Notice that at each iteration in Algorithm 16, the latter projection
region is different from the former. More precisely, the latter projection
region is an intersection of the domain set $K $ and a hyperplane, so it
does not increase computation cost if $K$ is a polyhedral.\\

For $g=I, $ we now   give some numerical experiments for Algorithms 15  and 16 and
some comparison with other double-projection methods. Throughout the
computational experiments, the parameters used are
set as $\sigma=0.5,\gamma=0.8,$ and we use  $\|R(u_n)\| \leq
10^{-7}$ as stopping criteria. All computational results were
undertaken on a PC-II by MATLAB. We use symbol $e$ to denote the
vector whose components are all ones.\\

\begin{example}\label{example2} Consider the mapping $T: R^n\to R^n$  defined
by
$$F(x_1,x_2,x_3,x_4)=\left(\begin{array}{l}-x_2+x_3+x_4\\
                     x_1-(4.5x_3+2.7x_4)/(x_2+1)\\
                     5-x_1-(0.5x_3+0.3x_4)/(x_3+1)\\
                     3-x_1
                     \end{array}\right).$$
with the domain set
\begin{eqnarray*}
K =\{x\in R^n_+~|~e^\top x=1\},
\end{eqnarray*}
\end{example}

\begin{example}\label{example3} This example was tested by Sun [173.174].
Let $T(x)=Mx+q$ , where
$$M=\left(\begin{array}{ccccc}4&-1&0&\cdots&0\\
                     -1&4&-1&\cdots&0\\
                     \vdots&\vdots&\vdots&\vdots\\
                     0&0&0&\cdots&-1\\
                     0&0&0&\cdots&4\\
                     \end{array}\right),~~~q=\left(\begin{array}{c}-1\\
                     -1\\
                     -1\\
                     \vdots\\
                     -1
                     \end{array}\right)$$
with the domain set
\begin{eqnarray*}
K=\{x\in R^n_+~|~x_i\le1,i=1,2\cdots,n\}.
\end{eqnarray*}
It is easy to see that $T$ is strongly monotone on $R^n$.
\end{example}
\begin{example}\label{example4}  Define $T(x)=Mx+q$ , where
$$M=\mbox{diag}(1/n,2/n,\cdots,1),~~~q=(-1,-1,\cdots,-1)^\top.$$
with the domain set
$K=\{x\in R^n_+~|~x_i\le1,i=1,2\cdots,n\}, \quad  \mbox{see [137].}$
\end{example}
Again $T$ is strongly monotone on $ K $. The corresponding strongly
monotonicity modulus depends on the dimension $n$ and approaches
zeros when $n$ tends to infinity. Obviously, $x=e$ is its unique
solution. We choose the starting point $u_0=e$  for Example  \ref{example2} and choose $u
=(0,\cdots,0)^\top$ as starting point for Examples \ref{example3} and \ref{example4}
for different dimensions $n$. For double-projection methods [124,139,153
],  there always exist two step size rules just as in Algorithms 15 and
16. In the following, we give numerical comparison for these methods
using two different steps. The numerical results for double-
projection methods using the basic step for Examples \ref{example2}, \ref{example3},\ref{example4}  are
listed in Table 1, and the numerical results for double projection
methods using the optimal step for Examples \ref{example2}, \ref{example3},\ref{example4}  are listed in
Table 2 (the symbol ``$\setminus$" denotes the number of
iterations exceeds 1000 times).\\

$$\mbox{\bf Table 1. Numbers experience for Algorithm 15}$$
$$\begin{tabular}{|c|c|c|c|c|}
\hline Problem&Alg.[48]&Alg.[174]&Alg. 15 & Alg.[173]\\
\hline 4.1(n=4)&$\setminus$&$\setminus$&$\setminus$&$\setminus$\\
\hline 4.2$(n=10)$&47&57&47&117\\
\hline 4.2$(n=20)$&50&60&50&121\\
\hline 4.2$(n=50)$&52&62&52&125\\
\hline 4.2$(n=100)$&53&64&53&127\\
\hline 4.3($n=10$)&$\setminus$&$\setminus$&$\setminus$&$\setminus$\\
\hline 4.3$(n=20)$&$\setminus$&$\setminus$&$\setminus$&$\setminus$\\
\hline 4.3$(n=50)$&$\setminus$&$\setminus$&$\setminus$&$\setminus$\\
\hline 4.3$(n=100)$&$\setminus$&$\setminus$&$\setminus$&$\setminus$\\
\hline
\end{tabular}
$$
\vskip .2pc

$$\mbox{\bf Table 2. Numbers experience for Algorithm 16}$$
$$\begin{tabular}{|c|c|c|c|}
\hline Problem&Alg.[150]&Alg. 16& Alg.\cite{173}]\\
\hline 4.1(n=4)&49&96&109\\
\hline 4.2$(n=10)$&54&44&111\\
\hline 4.2$(n=20)$&57&47&115\\
\hline 4.2$(n=50)$&59&49&119\\
\hline 4.2$(n=100)$&60&50&121\\
\hline 4.3($n=10$)&73&35&156\\
\hline 4.3$(n=20)$&79&37&166\\
\hline 4.3$(n=50)$&85&40&180\\
\hline 4.3$(n=100)$&90&43&192\\
\hline
\end{tabular}
$$

\bigskip

Obviously, optimal step $\overline\alpha_n$ is better than the basic step
$\alpha_n$ for any direction. Compared with other double
projection methods, Algorithm 16 also shows a better behavior.  From
Table 1 and Table 2, it is clear that our new methods are as efficient as
the methods of  Solodov and Svaiter \cite{170,171,173}.
This shows that our Algorithm 15 and Algorithm 16 can be considered as
practical alternative to the extragradient and other modified projection
methods. The comparison of  new methods developed in this paper with the
recent methods is an interesting problem for  future research.

\section{ Dynamical Systems Technique}
In this section, we consider the projected dynamical systems
associated with variational inequalities. We investigate  the
convergence analysis of these new methods involving only the
monotonicity of the operator. \\
We now define the residue vector $R(u)$ by the relation
\begin{eqnarray}\label{5.1}
R(u)=g(u)-P_K[g(u)-\rho Tu].
\end{eqnarray}
Invoking Lemma \ref{lemma3}, one can easily conclude that $u\in H: g(u)\in  K$
is a solution of (\ref{2.5}), if and only if, $u\in H: g(u) \in K$ is a zero of
the equation
\begin{eqnarray}\label{5.2}
R(u)=0.
\end{eqnarray}
We now  consider a projected  dynamical system associated with the
 variational inequalities. Using the equivalent
formulation (\ref{5.2}), we suggest a class of projected
dynamical systems as
\begin{eqnarray}\label{5.3}
\frac{dg(u)}{dt}=\lambda \{P_K[g(u)-\rho Tu]-g(u)\},\quad u(t_{0})=u_{0}\in K,
\end{eqnarray}
where $\lambda$ is a parameter. The system of type (\ref{5.3}) is
called the projected dynamical system associated with  variational
inequalities (\ref{2.5}). Here the right hand is related to the
resolvent and is discontinuous on the boundary. From the definition,
it is clear that the solution of the dynamical system always stays
in $H$. This implies that the qualitative results such as the
existence, uniqueness and continuous dependence of the solution of
(\ref{5.3}) can be studied. These  projected dynamical systems are
associated with the general variational inequalities (\ref{2.5}), which have been
studied extensively.\\
We use the projected dynamical system (\ref{5.3}) to suggest some
iterative  for solving   variational inequalities (\ref{2.5}).
These methods can be viewed in the sense of Koperlevich \cite{60}
and Noor\cite{122}  involving the double resolvent operator.\\
For simplicity, we consider  the  dynamical system
\begin{eqnarray}\label{5.4}
\frac{dg(u)}{dt}+g(u)=P_K[g(u)-\rho Tu],\quad u(t_{0})=\alpha.
\end{eqnarray}
 We construct the implicit iterative method
using the forward difference scheme. Discretizing the equation (\ref{5.4}), we
have
\begin{eqnarray}\label{5.5}
\frac{g(u_{n+1})-g(u_{n})}{h}+g(u_{n+1})=P_K[g(u_{n})-\rho Tu_{n+1}],
\end{eqnarray}
where $h$ is the step size. Now, we can suggest  the following
implicit iterative method for solving the variational inequality
(\ref{2.5}).\\

\begin{algorithm}\label{alg17}
For a given $u_{0}\in H$, compute ${u_{n+1}}$ by the iterative scheme
\begin{eqnarray*}
g(u_{n+1})=P_K\bigg[g(u_{n})-\rho Tu_{n+1}-\frac{g(u_{n+1})-g(u_{n})}{h}\bigg],\quad n=0,1,2,\ldots.
\end{eqnarray*}
\end{algorithm}
This is an implicit method and is quite different from the known
implicit method. Using Lemma \ref{lemma1}, Algorithm
17 can be rewritten in the equivalent form as:\\
\begin{algorithm}\label{alg18}
For a given $u_{0}\in H$, compute ${u_{n+1}}$ by the iterative scheme
\begin{eqnarray}\label{eq5.2}
\langle \rho Tu_{n+1}+\frac{1+h}{h}(g(u_{n+1})-g(u_{n})),g(v)-g(u_{n+1})\rangle \geq
0, \quad\forall g(v) \in K.
\end{eqnarray}
\end{algorithm}
We now study the  convergence analysis  of algorithm 18 under some mild conditions.
\begin{theorem}\label{theorem5.3}
Let $u\in H: g(v) \in K$ be a solution of the general  variational inequality (\ref{2.5}).
Let $u_{n+1}$ be the approximate solution obtained from ( \ref{eq5.2}). If $T$ is $g$-monotone, then
\begin{eqnarray}\label{eq5.3}
\|g(u)-g(u_{n+1})\|^{2}\leq \|g(u)-g(u_{n})\|^{2}-\|g(u_{n})-g(u_{n+1})\|^{2}.
\end{eqnarray}
\end{theorem}

\begin{proof}
Let $u\in H: g(v) \in K$ be a solution of (\ref{2.5}). Then
\begin{eqnarray}\label{eq5.4}
\langle Tv,g(v)-g(u) \rangle\geq 0,\quad\forall v\in H: g(v) \in K,
\end{eqnarray}
since $T$ is  a $g$-monotone operator.\\
 Set $v=u_{n+1}$ in (\ref{eq5.4}), to  have
\begin{eqnarray}\label{eq5.5}
\langle Tu_{n+1},g(u_{n+1})-g(u) \rangle\geq 0.
\end{eqnarray}
Taking  $v=u$ in ( \ref{eq5.2}), we have
\begin{eqnarray}\label{eq5.6}
\langle \rho
Tu_{n+1}+\{\frac{(1+h)g(u_{n+1})-(1+h)g(u_{n})}{h}\},g(u)-g(u_{n+1})\rangle \geq
0.
\end{eqnarray}
From (\ref{eq5.5}) and (\ref{eq5.6}), we have
\begin{eqnarray}\label{eq5.7}
\langle (1+h)(g(u_{n+1})-g(u_{n})),g(u)-g(u_{n+1})\rangle\geq 0.
\end{eqnarray}
From  (\ref{eq5.7}) and using $\quad 2\langle a,b\rangle= \|a+b\|^2-\|a\|^2-\|b\|^2, \quad \forall a,b \in H,$ we obtain
\begin{eqnarray}\label{eq5.8}
\|g(u_{n+1})-g(u)\|^{2}\leq \|g(u)-g(u_{n})\|^{2}-\|g(u_{n+1})-g(u_{n})\|^{2}.
\end{eqnarray}
the required result. \hfill \qquad $\Box$
\end{proof}

\begin{theorem}\label{theorem5.4}
Let $u\in K$ be the solution of general  variational inequality
(\ref{2.5}). Let $u_{n+1}$ be the approximate solution obtained
from (\ref{eq5.2}). If $T$ is a $g$-monotone operator and $g^{-1}$ exists, then $u_{n+1}$
converges to $u\in H$ satisfying (\ref{2.5}).
\end{theorem}

\begin{proof}
Let  $T$ be a $g$-monotone operator. Then,  from (\ref{eq5.3}), it follows
the sequence $\{u_{i}\}^{\infty}_{i=1}$ is a bounded sequence and
\begin{eqnarray*}
\sum_{i=1}^{\infty}\|g(u_{n})-g(u_{n+1})\|^{2}\leq\|g(u)-g(u_{0})\|^{2},
\end{eqnarray*}
which  implies that
\begin{eqnarray}\label{eq5.9}
\lim_{n\rightarrow \infty}\|u_{n+1}-u_{n}\|^{2}=0,
\end{eqnarray}
since $g^{-1} $ exists.\\
Since sequence $\{u_{i}\}^{\infty}_{i=1}$ is bounded, so there
exists a cluster point $\hat{u}$ to which the subsequence
$\{u_{ik}\}^{\infty}_{k=1}$ converges. Taking limit in (\ref{eq5.2}) and
using (\ref{eq5.9}), it follows that $\hat{u}\in K$ satisfies
\begin{eqnarray*}
\langle T\hat{u},g(v)-g(\hat{u})\rangle\geq 0,\quad \forall v\in H: g(v) \in K,
\end{eqnarray*}
and
\begin{eqnarray*}
\|g(u_{n+1})-g(u)\|^{2}\leq\|g(u)-g(u_n)\|^{2}.
\end{eqnarray*}
Using this inequality, one can show that the cluster point $\hat{u}$
is unique and
\begin{eqnarray*}
\lim _{n\rightarrow \infty}u_{n+1}=\hat{u}.
\end{eqnarray*}
\hfill \qquad $\Box$

\end{proof}
We now suggest another implicit iterative method for solving (\ref{2.5}).  Discretizing (34), we
have
\begin{eqnarray}\label{eq5.10}
\frac{g(u_{n+1})-g(u_{n})}{h}+g(u_{n+1})=P_K[g(u_{n+1})-\rho Tu_{n+1}],
\end{eqnarray}
where $h$ is the step size.\\
 This formulation enable us to suggest the following iterative
 method.\\
\begin{algorithm}\label{alg19} For a given $u_{0}\in K,$ compute ${u_{n+1}}$ by the iterative
scheme
\begin{eqnarray*}
g(u_{n+1})=P_K\bigg[g(u_{n+1})-\rho
Tu_{n+1}-\frac{g(u_{n+1})-g(u_{n})}{h}\bigg].
\end{eqnarray*}
\end{algorithm}
Using lemma \ref{lemma1}, Algorithm 19  can be rewritten
in the equivalent form as:\\
\begin{algorithm}\label{alg20}
For a given $u_{0}\in K,$ compute ${u_{n+1}}$ by the iterative scheme
\begin{eqnarray}
\langle \rho Tu_{n+1}+\{\frac{g(u_{n+1})-g(u_{n})}{h}\},g(v)-g(u_{n+1})\rangle
\geq 0, \quad\forall v\in H: g(v)\in K.
\end{eqnarray}
\end{algorithm}
Again using the dynamical systems, we can suggested some iterative
methods for solving the variational inequalities and related
optimization problems.\\
\begin{algorithm}\label{alg21}
For a given $u_{0}\in K,$ compute ${u_{n+1}}$ by the iterative
scheme
\begin{eqnarray*}
u_{n+1}=P_K\bigg[\frac{(h+1)(g(u_n)-g(u_{n+1}))}{h}-\rho Tu_{n}\bigg],\quad
n=0,1,2,\ldots,
\end{eqnarray*}
\end{algorithm}
which can be written in the equivalent form as:

\begin{algorithm}\label{alg22}
For a given $u_{0}\in K,$ compute ${u_{n+1}}$ by the iterative
scheme
\begin{eqnarray}
\langle \rho
Tu_{n}+\{\frac{h+1}{h}(g(u_{n+1})-g(u_{n}))\},v-u_{n+1}\rangle \geq 0,
\quad\forall g(v)\in K.
\end{eqnarray}
\end{algorithm}
In a similar way, one can suggest a wide class of implicit iterative methods for solving variational inequalities and related optimization problems.
the comparison of these methods with other methods is an interesting problem for future research.

\section{Auxiliary Principle Technique }

In the previous sections, we have considered and analyzed several
projection-type methods for solving variational inequalities. It is well
known that to implement such type of the methods, one has to evaluate
the projection, which is itself a difficult problems. Secondly, one
can't extend the technique of projection  for solving some other classes of
variational inequalities. These facts motivate us  to consider other
methods. One of these techniques is known as the auxiliary principle.
This technique is basically due to Lions and Stampacchia \cite{172}.
 Glowinski et. al \cite{47} used this technique to study the existence of a solution of mixed
variational inequalities. Noor \cite{93,94,95,114,121,122} has used this technique to
develop some predictor-corrector methods for solving variational
inequalities.  It have been  be shown that various
classes of methods including projection, Wiener-Hopf, decomposition and
descent can be obtained from this technique as special cases.\\

For a given $u \in H, g(u) \in K $ satisfying (\ref{2.5}),  consider the problem of finding a
unique $w \in H, g(w) \in K $ such that
\begin{eqnarray}\label{6.1}
\la \rho Tu + g(w)-g(u), g(v)-g(w) \ra \geq 0, \quad \forall g(v) \in K,
\end{eqnarray}
where $\rho > 0 $ is a constant.\\

Note that, if $w = u, $ then $w $ is clearly a solution of the general
variational inequality (\ref{2.5}). This simple observation enables us to
suggest and analyze the  following predictor-corrector method.\\
\begin{algorithm}\label{23} For a given $ u_0 \in H, $  compute the
approximate solution $u_{n+1}$ by the iterative schemes
\begin{eqnarray*}\label{6.2}
&&\la \mu Tu_n+g(y_n)-g(u_n),g(v)-g(y_n) \ra \geq 0,  \quad \forall  g(v) \in K, \\
&& \la \beta  Ty_n+g(w_n)-g(y_n),g(v)-g(w_n)\ra \geq 0,
\quad  \forall g(v) \in K, \\
&& \la \rho Tw_n+g(u_{n+1})-g(w_n),g(v)-g(u_{n+1}) \ra \geq 0, \quad
\forall  g(v) \in K,
\end{eqnarray*}
\end{algorithm}
where $\rho > 0, \beta > 0 $ and $ \mu > 0 $ are constants

Algorithm 23 can be considered  as a three-step predictor-corrector
method, which was suggested and studied by Noor \cite{110,122}.\\
  If $\mu = 0, $ then Algorithm 23  reduces to:
\begin{algorithm}\label{alg24} For a given $ u_0 \in H, $ compute the
approximate solution $u_{n+1}$ by the iterative schemes;
\begin{eqnarray*}
&& \la \beta Tu_n+g(w_n)-g(u_n),g(v)-g(w_n)\ra \geq 0, \quad \forall g(v) \in K, \\
&& \la \rho Tw_n+g(u_{n+1})-g(w_n),g(v)-g(u_{n+1}) \ra  \geq 0, \quad
\forall g(v) \in K,
\end{eqnarray*}
\end{algorithm}
which is known as the two-step predictor-corrector method, see [110,122].\\
If $\mu = 0 , \beta = 0, $ then Algorithm 23 becomes:

\begin{algorithm}\label{alg25} For a given $u_0 \in H, $ compute
$u_{n+1}$ by the iterative scheme
\begin{eqnarray*}
\la \rho Tu_n+g(u_{n+1})-g(u_n),g(v)-g(u_{n+1})\ra \geq 0, \quad
\forall  g(v) \in   K.
\end{eqnarray*}
\end{algorithm}
Using the projection technique, Algorithm 23 can be written as

\begin{algorithm}\label{alg26} For a given $u_0 \in H, $ compute
$u_{n+1}$ by the iterative schemes
\begin{eqnarray*}
g(y_n)& = & P_K[g(u_n)-\mu Tu_n] \\
g(w_n)& = & P_K[g(y_n)-\beta Ty_n] \\
g(u_{n+1}) & = & P_K[g(w_n)-\rho Tw_n], \quad n=0,1,2, \ldots
\end{eqnarray*}
or
\begin{eqnarray*}
g(u_{n+1}) = P_K[I-\mu Tg^{-1}]P_K[I-\beta Tg^{-1}]P_K[I-\rho
Tg^{-1}]g(u_n), \quad n=0,1,2, \dots
\end{eqnarray*}
or
\begin{eqnarray*}
g(u_{n+1}) & = & (I+\rho Tg^{-1})^{-1}\{P_K[I-\rho Tg^{-1}]P_K[I-\rho
Tg^{-1}]P_K[I-\rho Tg^{-1}] \\
&& + \rho Tg^{-1}\}g(u_n), \quad n=0,1,2,\ldots,
\end{eqnarray*}
\end{algorithm}
which is three-step forward-backward method. See also two-step  forward-backward  splitting method of Tseng \cite{178,179}
for solving the classical variational inequalities.

\begin{definition}\label{6.1} For all $u,v,z \in H, $ the operator
$T:H \longrightarrow H $ is said to be:
\vskip .2pc
\noindent{\it (i). $g$-partially relaxed strongly monotone, } if there
exists a constant $\alpha > 0 $ such that
\begin{eqnarray*}
\la Tu-Tv,g(z)-g(v) \ra \geq -\alpha \|z-u\|^2
\end{eqnarray*}
\noindent{\it (ii). $g$-cocoercive, } if there exists a constant $\mu >
0 $ such that
\begin{eqnarray*}
\la Tu-Tv,g(u)-g(v) \ra \geq \mu \|Tu-Tv\|^2.
\end{eqnarray*}
\end{definition}
We remark that if $z = u ,$ then $g$-partially relaxed strongly
monotonicity is equivalent to  monotonicity.
For $g = I, $ definition \ref{6.1} reduces to the standard definition of
partially relaxed strongly monotonicity and cocoercivity of the operator.
We now show that $g$-cocoercivity implies  $g$-partially relaxed strongly
monotonicity. This result is due to Noor \cite{110,122}. To convey an idea, we
include its proof.
\begin{lemma}\label{lemma4} If $T $ is a $g$-cocoercive operator with
constant $\mu > 0, $ then $T$ is $g$-partially relaxed strongly monotone
operator with constant $\frac{-1}{4\mu }. $
\end{lemma}
\begin{proof} For all $u,v, z \in H, $ consider
\begin{eqnarray*}
\la Tu-Tv,g(z)-g(v) \ra & = & \la Tu-Tv,g(u)-g(v) \ra + \la Tu-Tv,
g(z)-g(u) \ra \\
& \geq & \mu \|Tu-Tv\|^2 -\mu \|Tu-Tv\|^2 - \frac{-1}{4\mu
}\|g(z)-g(u)\|^2,  \\
& \geq & \frac{-1}{4\mu }\|g(z)-g(u)\|^2,
\end{eqnarray*}
which shows that $ T $ is a $g$-partially relaxed strongly monotone
operator. \hfill  \qquad $\Box$
\end{proof}
One can easily show that the converse is not true. Thus
we conclude that the concept of $g$-partially relaxed strongly monotonicity
is a weaker condition that $g$-cocoercivity.\\
One can study the  convergence criteria of Algorithm 23 using the technique of Noor \cite{104}. \\
\begin{remark}
In the implementation of these algorithms, one does not
have to evaluate the projection. Our method
of convergence is very simple as compared with other methods. Following
the technique of Tseng \cite{178}, one can obtained new parallel and
decomposition algorithms for solving a number of problems arising in
optimization and mathematical programming.
\end{remark}
\begin{remark}  We note that, if the operator $g $ is linear or
convex, then the auxiliary problem (\ref{6.1}) is equivalent to finding the
minimum of the functional $I[w] $ on the convex set $K, $ where
\begin{eqnarray}\label{6.2}
I[w] & = & \frac{1}{2} \la g(w)-g(u),g(w)-g(u) \ra + \la \rho Tu,
g(w)-g(u) \ra \nonumber \\
&= & \|g(w)-(g(u)-\rho Tu)\|^2.
\end{eqnarray}
 It can be easily shown that the optimal solution of (\ref{6.2}) is the
projection of the point $( g(u)-\rho Tu )$ onto the convex set $K, $ that
is,
\begin{eqnarray}\label{6.3}
g(w(u)) = P_K[g(u)-\rho Tu],
\end{eqnarray}
which is the fixed-point characterization of the general variational
inequality (\ref{2.5}).
\end{remark}
 Based on the above observations, one can
show that the general variational inequality (\ref{2.5}) is equivalent to
finding the minimum of the functional $N[u] $ on $K $  in $H, $ where
\begin{eqnarray}\label{6.4}
N[u] & = & -\la Tu,g(w(u))-g(u) \ra -\frac{1}{2} \la
g(w(u))-g(u),g(w(u))-g(u) \ra \nonumber \\
& = & \frac{1}{2}\{\|\rho Tu\|^2-\|g(w(u))-(g(u)-\rho Tu)\|^2 \},
\end{eqnarray}
where $g(w) = g(w(u)). $ The function $N[u]$ defined by (\ref{6.4}) is known as
the gap (merit) function associated with the general variational
inequality (\ref{2.5}). This equivalence has been used to suggest and analyze a
number of methods for solving variational inequalities and nonlinear
programming, see, for example, Patriksson \cite{155}. In this direction, we
have:
\begin{algorithm}\label{alg27} For a given $u_0 \in H, $ compute the
sequence $\{ u_n\} $  by the iterative scheme
\begin{eqnarray*}
g(u_{n+1}) = g(u_n) + t_nd_n, \quad n=0,1,2, \ldots,
\end{eqnarray*}
where $d_n = g(w(u_n))-g(u_n) = P_K[g(u_n)-\rho Tu_n]- g(u_n) ,$ and $t_n
\in[0,1] $ are determined by the Armijo-type rule
\begin{eqnarray*}
N[u_n+ \beta_ld_n] \leq N[u_n]- \alpha \beta _l\|d_n\|^2.
\end{eqnarray*}
\end{algorithm}
It is worth to note the sequence $\{u_n\}$ generated by
\begin{eqnarray*}
g(u_{n+1})& = & (1-t_n)g(u_n) + t_nP_K[g(u_n)-\rho Tu_n ] \\
& = & g(u_n)-t_nR(u_n), \quad n=0,1,2,\ldots ,
\end{eqnarray*}
is very much similar to that generated by the projection-type Algorithm 3.
Based on the above observations and discussion, it is clear that the
auxiliary principle approach is quite general and flexible. This approach
can be used  not only to study the existence theory but also to suggest
and analyze various iterative methods for solving  variational
inequalities. Using the technique of Fukushima \cite{42}, one can easily study
the convergence analysis of Algorithm 27.\\

We have shown that the auxiliary principle technique can be sued  to
construct gap (merit) functions for general variational inequalities
(\ref{2.5}). We use the gap function to consider an optimal control problem
governed by the general variational inequalities (\ref{2.5}). The control
problem is an optimization problem, which  is also referred as  a generalized
bilevel programming problem or mathematical programming with equilibrium
constraints. It is known that the techniques of the classical optimal
control problems cannot be extended for variational inequalities, see Dietrich\cite{22}. This has
motivated us to develop some other techniques  including the notion of
conical derivatives, the penalty method and formulating the variational
inequality as operator equation with a set-valued operator. Furthermore, one can construct a so called gap
function associated with a variational inequality, so that the variational
inequality is equivalent to a scalar equation of the gap function. Under
suitable conditions such a gap function is Frechet differentiable and one
may use a penalty method to approximate  the optimal control problem and
calculate a regularized gap function in the sense of Fukushima \cite{42} to the
general variational inequality (\ref{2.5}) and determine their Frechet
derivative. Dietrich \cite{32,33} has  developed the similar results for the general variational
inequalities. We only give the basic properties of the optimal control
problem  and the associated gap functions to give an idea of the approach.\\

We now consider the following problem of optimal control for the general
variational inequalities (\ref{2.5}), that is, to find $u \in H: g(u) \in K, z
\in U $ such that
\begin{eqnarray*}
{\bf  P. }  \quad \min I(u,z), \quad \la T(u,z),g(v)-g(u) \ra \geq 0,
\quad \forall v \in H:g(v) \in K,
\end{eqnarray*}
where $H $ and $U$ are Hilbert spaces. The sets $K $ and $E $ are closed
 convex sets in $H $ and $U $ respectively. Here $H $ is the space of
state and $K \subset H $ is the set of state constraints for the problem.
$U $ is the space of control and closed convex set $E \subset U $ is the
set of control constraints.  $T(.,.): H\times U \longrightarrow H $ is a
an operator which is Frechet differentiable. The functional $I(.,.) : H
\times U \longrightarrow R\cup \{+\infty \} $ is a proper, convex and
lower-semicontinuous function. Also we assume that the problem ${\cal P }$
has at least one optimal solution denoted by $(u^*,z^*) \in H\times U. $\\

Related to the optimization problem $({\bf P }),$ we consider the
regularized gap (merit) function $h_\rho (u, z):H\times U \longrightarrow R
$ as
\begin{eqnarray}\label{6.5}
h_\rho (u,z) &=& \sup_{v \in H:g(v) \in K }\{ \la -\rho T(u,z),g(v)-g(u) \ra \nonumber \\
&&-\frac{1}{2}\|g(v)-g(u)\|^2 \} \quad \forall v\in H:g(v) \in K.
\end{eqnarray}
We remark that the regularized function (\ref{6.5}) is a natural generalization
of the regularized gap function (\ref{6.4}) for variational inequalities. It
can be shown that the regularized gap function $h _\rho (.,.) $ defined by
(\ref{6.5}) has the following properties. The analysis is in the spirit of Dietrich \cite{33}.

\begin{theorem}\label{theorem6}\quad The gap function $h_\rho (.,.) $
defined by (\ref{6.5}) is well-defined and
\begin{eqnarray*}
(i).  &&\forall  u \in H :g(u) \in K, z \in U, \quad h_\rho (u,z) \geq 0. \\
(ii). && h_\rho (u,z)=\frac{1}{2}\{\|\rho
^2\|T(u,z)-d^{2}_{K}(g(u)-\rho T(u,z))\}, \\
(iii).  &&h_\rho (u,z)=-\rho \la T(u,z),g(u_K)-g(u)\ra
-\frac{1}{2}\|g(u_K)-g(u)\|^2,
\end{eqnarray*}
where $d_K $ is the distance to $  K $ and
\begin{eqnarray*}
g(u)=P_K[g(u)-\rho T(u,z)]
\end{eqnarray*}
\end{theorem}
\begin{proof}  It is well-known that
\begin{eqnarray*}
d^{2}_{K} = \min _{v \in H:g(v)\in K}\|g(v)-g(u)\|^2 =
\|g(u)-P_K[g(u_K)]\|^2
\end{eqnarray*}
Take $v = u $ in (\ref{6.5}). Then clearly (i) is satisfied.\\
Let $(u,z) \in H\times U. $ Then
\begin{eqnarray*}
h_{\rho }(u,z) & = & \rho \la T(u,z),g(u)\ra -\frac{1}{2}\|g(u)\|^2 \\
&&+ \sup_{v\in H:g(v)\in K}\left[\la -\rho T(u,z),g(v)\ra
-\frac{1}{2}\|g(v)\|^2+\la g(u),g(v)\ra \right] \\
 &=& \rho \la T(u,z),g(u)\ra -\frac{1}{2}\|g(u)\|^2 \\ &&+ \inf _{v \in H:g(v)\in K }\left[\frac{1}{2}\|g(v)\|^2-\la g(u)-\rho T(u,z),g(v) \ra
\|^2 \right] \\
& = &\rho \la T(u,z),g(u)\ra -\frac{1}{2}\|g(u)\|^2|| \\ &&-\frac{1}{2}\inf _{v\in
H:g(v) \in K}\|g(v)-(g(u)-\rho T(u,z))\|^2 + \frac{1}{2}\|g(u)-\rho T(u,z) \|^2 \\
 & = & \frac{\rho ^2}{2}\|T(u,z)\|^2-\frac{1}{2}d^{2}_{K}(g(u)-\rho
T(u,z)).
\end{eqnarray*}
Setting $g(u_K) = P_K[g(u)-\rho T(u,z)], $ we have
\begin{eqnarray*}
h_\rho (u,z) & = & \frac{\rho ^2}{2}\|T(u,z)\|^2-\frac{1}{2}\|g(u)-\rho
T(u,z)-g(U_K)\|^2 \\
& = & -\rho \la T(u,z),g(v)-g(u)\ra -\frac{1}{2}\|g(u_K)-g(u)\|^2.
\end{eqnarray*}  \hfill \qquad $\Box$
\end{proof}

\begin{theorem}\label{theorem7}\quad If the set $K$ is $g$-convex  in $H, $
then the following are equivalent.
\begin{eqnarray*}
(i). \quad && h_\rho (u,z) = 0, \mbox{for all }\quad  u\in H: g(u) \in K,
z \in U \\
(ii). \quad&& \la T(u,z),g(v)-g(u) \ra \geq 0, \quad \forall
u,v \in H: g(u),g(v) \in K, z \in U. \\
(iii). \quad && g(u)= P_K[g(u)-\rho T(u,z)].
\end{eqnarray*}
\end{theorem}
\begin{proof} We show that$ (ii) \Longrightarrow  (i).$\\
Let $u \in H $ and $z \in U $ be a solution of
\begin{eqnarray*}
\la T(u,z),g(v)-g(u) \ra \geq 0, \quad \forall v \in H: g(v) \in K.
\end{eqnarray*}
Then we have
\begin{eqnarray*}
h_{\rho } (u,z) = -\rho \la T(u,z),g(v)-g(u) \ra
-\frac{1}{2}\|g(v)-g(u)\|^2 \leq 0,
\end{eqnarray*}
which implies that
\begin{eqnarray*}
h_\rho (u,z) \leq 0.
\end{eqnarray*}
Also for $ v\in H:g(v) \in K, $ we know that
\begin{eqnarray*}
h_\rho (u,z) \geq 0.
\end{eqnarray*}
From these above inequalities, we have (i), that is, $ h_\rho (u,z) = 0. $\\
Conversely, let (i) hold. Then
\begin{eqnarray}\label{6.7}
 -\rho \la T(u,z),g(v)-g(u) \ra -\frac{1}{2}\|g(v)-g(u)\|^2 \leq 0,
\forall v \in H : g(v) \in K.
\end{eqnarray}
Since $K$ is a $g$-convex set, so for all $g(w),g(u) \in K, \quad t \in
[0,1],\\ \quad g(v_t)= (1-t)g(u) +g(w) \in K. $ Setting $g(v)=g(v_t) $ in
(\ref{6.7}), we have
\begin{eqnarray*}
-\rho \la T(u,z),g(w)-g(u) \ra -\frac{t}{2}\|g(w)-g(u)\|^2 \leq 0.
\end{eqnarray*}
Letting $t \longrightarrow 0, $ we have
\begin{eqnarray*}
\la T(u,z),g(v)-g(u) \ra \geq 0, \quad \forall g(w) \in K.,
\end{eqnarray*}
the required (ii).
Thus we conclude that (i) and (ii) are equivalent. Applying Lemma \ref{lemma1}, we
have $(ii) = (iii). $ \hfill \quad $\Box $
\end{proof}
From Theorem \ref{theorem6} and Theorem \ref{theorem7}, we conclude that the optimization
problem ${\cal P}$  is equivalent to
\begin{eqnarray*}
\min I(u,z), \quad h_\rho (u,z) = 0, \quad \forall u \in
H:g(u) \in K, z \in U,
\end{eqnarray*}
where $h_\rho (u,z) $ is ${\cal C}^{1}$-differentiable in the sense of
Frechet, but is not convex.\\

If the operators $T, g $ are  Frechet differentiable, then the gap
function $h_\rho (u,z) $ defined by (\ref{6.5}) is also Frechet differentiable.
In fact,
\begin{eqnarray*}
h^{\prime }_{\rho }(u,z) = \rho ^2[T^{\prime }(u,z)]^{\ast
}T(u,z)-([g^{\prime }(u)]^{\ast }-\rho [T^{\prime }(u,z)]^{\ast
})(I-P_k)[g(u)-\rho T(u,z)],
\end{eqnarray*}
where $[T^{\prime }(u,z)]^{\prime } $ is the adjoint operator of
$T^{\prime }(u,z). $ This implies the following connection at a point
$(u_1,z_1) $
\begin{eqnarray*}
h^{\prime }_{\rho }(u_1,z_1) = \rho \cdot [g^{\prime }(u_1)]^{\ast
}T(u_1,z_1),
\end{eqnarray*}
which is a solution of the general variational inequality (\ref{2.5}), that is,
for $u_1,z_1) $ with $h_\rho (u_1,z_1) = 0. $

For the optimal problem ${\cal P }, $ we have
\begin{eqnarray*}
h^{\prime }_{\rho }(u^*,z^*) = \rho \cdot [g^{\prime }(u^*)]^{\ast
}T(u^*,z^*).
\end{eqnarray*}

We now consider a simple example of optimal control problem to illustrate
\begin{eqnarray*}
\min ({\bf P}_{1}) &:&=\min \left\{ u^{2}+z^{2}\left|
\begin{array}{c}
(u+z-1)(v^{2}-u-z^{2})\geq 0\quad \\ \forall v\in R:v^{2}\geq 1 \\ (u,z)\in
R^{2}:u+z^{2}\geq 1
\end{array}
\right. \right\}  \\
T(u,z) &=&u+z-1,\quad g(u)=u,\quad K=[1,+\infty ).
\end{eqnarray*}
First, we solve the general variational inequality (\ref{2.5})
\begin{eqnarray*}
\text{Case 1} &:&\text{}T(u,z)=z+u-1=0 \\ &\Longrightarrow
&{\cal L}_{1}=\left\{ (u,z)=(1-z,z)\in R^{2}\left| \,\
z\in (-\infty ,0]\cup \lbrack 1,+\infty )\right. \right\}  \\
\text{Case 2 } &:&\text{}T(u,z)=z+u-1>0 \\ &\Longrightarrow
&{\cal L}_{2}=\left\{ (u,z)=(1-z^{2},z)\in R^{2}\left|
\,\ u\in (0,1)\right. \right\}  \\ \text{Case 3 }
&:&\text{}T(u,z)=z+u-1<0
\\ &\Longrightarrow &{\cal L}_{3}=\emptyset  \\ {\cal L}
&=&\left\{ (u,z)\in R^{2}\left|
\begin{array}{c}
u =1-z\quad \text{for\quad }\ z\in (-\infty ,0]\cup \lbrack 1,+\infty ) \\
u=1-z^{2}\quad \text{for}\quad z\in (0,1)
\end{array}
\right. \right\} .
\end{eqnarray*}
We obtain as the unique optimal solution of \ ${\bf P}_{1}
(u_{opt},z_{opt})=(\frac{1}{2}\sqrt{2},\frac{1}{2})$ \ with \ \ $\min (
{ \cal P}_{1})=\frac{3}{4}.$\\
Next, we calculate the gap function of the general variational
inequality problem (\ref{2.5}).
\begin{eqnarray*}
h_{1}(u,z) &=&\frac{1}{2}\left( z+u-1\right)
^{2}-\frac{1}{2}\left[ \left(
I-P_{[1,+\infty )}\right) \left( z^{2}-z+1\right) \right] ^{2} \\
&=&\left\{
\begin{array}{c}
\frac{1}{2}\left( z+u-1\right) ^{2}-\frac{1}{2}\left(
z^{2}-z\right)
^{2}\quad \text{for\quad }z\in (0,1) \\
\frac{1}{2}\left( y+u-1\right) ^{2}\quad \text{for}\quad z\in
(-\infty ,0]\cup \lbrack 1,+\infty )
\end{array}
\right.
\end{eqnarray*}
 This shows that equivalence between these problems.
\[
(u,z)\in R^{2}:u+z^{2}\geq 1\,\,\text{and}\,\
h_{1}(u,z)=0\,\,\Longleftrightarrow \text{ \ }(u,z)\in H\times U:
\]
\[
\la T(u,z), g(v)-g(u) \ra \geq 0, \quad \forall g(v) \in K.
\]

\section{ Penalty Function Method }

In this section,  we consider a system of third-order boundary value
problems, where the solution is required to satisfy some extra continuity
conditions on the subintervals in addition to the usual boundary
conditions. Such type of systems of boundary value arise in the study of
obstacle, free, moving and unilateral problems and has important
applications in various branches of pure and applied sciences. Despite of
their importance, little attention has been given to develop efficient
numerical methods for solving numerically these systems except
for special cases. In particular, it is known that if the obstacle
function is known then the general variational inequalities can be
characterized by a system of odd-order boundary value problems  by using
the penalty method. This technique is called the
penalty function method and was used by Lewy and Stampacchia \cite{64} to study
the regularity of a solution of variational inequalities.
The computational advantage of this technique is  its simple
applicability for solving the system of differential equations. This
technique has been explored and developed by Noor and its co-worker
to solve the systems of differential equations associated with even and
odd-order obstacle problems.  Our approach  to
these problems is to consider them in a general manner  and specialize them
later on. To convey an idea of the technique involved, we first
introduce two numerical schemes for solving a system
of third boundary value problems using the  splines. An
example involving the odd-order obstacle is given.\\

For simplicity, we consider a system of obstacle third-order boundary value problem
of the type
\begin{eqnarray}\label{7.1}
u^{\prime \prime \prime } = \left \{\begin{array}{lll}
f(x), \quad \quad      a \leq x \leq c, \\ \\
p(x)u(x)+f(x)+r, \quad \quad   c \leq x \leq d, \\ \\
f(x), \quad \quad  d \leq x \leq b, \end{array} \right.
\end{eqnarray}
with the boundary conditions
\begin{eqnarray}\label{7.2}
u(a) = \alpha \quad  u^{\prime}(a) = \beta, _1 \quad \mbox{and} \quad
u^{\prime } = \beta _2,
\end{eqnarray}
and the continuity conditions of $u, u^{\prime } $ and $ u^{\prime \prime}
$ at $c$ and $d. $ Here  $f $ and $p $ are continuous functions on $[a,b]
$ and $[c,d]$ respectively. The parameters $,r, \alpha , \beta _1 $ and $
\beta _2$ are real finite constants. Such type of systems arise in the
study of obstacle, free, moving and unilateral boundary value problems and
has important applications in other branches of pure and applied sciences.
In general, it is not possible  to obtain the  analytical solutions of
(\ref{7.1}) for arbitrary choice of $ f(x) $ and $ p(x) .$ We usually resort to
numerical methods for obtaining the approximate solutions of (\ref{7.1}). Here
we use cubic spline functions to derive some consistency relations which
are then used to develop a numerical technique for solving system of
third-order boundary value problems.  Without loss of generality, we take
$c =\frac{3a+b}{4} $ and $d = \frac{a+3b}{4}$ in order to derive a
numerical method for approximating the solution of the system (\ref{7.1}).  For
this purpose, we divide the interval [a,b] into $n$ equal subintervals
using the grid point
\begin{eqnarray*}
x_i= a + ih, \quad i=0,1,2, \ldots ,
\end{eqnarray*}
with
\begin{eqnarray*}
x_0 = a, \quad  x_n = b, \quad h = \frac{b-a}{n+1},
\end{eqnarray*}
where $n$ is a positive integer chosen such that both $\frac{n+1}{4} $ and
$\frac{3(n+1)}{4}$ also positive integer. Also, let $u(x) $ be the exact
solution of (\ref{7.1}) and $s_i$ be an approximation to $u_i = u(x_i) $
obtained by the cubic $P_i(x) $ passing through the points $(x_i,s_i)$ and
$(x_{i+1},s_{i+1}) .$ We write $P_i(x) $ in the form
\begin{eqnarray}\label{7.3}
P_i(x) = a_i(x-x_i)^3 +b_i(x-x_i)^2 +c_i(x-x_i) + d_i,
\end{eqnarray}
for $ i = 0,1,2, \ldots ,n-1. $ Then the cubic spline is defined by
\begin{eqnarray}\label{7.4}
s(x) & = & P_i(x), \quad i = 0,1,2, \ldots ,n-1, \nonumber \\
s(x) &\in & C^2[a,b].
\end{eqnarray}
We now develop explicit expressions for the four coefficients in (\ref{7.3}).
To do this, we first design
\begin{eqnarray}\label{7.5}
P_i(x_i) & = & s_i, \quad P_i(x_{i+1}) = s_{i+1} \quad P^{\prime }_i(x_i)
= D_i, \nonumber \\
P_{i}^{\prime \prime \prime }(x_i)& = &\frac{1}{2}[T_{i+1}+T_i], \quad
\mbox{for} \quad i=0,1,2, \ldots ,n-1,
\end{eqnarray}
and
\begin{eqnarray}\label{7.6}
T_{i} = \left \{\begin{array}{ll}
f_{i}, \quad \quad  \mbox{for} \quad 0 \leq i \leq \frac{n}{4} \quad
\mbox{and} \quad \frac{n}{4} < i \leq n, \\ \\
p_{i}s_{i}+ f_{i}+ r, \quad \quad \mbox{for} \quad \frac{n}{4} < i \leq
\frac{3n}{4},
\end{array} \right.
\end{eqnarray}
where $f_i = f(x_i) $ and $p_i = p(x_i).$

Using the above discussion, we obtain the following relations
\begin{eqnarray}\label{7.7}
a_i & = & \frac{1}{12}[T_{i+1} + T_i], \nonumber \\
b_i & = & \frac{1}{h^2}[s_{i+1}-s_i]-\frac{1}{h}[T_{i+1}+ T_i], \nonumber
\\
c_i & = & D_i, \\
d_i & = & s_i,   \quad   i=0,1,2,\ldots ,n-1.  \nonumber
\end{eqnarray}
Now from the continuity of the cubic spline $s(x)$ and its derivatives
up to  order two at the point $(x_i,s_i) $ where the two cubic
$P_{i_1}(x)$ and $P_i(x) $ join, we can have
\begin{eqnarray}\label{7.8}
P^{(m)}_{i-1} = P_{i}^{(m)}, \quad \quad  m=0,1,2, \ldots
\end{eqnarray}
From the above relations, one can easily obtain the following consistency
relations
\begin{eqnarray}
h[D_i +D_{i-1}] & = & 2[s_i-s_{i-1}] + \frac{h^3}{12}[T_i +T_{i-1}],\label{7.9} \\
h[D_i +D_{i-1}] & = & s_{i+1}-2s_i + s_{i-1}
-\frac{h^3}{12}[T_{i+1}+3T_i+2T_{i-1}]. \label{7.10}
\end{eqnarray}
From (\ref{7.9}) and (\ref{7.10}), we obtain
\begin{eqnarray}\label{7.11}
2hD_i = s_{i+1}-S_{i-1}-\frac{h^3}{12}[T_{i+1}+ 2T_i + T_{i-1}].
\end{eqnarray}
Eliminating $D_i $ from (\ref{7.11}).(\ref{7.10}) and (\ref{7.9}), we have
\begin{eqnarray}\label{7.12}
-s_{i-2}+3s_{i-1}-3s_i+s_{i+1} = \frac{1}{12}h^3[T_{i-2}+5T_{i-1}+5T_i
+ T_{i+1}],
\end{eqnarray}
for $ i = 2,3,  \ldots ,n-1.$ The recurrence relations (\ref{7.12}) gives
$(n-2)$ linear equations in the unknowns $s_i, i = 1,2, \ldots ,n.$  We
need two more equations one at each end of the range of integration. These
two equations are:
\begin{eqnarray}\label{7.13}
3s_0 -4s_1 + S_2 & = & -2hD_0 + \frac{h^3}{12}[3T_0 +4T_1 +T_2], \quad i
=1, \\
-3s_{n-2} + 8s_{n-1} -5s_n & = & -2hD_{n+1} + \frac{h^3}{12}[3T_{n-2}
+10T_{n-1} + 31T_n],  i = n.
\end{eqnarray}
The cubic spline solution of (\ref{7.3}) is based on the linear equations given
by (\ref{7.11})-(\ref{7.12}).
The local truncation errors $T_i, \quad i =1,2, \ldots $ associated with
the cubic spline method  are given by
\begin{eqnarray*}
t_{i} = \left \{\begin{array}{lll}
-\frac{1}{10}h^5u^{(5)}(\zeta _1) + O(h^6), \quad a < \zeta _1 < x_2 \quad
i =1, \\ \\
-\frac{1}{6}h^5u^{(5)}(\zeta _i) + O(h^6), \quad x_{i-2} < \zeta _i <
x_{i+1} \quad  2 \leq i \leq n -1, \\ \\
-\frac{1}{10}h^5u^{(5)}(\zeta _n) + O(h^6), \quad x_{n-2} < \zeta _n < b,
\end{array} \right.
\end{eqnarray*}
which indicates that this method is a second order convergent process.\\

To illustrate the applications of the numerical methods developed above,
we consider the third order-order obstacle boundary value problem (\ref{2.15}).
Following the penalty function technique of Levy and Stampacchia [64], the
variational inequality (\ref{7.13}) can be written as
\begin{eqnarray}\label{7.14}
\la Tu, g(v) \ra + \la \nu \{u-\psi \}(u-\psi ),g(v) \ra = \la f, g(v)
\ra, \quad \mbox{for all } \quad g(v) \in H,
\end{eqnarray}
where $ \nu \{t\} $ is the discontinuous function defined by
\begin{eqnarray}\label{7.15}
\nu \{t\} = \left \{\begin{array}{ll}
1, \quad \quad  \mbox{for} \quad t \geq 0 \\ \\
0, \quad \quad \mbox{for} \quad  t < 0
\end{array} \right.
\end{eqnarray}
is known as the penalty  function and $ \psi < 0 $ on the boundary is the
called the obstacle function. It is clear that problem (\ref{2.15}) can be
written in the form
\begin{eqnarray}\label{7.16}
-u^{\prime \prime \prime } + \nu \{u-\psi \}(u-\psi ) = f, \quad  \quad 0
< x < 1.
\end{eqnarray}
with
\begin{eqnarray*}
u(0) = u^{\prime }(0) \quad =\quad u^{\prime }(1) = 0
\end{eqnarray*}
where $ \nu \{t\}$ is defined by (\ref{7.15}). If the obstacle function $\psi $
is known and is given by the relation
\begin{eqnarray}\label{7.18}
\psi (x) = \left \{\begin{array}{ll}
-1, \quad \quad \mbox{for} \quad 0 \leq x \leq \frac{1}{4} \quad
\mbox{and} \quad \frac{3}{4} \leq x \leq 1 \\ \\
1, \quad \quad  \mbox{for} \quad \frac{1}{4} \leq x \leq
\frac{3}{4},\end{array} \right.
\end{eqnarray}
then problem (\ref{2.15}) is equivalent to the following system of third-order
differential equations
\begin{eqnarray}\label{7.19}
u^{\prime \prime \prime } = \left \{\begin{array}{ll}
f, \quad \quad \mbox{for} \quad 0 \leq x \frac{1}{4} \quad \mbox{and}
\frac{3}{4} \leq x \leq 1 \\ \\
u +f -1 \quad \quad \mbox{for}\quad  \frac{1}{4} x \leq \frac{3}{4},
\end{array} \right.
\end{eqnarray}
with the boundary conditions
\begin{eqnarray}\label{7.20}
u(0) = u^{\prime } (0) = u^{\prime}(1) = 0
\end{eqnarray}
and the conditions of continuity of $u, u^{\prime }$ and $u^{\prime \prime
} $ at $x = \frac{1}{4} $ and $\frac{3}{4}.$ It is obvious that problem
(\ref{7.19}) is a special case of problem (\ref{7.1}) with $p(x) =1 $ and $r =-1.$
\vskip .4pc
Note that for $f =0, $ the system of differential equations (\ref{7.19})
reduces to
\begin{eqnarray}\label{7.21}
u^{\prime \prime \prime } = \left \{ \begin{array}{ll}
0, \quad \quad \mbox{for} \quad 0 \leq x \leq \frac{1}{4} \quad \mbox{and}
\quad \frac{3}{4} \leq x \leq 1 \\ \\
u-1, \quad \quad \mbox{for} \quad \frac{1}{4} \leq x \leq \frac{3}{4}
\end{array} \right.
\end{eqnarray}
with the boundary condition (\ref{7.20}).
\vskip .4pc
The analytical solution for  this problem is
\begin{equation}\label{7.22}
u(x) = \left\{ \begin{array}{ll}
\frac{1}{2} a_{1} x^{2}, \quad & 0 \leq x \leq \frac{1}{4} \\ \\
1 + a_{2} e^{x}  + e^{- \frac{x}{2}} [ a_{3} \cos
\frac{\sqrt{3}}{2} x + a_{4} \sin \frac{\sqrt{3}}{2} x ], \quad &
\frac{1}{4} \leq x \leq \frac{3}{4} \\ \\
\frac{1}{2} a_{5} x(x-2) + a_{6}, \quad &
\frac{3}{4} \leq x  \leq 1.
\end{array} \right.
\end{equation}
To find the constants $a_{i}, \; i = 1,2, \ldots , 6$, we apply
the continuity conditions of $u, u^{\prime}$  and
$u^{\prime \prime}$  at $x = \frac{1}{4}$  and $\frac{3}{4}$,
which leads to the following system of linear equations
\begin{displaymath}
\left[ \begin{array}{llllll}
\frac{1}{32} & -S_{1}  & - S_{2} CS_{1}  & - S_{2} S C_{1} & \ \,
0 & \ \, 0 \\ \\
\frac{1}{4} & -S_{1}  & \ \, \frac{1}{2} S_{2} (\sqrt{3} SC_{1}
+ CS_{1})  & - \frac{1}{2} S_{2} (\sqrt{3} CS_{1} - S C_{1}) &
\ \, 0 & \ \, 0 \\ \\
1 & -S_{1}  & -\frac{1}{2} S_{2} (\sqrt{3} SC_{1} - CS_{1})  &
\ \, \frac{1}{2} S_{2} (\sqrt{3} CS_{1} + S C_{1}) & \ \, 0 &
\ \, 0 \\ \\
0 & \ \, S_{3}  & \ \, S_{4} CS_{2}  & \ \, S_{4} S C_{2} &
\ \, \frac{15}{32} & -1 \\ \\
0 & \ \, S_{3}  & -\frac{1}{2} S_{4} (\sqrt{3} SC_{2}+ CS_{2})  &
\ \, \frac{1}{2} S_{4} (\sqrt{3} CS_{2} - S C_{2}) & \ \,
\frac{1}{4} & \ \, 0 \\ \\
0 & \ \, S_{3}  & \ \ \, \frac{1}{2} S_{4}
(\sqrt{3} SC_{2} - CS_{2})  & \ \, \frac{1}{2} S_{4}
(- \sqrt{3} CS_{2} - S C_{2}) & -1 & \ \, 0
\end{array} \right] \left[ \begin{array}{c}
a_{1} \\ \\
a_{2} \\ \\
a_{3} \\ \\
a_{4} \\ \\
a_{5} \\ \\
a_{6}    \end{array} \right] = \left[ \begin{array}{c}
\ \, 1 \\ \\
\ \, 0 \\ \\
\ \, 0 \\ \\
-1 \\ \\
\ \, 0 \\ \\
\ \, 0  \end{array} \right],
\end{displaymath}
where
\begin{displaymath}
S_{1} = \exp (\frac{1}{4}), \; S_{2} = \exp ( - \frac{1}{8}),
\; S_{3} = \exp ( \frac{3}{4}), \; S_{4} = \exp (- \frac{3}{8}),
\end{displaymath}
\begin{displaymath}
CS_{1} = \cos \frac{\sqrt{3}}{8}, \; SC_{1} = \sin
\frac{\sqrt{3}}{8},  \; CS_{2} = \cos \frac{3 \sqrt{3}}{8} \; \;
\mbox{and} \;  \; SC_{2} = \sin  \frac{3 \sqrt{3}}{8}.
\end{displaymath}
One can find the exact solution of this system of linear equations
by using Gaussian elimination method.\\

For various values of $h,$ the system of third-order of boundary value
problem defined by (\ref{7.19}) and (\ref{7.20}) was solved using the numerical
method developed in this section. A detailed comparison is given in Table 3.
\begin{center}
Table 3 : Observed maximum errors $\| {\bf e} \|$ for problem (6.x).

\begin{tabular}{ c c c c }    \hline
$h$ & Quartic spline & Cubic spline  & Colloc-quintic \cite{3}\\ \hline
$\frac{1}{16}$ & 1.15 $ \times 10^{-3}$ & 1.23 $ \times 10^{-3}$ &
1.26 $\times 10^{-3}$ \\
$\frac{1}{32}$ & 5.32 $ \times 10^{-4}$ & 5.53 $ \times 10^{-4}$ &
5.60 $\times 10^{-4}$ \\
$\frac{1}{64}$ & 2.56 $ \times 10^{-4}$ & 2.61 $ \times 10^{-4}$ &
3.10 $\times 10^{-4}$ \\
$\frac{1}{128}$ & 1.26 $ \times 10^{-4}$ & 1.27 $ \times 10^{-4}$ &
1.61 $\times 10^{-4}$ \\ \hline
\end{tabular}
\end{center}
\vskip .4pc
From Table 3, it is clear that quartic spline method
gives better results than cubic and quintic splines methods
developed earlier for solving system of third-order boundary value
systems.\\ \\
 For more details for solving various classes of  obstacle boundary
value problems using the penalty technique, see \cite{3,4,5,114,149} and the references therein.
In recent years, homotopy (analysis)  perturbation method, Adomonian decomposition , Laplace transformation and variational iteration techniques
are being used to find the analytical solutions of  fractional unilateral and obstacle boundary value problems.

\section{ General Equilibrium Problems}

In this section, we introduce and consider a class of equilibrium problems
known as general equilibrium problems. It is known that equilibrium
problems \cite{19,141} include variational and complementarity problems as special
cases. We note that the projection and its variant forms including the
Wiener-Hopf equations cannot be extended to equilibrium problems, since it
is not possible to find the projection of the bifunction $F (.,.).$
Noor \cite{94,95,121} used the auxiliary principle technique to analyse some iterative methods
for equilibrium problems. In this chapter, we introduce and study a class of equilibrium problems
involving the arbitrary function, which is called the general equilibrium problem.
We show that the auxiliary principle technique can be used
to suggest and analyze some iterative methods for solving general
equilibrium problems. We also study the convergence analysis of these
iterative methods and discuss some special cases.\\
 For  given nonlinear function $F(.,.) : H\times H \longrightarrow  R$
and operator $g : H \longrightarrow R, $ we consider the problem of
finding $u \in H, \quad g(u) \in K $ such that
\begin{eqnarray}\label{8.1}
F(u,g(v)) \geq 0, \quad  \forall  g(v) \in K,
\end{eqnarray}
which is known as the {\it general equilibrium problem}.\\
 We now discuss some special cases of general equilibrium problem (\ref{8.1}),\\
 \noindent{\bf(I).} For $g \equiv I, $ the identity operator,  problem (\ref{8.1}) is equivalent to finding $u \in K $
such that
\begin{eqnarray}\label{8.2}
F(u,v ) \geq 0, \quad \forall \quad v \in K,
\end{eqnarray}
which is called the equilibrium problem, which was introduced and studied
by Blum and Oettli \cite{19}. For the recent applications and development, see
\cite{102,109,152} and the reference therein.\\
\noindent{\bf (II).} If $F(u,g(v)) = \la Tu, \eta (g(v),g(u)) \ra $ and the set $K_{g\eta}  $ is an invex set
in $H, $ then problem (\ref{8.1}) is equivalent to finding $u \in K_{g\eta} $ such that
\begin{eqnarray}\label{8.3}
\la Tu, \eta (g(v),g(u)) \ra \geq 0, \quad \forall v \in K_\eta.
\end{eqnarray}
Inequality of type (\ref{8.3}) is known as the general variational-like inequality,
which arises as a minimum of general preinvex functions on the general invex set $K_{g\eta}. $ \\

\noindent{\bf (III).}  We  note that, for $F(u,g(v) \equiv \la Tu, g(v)-g(u) \ra, $ problem (\ref{8.1})
reduces to problem (\ref{2.5}), that is, find $u \in K, \quad g(u) \in K $ such
that
\begin{eqnarray*}
\la Tu,g(v)-g(u) \ra \geq 0, \quad \forall g(v) \in K,
\end{eqnarray*}
which is exactly the general variational inequality (\ref{2.5}). Thus we
conclude that general equilibrium problems (\ref{8.1}) are quite general and unifying one.
 \\

We now use the auxiliary principle technique as
developed in Section 6 to suggest and analyze some iterative methods for
solving general equilibrium problems (\ref{8.1}).\\
For a given $u \in H, \quad g(u) \in K $ satisfying (\ref{8.1}), consider the auxiliary
equilibrium problem of finding $ w \in H, g(w) \in K $ such that
\begin{eqnarray}\label{8.4}
\rho F(u,g(v) ) + \la g(w)-g(u),g(v)-g(w) \ra \geq 0, \quad \forall g(v) \in K.
\end{eqnarray}
Obviously, if $w = u, $ then $w $ is a solution of the general equilibrium
problem (\ref{8.1}). This fact allows us to suggest  the following
iterative method for solving (\ref{8.1}).\\

\begin{algorithm}\label{alg28}  For a given $u_0 \in H, $ compute the
approximate solution $u_{n+1} $ by the iterative scheme:
\begin{eqnarray}\label{8.5}
\rho F(w_n,g(v)) + \la g(u_{n+1})-g(w_n),g(v)-g(u_{n+1}) \ra \geq 0,
\quad \forall g(v) \in K.
\end{eqnarray}
\begin{eqnarray}\label{8.6}
\beta F(u_n,g(v) ) + \la g(w_n)-g(u_n),g(v)-g(w_n) \ra \geq 0, \quad
\forall g(v) \in K,
\end{eqnarray}
where $\rho > 0 $ and $ \beta > 0 $ are constants.
\end{algorithm}
Algorithm 28 is called the predictor-corrector method for solving
general equilibrium problem (\ref{8.1}).\\ For $ g = I, $ where $I $ is the
identity operator, Algorithm 28 reduces to:
\begin{algorithm}\label{alg29} For a given $u_0 \in H, $ compute the
approximate solution $u_{n+1} $ by the iterative schemes
\begin{eqnarray*}
\rho F(w_n, v) &+ &\la u_{n+1}-w_n, v-u_{n+1} \ra \geq 0, \quad \forall v \in K. \\
\beta F(u_n, v) & + & \la w_n-u_n, v-w_n \ra \geq 0, \quad \forall v \in K
\end{eqnarray*}
\end{algorithm}
Algorithm 29 is also a  predictor-corrector method for solving
equilibrium problem  and appears to be a new one.\\

If $F(u,g(v)) = \la Tu, g(v)-g(u) \ra , $ then Algorithm 28 becomes:

\begin{algorithm}\label{alg30} For a given $u_0 \in H, $ compute the
approximate solution $u_{n+1} $ by the iterative scheme
\begin{eqnarray*}
\la \rho Tw_n & + & g(u_{n+1})-g(w_n),g(v)-g(u_{n+1}) \ra \geq 0, \quad \forall g(v) \in K, \\
\la \beta Tu_n & + & g(w_n)-g(u_n),g(v)-g(w_n) \ra \geq 0, \quad \forall g(v) \in K,
\end{eqnarray*}
\end{algorithm}
which is a two-step method for solving general variational
inequalities (\ref{2.5}).\\

In brief, for suitable and appropriate choice of the functions $F(.,.) $
and the operators $T, g, $ one can obtain various algorithms developed in
the previous sections.\\

For the convergence analysis of Algorithm 28, we need the following
concepts.\\
\begin{definition}\label{def8.1} The function $F(.,.): H \times H
\longrightarrow H $ is said to be: \\
\noindent{\bf (i). {\it $g$-monotone} }, if
\begin{eqnarray*}
F(u,g(v) ) + F(v,g(u)) \leq 0, \quad \forall u,v \in H.
\end{eqnarray*}
\noindent{\bf (ii). {\it $g$-pseudomonotone, }} if
\begin{eqnarray*}
F(u,g(v) ) \leq 0, \quad \mbox{implies } \quad  F(v,g(u)) \leq 0, \quad
\forall u,v \in H.
\end{eqnarray*}
\noindent{\bf (iii). {\it $g$-partially relaxed strongly monotone, }} if
there exists a constant $ \alpha > 0 $ such that
\begin{eqnarray*}
F(u,g(v)) + F(v,g(z)) \leq \alpha \|g(z)-g(u)\|^2, \quad \forall   u,v,z \in H.
\end{eqnarray*}
\end{definition}
Note that for $u = z, $ $g$-partially relaxed strongly monotonicity
reduces to $g$-monotonicity of $F(.,.).$ \\ For $g = I, $ Definition \ref{def8.1}
coincides with the standard definition of monotonicity, pseudomonotonicity
of the function $F(.,.) $.

\vskip .2pc
We now consider the convergence analysis of Algorithm 28.  \\

\begin{theorem}\label{theorem8.1}Let $ \bar{u} \in H $ be a solution of (\ref{8.1}) and let $u_{n+1} $ be an approximate solution obtained from
Algorithm 28. If the bifunction $F(.,.) $ is $g$-partially relaxed
strongly monotone with constant $ \alpha > 0, $ then
\begin{eqnarray}
\|g(\bar{u})-g(u_{n+1})\|^2 &\leq & \|g(\bar{u})-g(w_n) \|^2 -(1-2\alpha
\rho ) \|g(w_n)-g(u_{n+1})\|^ 2 \label{8.7}\\
\|g(\bar{u})-g(w_{n})\|^2 &\leq & \|g(\bar{u})-g(u_n) \|^2 -(1-2\beta
\rho ) \|g(w_n)-g(u_{n})\|^ 2.  \label{8.8}
\end{eqnarray}
\end{theorem}
\begin{proof} Let $ g(\bar{u} )  \in H $ be a solution of (\ref{8.1}).  Then
\begin{eqnarray}
\rho F(\bar{u}, g(v) ) &\geq & 0, \quad \forall g(v) \in K. \label{8.9}\\
\beta F(\bar{u},g(v)) & \geq & 0, \quad \forall g(v) \in K,\label{8.10}
\end{eqnarray}
where $\rho > 0 $ and $\beta > 0 $ are constants.

Now taking $ v= u_{n+1} $ in (\ref{8.5}) and $v = \bar{u} $ in (\ref{8.9}), we have
\begin{eqnarray}\label{8.11}
\rho F(\bar{u}, g(u_{n+1})) \geq 0
\end{eqnarray}
and
\begin{eqnarray}\label{8.12}
\rho F(w_n,g(\bar{u})) + \la g(u_{n+1})- g(w_n),g(\bar{u})-g(u_{n+1})\ra
\geq 0.
\end{eqnarray}
Adding (\ref{8.11}) and (\ref{8.12}), we have
\begin{eqnarray}\label{8.13}
\la g(u_{n+1})-g(w_n),g(\bar{u})-g(u_{n+1})& \geq & -\rho
\{F(w_n,g(\bar{u})+ F(\bar{u},g(u_{n+1}))\} \nonumber \\
& \geq & -\alpha \rho \|g(u_{n+1})-g(w_n)\|^2,
\end{eqnarray}
where we have used the fact that $F(.,.) $ is $g$-partially relaxed
strongly monotone with constant $\alpha > 0 .$ Using  the inequality
$$ 2\langle a,b \rangle - \|a+b\|^2-\|a\|^2-\|b\|^2, \forall a,b \in H, $$ we obtain
\begin{eqnarray}\label{8.14}
2\la g(u_{n+1})-g(w_n),g(\bar{u})- g(u_{n+1}) \ra & = &
\|g(\bar{u})-g(w_n)\|^2 - \|g(\bar{u})-g(u_{n+1})\|^2 \nonumber \\
&& -\|g(u_{n+1})-g(w_n)\|^2.
\end{eqnarray}
Combining (\ref{8.13}) and (\ref{8.14}), we have
\begin{eqnarray}\label{8.15}
\|g(\bar{u})-g(u_{n+1})\|^2 \leq \|g(\bar{u})-g(w_n)\|^2 -(1-2\rho \alpha
)\|g(w_n)-g(u_{n+1})\|^2,
\end{eqnarray}
the required (\ref{8.7}).\\
Taking $ v = \bar{u} $ in (\ref{8.6}) and $ v = w_n $ in (\ref{8.10}), we obtain
\begin{eqnarray}\label{8.16}
\beta F(\bar{u},g(w_n) ) \geq 0
\end{eqnarray}
and
\begin{eqnarray}\label{8.17}
\beta F(u_n,g(\bar{u})) + \la g(w_n)-g(u_n), g(\bar{u})-g(w_n) \ra \geq 0.
\end{eqnarray}
Adding (\ref{8.16}), (\ref{8.17}) and rearranging the terms, we have
\begin{eqnarray}\label{8.18}
\la g(w_n)-g(u_n), g(\bar{u})-g(w_n) \ra \geq -\alpha \beta \|g(u_n)-g(w_n)\|^2,
\end{eqnarray}
since $F(.,.) $ is $g$-partially strongly monotone with constant $\alpha > 0. $\\
Consequently, from (\ref{8.18}), we have
\begin{eqnarray*}
\|g(\bar{u})-g(w_n)\|^2  \leq  \|g(\bar{u})-g(u_n)\|^2  - (1-2\alpha
\beta )\|g(u_n)-g(w_n) \|^2 ,
\end{eqnarray*}
the required (\ref{8.8}).
\end{proof}

\begin{theorem}\label{theorem8.2} Let $H$ be a finite dimension subspace and
let $ 0 < \rho < \frac{1}{2\alpha }$ and\\  $ 0 < \beta < \frac{1}{2\alpha
}.$  If $\bar{u} \in H :g(\bar{u}) \in K$ is a solution of
(\ref{8.1}) and $u_{n+1}$ is an approximate solution obtained from Algorithm 28,
then $$ \lim_{n \longrightarrow \infty }u_n = \bar{u}.$$
\end{theorem}
\begin{proof} Its proofs is very much similar to that of Noor \cite{122}.
\end{proof}

We again use  the auxiliary principle technique to suggest an inertial
proximal method for solving general equilibrium problem (\ref{8.1}). It is
noted that  inertial proximal method include the proximal method as a
special case.\\

For a given $u \in H, \quad g(u) \in K  $ satisfying (\ref{8.1}), consider the auxiliary general
equilibrium problem of finding $ w \in H, \quad g(w) \in K $ such that
\begin{eqnarray}\label{8.20}
\rho F(w,g(v)) + \la g(w)-g(u)-\alpha_n (g(u)-g(u)),g(v)-g(w) \ra \geq 0, \forall g(v) \in K,
\end{eqnarray}
where $\rho > 0 $ and $\alpha_n > 0 $ are constants.\\
It is clear that if $ w = u, $ then $ w$ is a solution of the general
equilibrium problem (\ref{8.1}). This fact enables us to suggest  an iterative
method for solving (\ref{8.1}) as:\\
\begin{algorithm}\label{alg31} For a given $ u_0 \in H, $ compute the
approximate solution $u_{n+1} $  by the iterative scheme
\begin{eqnarray*}\label{8.21}
&&(\rho F(u_{n+1},g(v))) + \la g(u_{n+1})-g(u_n) \nonumber \\
&&-\alpha _n (g(u_n)-g(u_{n-1})),g(v)-g(u_{n+1}) \ra
 \geq 0, \quad \forall g(v) \in K,
\end{eqnarray*}
\end{algorithm}
where $\rho > 0 $ and $\alpha _n > 0 $ are constants.

Algorithm 31 is called the inertial proximal point method. For $\alpha
_n =0, $ Algorithm 31 reduces to:

\begin{algorithm}\label{alg32}  For a given $u_0 \in H, $ find the
approximate solution $u_{n+1} $ by the iterative schemes
\begin{eqnarray*}
(\rho F(u_{n+1}),g(v))) + \la g(u_{n+1})-g(u_n),g(v)-g(u_{n+1}) \ra \geq 0,
\quad \forall g(v) \in K,
\end{eqnarray*}
\end{algorithm}
which is known as the proximal method and appears to be a new one. Note
that for $g \equiv I, $ the identity operator, one can obtain inertial
proximal method for solving equilibrium problems (\ref{8.2}).
In a similar way, using the technique of Noor \cite{122}, one can suggest and analyze
several new inertial type methods for solving general equilibrium problems.
It is an challenging  problem to compare the efficiency of these methods with other techniques
for solving general equilibrium problems.

\section{General variational-like inequalities}

It is well known that the minimum of the (non) differentiable preinvex
functions on the invex set can be characterized by a class of variational inequalities,
called variational-like inequalities. For the applications and numerical methods of variational-like inequalities,
see\cite{18,93,95,106} and the references therein. In this section, we introduce the general variational-like inequalities
 with respect to an arbitrary function. Due the structure of the general variational inequalities,
the projection method and its variant forms cannot
be used to study the problem of the existence of the solution.  This implies that the variational-like inequalities are not
equivalent to the projection (resolvent) fixed-point problems. We use
the  auxiliary principle technique to  suggest and analyze some implicit and explicit iterative methods
for solving variational-like inequalities. We also show that the
general variational-like inequalities are equivalent to the optimization problems,
which can be used to study the associated optimal control problem. Such
type of the problems have been not studied for general variational-like
inequalities and this is another direction for future research.\\

We recall some known basic concepts and results.

Let $F:K_{\eta} \rightarrow R$
be a continuous function and let $\eta(.,.) :K_{\eta}\times K_{\eta} \rightarrow R$ be
an arbitrary continuous bifunction. Let $g(.)$ be a non-negative function.

\begin{definition}\emph{\cite{14}} The set $K_{\eta}$ in $H$ is said to be invex set with
respect to an arbitrary  bifunction $\eta(\cdot,\cdot),$ if
\begin{eqnarray*}
u+t\eta(v,u)\in K,\quad\quad \forall u,v\in K_{\eta}, t\in[0,1].
\end{eqnarray*}
\end{definition}
The invex set $K_{\eta}$ is also called $\eta$-connected set. Note that the innvex set with $\eta(v,u)=v-u$
is a convex set, but the converse is not true. \\

From now onward,  $K_{\eta}$ is a nonempty closed invex set in $H$ with
respect to the bifunction $\eta(\cdot,\cdot),$ unless otherwise
specified.\\

\begin{definition}\emph{\cite{14}} The set $K_{g\eta}$ in $H$ is said to be general invex set with
respect to an arbitrary  bifunction $\eta(\cdot,\cdot)$ and the function $g, $if
\begin{eqnarray*}
g(u)+t\eta(g(v),g(u))\in K_{g\eta},\quad\quad \forall g(u),g(v)\in K_{g\eta}, t\in[0,1].
\end{eqnarray*}
\end{definition}
The invex set $K_{g\eta}$ is also called $g\eta$-connected set. Note that the general invex set with $\eta(g(v),g(u))=g(v)-g(u) $
is a general convex set, but the converse is not true. See Youness \cite{196}.\\
We now the concept of the general preinvex function.
\begin{definition}.\label{def9.1}  Let $K_{g\eta} \subseteq H$ be a general invex set
with respect to $\eta (.,.): K_{\eta} \times K_{\eta} \lrt R^n$ and $g: H \lrt  H. $ A
function $F: K_{g\eta}  \lrt R $ is said to be general preinvex function, if,
\begin{eqnarray*}
F(g(u)+t\eta (g(v),g(u))) \leq (1-t)F(g(u)) + t F(g(v)),
\\  \quad \forall   u,v \in H: g(u),g(v) \in K_{g\eta}, t \in [0,1],
\end{eqnarray*}
\end{definition}
Note that for $g \equiv I$, the  general preinvex functions are called the
preinvex functions.  For $\eta (v,u) = g(v)-g(u),$  general preinvex function are
known as general convex functions.  Every convex function is a general convex
function and every general convex function is a general preinvex function, but the
converse is not true.  see [11,43].\\

{\it From now onward, we assume that the set $K_{g\eta} $ is a general invex set with
respect to the functions $\eta (.,.) : K_{g\eta} \times K_{g\eta} \lrt H,\quad  g: K_{g\eta}
\lrt H ,$   unless otherwise specified.}

\begin{definition}\label{def9.2}  The function $F $ is said to
be general semi preinvex, if
\begin{eqnarray*}
F(g(u)+t\eta (g(v),g(u))) \leq (1-t)F(u)+tF(v), \\
\quad \forall  u,v \in H: g(u),g(v) \in K_{g\eta}, \quad t\in [0,1].
\end{eqnarray*}
\end{definition}
For $g \equiv I,$ and $t= 1,$ we have
\begin{eqnarray*}
F(u+\eta (v,u)) \leq  F(v), \quad  \forall   u,v \in K_{g\eta}.
\end{eqnarray*}

\begin{definition}\label{def9.3} The function $F $ is called general quasi  preinvex, if
\begin{eqnarray*}
F(g(u)+t \eta (g(v),g(u))) \leq \max\{F(g(u)),F(g(v))\},\\  \quad \forall  u,v \in H: g(u),g(v) \in K_{g\eta}, \quad t \in [0,1].
\end{eqnarray*}
\end{definition}
The function $F$ is called the strictly  general quasi preinvex, if strict
inequality holds for all $g(u),g(v) \in K_{g\eta}, g(u) \neq g(v).$ The function $F$ is said
to be  general quasi preconcave, if and only if,$-F$ is general quasi preinvex. A
function which is both general quasi preinvex and general quasi preconcave is called the general quasimonotone

\begin{definition}\label{def9.4} The function $F $ is said to be general
logarithmic preinvex on the general invex set $K_{g\eta}$ with respect to the
bifunction $\eta(.,.)  $ and the function  $g,$ if
\begin{eqnarray*}
F(g(u)+t\eta (g(v),g(u))) \leq (F(g(u)))^{1-t}(F(g(v)))^t, \\ \quad \forall  u,v \in H; g(u),g(v) \in K_{g\eta}, \quad t \in [0,1],
\end{eqnarray*}
where $F(.) > 0.$
\end{definition}
Clearly for $t = 1, $ and $ g = I, $ we have
\begin{eqnarray*}
F(u + \eta (v,u)) \leq F(v), \quad  \forall  u,v \in K_{g\eta}.
\end{eqnarray*}

It  follows that\\

general logarithmic preinvexity $\Longrightarrow $ general preinvexity
$\Longrightarrow $  general quasi preinvexity.  \\

For appropriate and suitable choice of the operators  and spaces, one can
obtain several classes of generalized preinvexity.

In this section, we prove that the minimum of a differentiable general preinvex function on
the general invex sets can be characterized by a class of variational-like
inequalities, which is called the general variational-like inequality..

\begin{theorem}\label{theorem9.1}
Let $ F$ be a differentiable general preinvex. Then $u \in H: g(u) \in K_{g\eta} $ is
a minimum of $F$ on $K_{\eta},$ if and only if, $u \in H: g(u)\in K_{g\eta}$ satisfies
\begin{eqnarray}\label{eq9.1}
\la F'(g(u)), \eta (g(v),g(u)) \ra \geq 0, \quad \forall v\in H: g(v) \in K_{g\eta},
\end{eqnarray}
where $F'$ is the Frechet derivative of $F$ at $g(u) \in K_{g\eta}.$
\end{theorem}

\begin{proof}Let $ u\in H:g(u) \in K_{g\eta} $ be a minimum of the  $F. $ Then
\begin{eqnarray}\label{eq9.2}
F(g(u)) \leq F(g(v)), \quad \forall v\in H: g(v) \in K_{g\eta}.
\end{eqnarray}
Since the set $K_{\eta} $ is a general invex set, so  $ \forall  g(u), g(v) \in K_{g\eta}, $
and $t \in [0,1], $ $$g(v_t)= g(u) +\eta (g(v),g(u)) \in K_{g\eta}. $$
Setting $g(v)  = g(v_t) $ in (\ref{eq9.2}), we have
\begin{eqnarray*}
F(g(u)) \leq F(g(u)+ \eta (g(v),g(u))).
\end{eqnarray*}
Dividing the above inequality by $t$  and taking limit as $t \longrightarrow 0, $ we have
\begin{eqnarray*}
\la F^{\prime }(g(u)),\eta (g(v),g(u)) \ra \geq 0, \quad \forall v\in H: g(v) \in K_{g\eta}
\end{eqnarray*}
the required (\ref{eq9.1}).

Conversely, let $ u \in H:g(u) \in K_{g\eta} $ satisfy the inequality (\ref{eq9.1}). Then,
using the fact that function $F $ is a general preinvex function, we have
\begin{eqnarray*}
F(g(u)+t\eta (g(v),g(u)))- F(g(u)) \leq t\{F(g(v))-F(g(u))\}, \quad  g(u, g(v) \in K_{g\eta}.
\end{eqnarray*}
Dividing the above inequality by $t$ and letting $t \longrightarrow 0, $
we have
\begin{eqnarray*}
F(g(v))-F(g(u)) & \geq & \la F^{\prime }(g(u)),\eta (g(v),g(u))\ra  \geq  0, \quad \mbox{using (\ref{eq9.1}) }
\end{eqnarray*}
which implies that
\begin{eqnarray*}
 F(g(u)) \leq F(g(v)),
\end{eqnarray*}
which shows that $u \in H: g(u)\in K_{g\eta} $ is a minimum of the general preinvex
function on the general invex set $K_{g\eta}$ in $H. $
\end{proof}

Inequalities of the type (\ref{eq9.1}) are called the general
variational-like inequalities. For $g \equiv I,$ where $I$ is the identity
operator, Theorem \ref{theorem9.1} is mainly  due to Noor \cite{93}. From Theorem \ref{theorem9.1}  it
follows that the general variational-like inequalities (\ref{eq9.1}) arise
naturally in connection with the minimum of general preinvex function over general invex sets. In many
applications, problems like (\ref{eq9.1}) do not arise as a result
of minimization. This fact motivated us to  consider a problem of finding a solution of a more
general variational-like inequality of which (\ref{eq9.1}) is a special case.\\

Given (nonlinear) operator $T: H \longrightarrow H,$ and $\eta :K_{g\eta}\times K_{g\eta}
\longrightarrow R,$ where $K_{g\eta} $ is a  nonempty general invex set in $ H, $ we
consider the problem of finding $ u \in H: g(u)\in K_{g\eta}$ such that
\begin{eqnarray}\label{eq9.3}
\langle Tu, \eta (g(v),g(u)) \rangle  \geq 0,  \quad   v \in H: g(v) \in K_{g\eta},
\end{eqnarray}
which  is known as the  general variational-like  inequality.\\

If $\eta (v,u) = g(v)-g(u) ,$  then the general invex set $K_{g\eta}$
becomes a general convex set $K_g.$ In this case, problem (\ref{eq9.3}) is equivalent to
finding $u \in H : g(u) \in K_g $ such that
\begin{eqnarray*}
\langle Tu, g(v)-g(u) \rangle  \geq 0, \quad \forall  v\in H: g(v) \in K_g,
\end{eqnarray*}
which is exactly  the general  variational inequality(\ref{2.5}).
 For formulation, numerical methods, sensitivity analysis, dynamical system and other aspects of general variational inequalities, see \cite{109,110,122,143} and the references therein.\\

For suitable and appropriate choice of the operators $T,$ $\eta ,$ and the
general invex set, one may derive a wide class of known and new variational
inequalities as special cases of problem (\ref{eq9.3}). It is clear that
general variational-like inequalities provide us a  framework to study a wide class of unrelated problems in  a unified setting.\\

We now use  the auxiliary principle technique to suggest and analyze some iterative methods for general
variational-like inequalities (\ref{eq9.3}). \\

For a given $u \in H; g(u) \in K_{g\eta }$ satisfying (\ref{eq9.3}), consider the problem of finding a  solution $w \in H: g(w)
\in K_{g\eta}$ satisfying the auxiliary variational-like inequality
\begin{eqnarray}\label{eq9.4}
\langle \rho Tw+E'(g(w))-E'(g(u)), \eta (g(v),g(w)) \rangle   \geq 0, \quad \forall v\in H: g(v) \in K_{g\eta},
\end{eqnarray}
where $\rho > 0 $ is a constant and $E^{\prime } $ is the differential of
a strongly general preinvex function $E. $ The inequality (\ref{eq9.4}) is called
the auxiliary general variational-like inequality.\\

Note that if $w =u $, then clearly $w$ is solution of the general variational-like
inequality (\ref{eq9.3}). This observation enables us to suggest the following
algorithm for solving (\ref{eq9.3}).
\vskip .3pc

\begin{algorithm}\label{alg9.1} For a given $u_0 \in H,$ compute
the approximate solution $u_{n+1}$ by the iterative scheme
\begin{eqnarray}\label{9.5}
\langle \rho Tu_{n+1}+E'(g(u_{n+1}))-E'(g(u_n)),\eta (g(v),g(u_{n+1}))\rangle  \geq 0, \\
\quad \forall v\in H; g(v) \in K_{g\eta}.\nonumber
\end{eqnarray}
\end{algorithm}
Algorithm \ref{alg9.1} is called the proximal point algorithm for solving
the general variational-like inequalities (\ref{eq9.3}).\\

For $ \eta (g(v),g(u))= g(v)-g(u) , $ general preinvex function $E$ is equivalent to the
convex function and the invex set $K_{g\eta} $ becomes the general  convex set. Consequently Algorithm \ref{alg9.1} reduces to:
\begin{algorithm}\label{alg9.2} For a given $u_0 \in K_g,$ compute the
approximate solution $u_{n+1}$ by the iterative scheme
\begin{eqnarray*}
\la \rho Tu_{n+1}+E'(g(u_{n+1}))-E'(g(u_n)), g(v)- g(u_{n+1}) \ra \geq 0, \quad \forall g(v) \in K_g,
\end{eqnarray*}
\end{algorithm}
which is known as the proximal point algorithm for solving general variational
inequalities.

\begin{remark}
The function $$B(g(w),g(u)) = E(g(w))-E(g(u))-\la
E'(g(u)),\eta (g(w),g(u)) \ra $$ associated with the preinvex functions $E(u) $ is
called the general Bregman function. We note that,if $\eta (g(v),g(u)) =
g(v)-g(u), $ then $$B(g(w),g(u))= E(g(w))-E(g(u))-\la E'(g(u)),g(v)-g(u) \ra $$ is the well known
Bregman function. For the applications of Bregman function in solving
variational inequalities and related optimization problems, see \cite{200}.
\end{remark}

We now study the convergence analysis of Algorithm \ref{alg9.1}. For this purpose,
we recall the following concepts.

\begin{definition}\label{def9.15}  $ \forall u,v,z \in H,$ an operator
$T: H \lrt H $ is said with respect to an arbitrary function $g: H \lrt H $  to be: \\
(i). \quad {\it  general $g\eta $-pseudomonotone,} if
\begin{eqnarray*}
\la Tu, \eta (g(v),g(u)) \ra \geq 0 \quad \Longrightarrow \quad \la Tv, \eta
(g(v),g(u)) \ra \geq 0.
\end{eqnarray*}
(ii). \quad {\it general $g\eta $-Lipschitz continuous, } if there exists a constant
$\beta > 0 $ such that
\begin{eqnarray*}
\la Tu-Tv, \eta (g(u),g(v) \ra \leq \beta \|g(u)-g(v)\|^2.
\end{eqnarray*}
(iii). \quad {\it general  $g\eta $-cocoercive, } if there exists a constant $\mu >
0 $ such that
\begin{eqnarray*}
\la Tu-Tv, \eta (g(u),g(v)) \ra \geq \mu \|T(g(u))-T(g(v))\|^2.
\end{eqnarray*}
(iv). \quad {\it general  $g\eta $-partially relaxed strongly monotone, } if there
exists a constant $\alpha > 0 $ such that
\begin{eqnarray*}
\la Tu-Tv, \eta (g(z),g(v)) \ra \geq \mu \|g(z)-g(u)\|^2.
\end{eqnarray*}
\end{definition}

For $ \eta (g(v),g(u)) = g(v)-g(u)  , $ definition \ref{def9.15}  reduces to the
definition of general monotonicity, general Lipschitz continuity, general co-coercivity
and partially relaxed general strongly monotonicity of the operator $T.$ We note
that for $g(z)=g(u), $ partially strongly monotonicity reduces to
monotonicity. One can easily show that general $g\eta$-cocoercivity implies general $g\eta $-partially relaxed general strongly monotonicity, but
the converse is not true.

\begin{definition}\label{def9.16} A function $F $ is said to be strongly
general preinvex function on $K_{g\eta} $ with respect to the function $ \eta (.,.)$ with
modulus $\mu > 0$ and function $g, $ if,  \\ $\forall u,v \in H: g(u), g(v) \in K_{g\eta} , t \in [0,1], $ such that
\begin{eqnarray*}
F(g(u)+t\eta (g(v),g(u)))\leq (1-t)F(g(u))+tF(g(v)) -t(1-t)\mu\|\eta(g(v),g(u))\|^2.
\end{eqnarray*}
\end{definition}
We note that the differentiable strongly  general preinvex function $F $ implies the
strongly general invex function, that is,
\begin{eqnarray*}
F(g(v))-F(g(u)) \geq \la F^{\prime }(g(u)), \eta (g(v),g(u))\ra + \mu \|\eta(g(v),g(u))\|^2,
\end{eqnarray*}
but the converse is also true under some conditions.

\begin{assumption}\label{assump9.1}    $ \forall u,v,z \in H,$ the operator
$\eta :H\times H \lrt H $  and the function $g $ satisfy the condition
\begin{eqnarray*}
\eta (g(u),g(v)) = \eta (g(u),g(z)) + \eta (g(z),g(v)).
\end{eqnarray*}
\end{assumption}
In particular, from Assumption \ref{assump9.1}, we  obtain
$$\eta (g(u),g(v)) = -\eta (g(v),g(u)) $$ and \\$$\eta (g(v),g(u)) = - \eta (g(u),g(v)), \forall
 g(u),g(v), \in H. $$ Assumption \ref{assump9.1} has been used  to
study the existence of a solution of general variational-like inequalities.\\
\begin{theorem}\label{thm9.2}Let $T$ be a general $\eta $-pseudomonotone operator.
Let $E$ be a strongly differentiable general preinvex function with modulus
$\beta $ and Assumption \ref{assump9.1} hold. Then the approximate solution $u_{n+1}$
obtained from Algorithm \ref{alg9.1} converges to a solution of (\ref{eq9.3}).
\end{theorem}
\begin{proof} Since the function  $E$ is strongly general preinvex ,
so the solution $u_{n+1}$ is unique. Let $u \in H: g(u) \in K_{\eta} $ be a solution of the
general variational-like inequality (\ref{eq9.3}).  Then
\begin{eqnarray*}
\la Tu, \eta (g(v),g(u)) \ra  \geq 0, \quad \forall v\in H; g(v) \in K_{\eta},
\end{eqnarray*}
which implies that
\begin{eqnarray}\label{eq9.6}
\la Tv, \eta (g(v),g(u)) \ra  \geq 0,
\end{eqnarray}
since $T$ is a general $\eta $-pseudomonotone.
 Taking $v = u_{n+1} $ in (\ref{eq9.6}), we have
\begin{eqnarray}\label{eq9.7}
\la Tu_{n+1},\eta (g(u_{n+1}), g(u) )\ra \geq 0.
\end{eqnarray}
We consider the Bregman function
\begin{eqnarray}\label{eq9.8}
B(g(u),g(w)) & = & E(g(u))-E(g(w))-\la E'(g(u)),\eta (g(v),g(u)) \ra \nonumber \\
  & \geq & \frac{\beta }{2}\|\eta(g(u),g(w))\|^2,
\end{eqnarray}
 using strongly general preinvexity. \\
Now
\begin{eqnarray}\label{eq9.9}
B(g(u),g(u_n))-B(g(u),g(u_{n+1}))& = & E(g(u_{n+1}))-E(g(u_n)) \nonumber \\ &&+ \la E'(g(u_{n+1})),\eta
(g(u),g(u_{n+1})) \ra \nonumber \\
&& - \la E'(g(u_n)), \eta (g(u),g(u_n)) \ra
\end{eqnarray}
Using Assumption \ref{assump9.1}, we have
\begin{eqnarray}\label{eq9.10}
\eta (g(u),g(u_n)) = \eta (g(u),g(u_{n+1}))+\eta (g(u_{n+1}),g(u_n) ).
\end{eqnarray}
Combining (\ref{eq9.7}), (\ref{eq9.8}), (\ref{eq9.9}) and (\ref{eq9.10}) , we have
\begin{eqnarray*}
&&B(g(u),g(u_n))-B(g(u),g(u_{n+1}))\\
 & = & E(g(u_{n+1}))-E(g(u_n))-\la E'(g(u_n)),\eta(g(u_{n+1}),g(u_n)) \nonumber \\
&& + \la E'(g(u_{n+1}))-E'(g(u_n)), \eta (g(u),g(u_{n+1})) \ra \\
& \geq & \beta \|\eta(g(u_{n+1}),g(u_n))\|^2 + \la E'(g(u_{n+1}))-E'(g(u_n)), \eta
(g(u),g(u_{n+1})) \ra   \\
& \geq & \beta \|\eta(g(u_{n+1}),g(u_n))\|^2+\la \rho  Tu_{n+1}, \eta (g(u_{n+1}),g(u))
\ra   \\
& \geq & \beta \|\eta(g(u_{n+1}),g(u_n))\|^2 , \q \mbox{using (\ref{eq9.8}).}
\end{eqnarray*}
If $g(u_{n+1}) = g(u_n), $ then clearly $g(u_n)$ is a solution of the
general variational-like inequality (\ref{eq9.3}). Otherwise, it follows that
$B(u,u_n)-B(u,u_{n+1}) $ is nonnegative , and we must have
\begin{eqnarray*}
\lim_{n \rt \infty }\|\eta(g(u_{n+1}),g(u_n))\| = 0.
\end{eqnarray*}
Now using the technique of Zhu and Marcotte \cite{200}, one can easily show
that the entire sequence $\{u_n\}$ converges to the cluster point
$\bar{u}$ satisfying the variational-like inequality (\ref{eq9.3}).
\end{proof}

To implement the proximal method, one has to calculate the solution
implicitly, which is in itself a difficult problem. We again use the
auxiliary principle technique to suggest another iterative method, the
convergence of which requires only the $g\eta $-partially relaxed strongly
general monotonicity.\\

For a given $u \in H: g(u) \in K_{g\eta}, $ satisfying (\ref{eq9.3}), find a solution $w \in H: g(w) \in K_{g\eta} $ such that
\begin{eqnarray}\label{eq9.11}
\la \rho Tu +E^{\prime }(g(w))-E^{\prime }(g(u)), \eta (g(v),g(w)) \ra \geq 0, \quad
\forall v \in H: g(v) \in K_{\eta},
\end{eqnarray}
which is called the
auxiliary general  variational-like inequality, where  $E(u) $ is a  differentiable strongly general preinvex function. It is clear
that,  if $w = u, $ then $w$ is a solution of the general variational-like
inequality (\ref{eq9.3}). This fact allows us to suggest and analyze the
following iterative method for solving (\ref{eq9.3}).

\begin{algorithm}\label{alg9.35}
For a given $u_0 \in H, $ compute the approximate solution $u_{n+1} $ by the iterative scheme
\begin{eqnarray}\label{eq9.12}
\la \rho Tu_n+ E^{\prime }(g(u_{n+1}))-E^{\prime }(g(u_n)), \eta (g(v), g(u_{n+1}))
\ra \geq 0,\\
 \quad \forall v \in H: g(v) \in K_{g\eta}.\nonumber
\end{eqnarray}
\end{algorithm}

 Note that for $\eta (g(v),g(u) ) = g(v)-g(u), $ Algorithm \ref{alg9.35} reduces to:
\begin{algorithm}\label{alg9.36} For a given $u_0 \in H_g, $ find the
approximate solution $u_{n+1} $ by the iterative scheme
\begin{eqnarray*}
\la \rho Tu_n +E^{\prime }(g(u_{n+1}))-E^{\prime }(g(u_n)), g(v)-
g(u_{n+1})\ra \geq 0, \quad \forall v \in H: g(v) \in K_{g\eta},
\end{eqnarray*}
\end{algorithm}
 Algorithm \ref{alg4} for solving general variational inequalities  appears
to be a new one. In a similar way, one can obtain a number of new and
known iterative methods for solving various classes of variational
inequalities and complementarity problems.\\

We now study the convergence analysis of Algorithm \ref{alg9.35}. The analysis is
in the spirit of Theorem \ref{thm9.2}. We only give the main points.\\

\begin{theorem}\label{thm9.3} Let $T$ be a partially relaxed
strongly general $g\eta$- monotone with  a constant $\alpha > 0. $
Let $E$ be a  differentiable strongly general preinvex function with modulus
$\beta $ and Assumption \ref{assump9.1} hold. If $ 0 < \rho < \frac{\beta
}{\alpha ,}$ then the approximate solution $u_{n+1}$
obtained from Algorithm \ref{alg9.35} converges to a solution of (\ref{eq9.3}).
\end{theorem}

\begin{proof} Since the function  $E$ is strongly general preinvex ,
so the solution $u_{n+1}$ of (\ref{eq9.3}) is unique. Let $u \in H: g(u) \in  K_{g\eta}$ be a
solution of the general   variational-like inequality (\ref{eq9.3}).  Then
\begin{eqnarray*}
\la Tu, \eta (g(v),g(u)) \ra  \geq 0, \quad \forall v\in H: g(v)\in K_{\eta}.
\end{eqnarray*}
Taking $v = u_{n+1} $ in the above inequality, we have
\begin{eqnarray}\label{eq9.13}
\la \rho Tu,\eta (g(u_{n+1}),g(u) ) \ra \geq 0.
\end{eqnarray}
Combining (\ref{eq9.9}), (\ref{eq9.12}) and (\ref{eq9.13}), we have
\begin{eqnarray*}
&&B(g(u),g(u_n))-B(g(u),g(u_{n+1}))\\ & = & E(g(u_{n+1}))-E(g(u_n))-\la E'(g(u_n)),\eta (g(u_{(n+1}),g(u_n)) \\
&& + \la E'(g(u_{n+1}))-E'(g(u_n)), \eta (g(u),g(u_{n+1})) \ra  \\
& \geq & \beta \|\eta(g(u_{n+1}),g(u_n))\|^2 + \la E'(g(u_{n+1}))-E'(g(u_n)), \eta(g(u),g(u_{n+1})) \ra   \\
& \geq & \beta \|\eta(g(u_{n+1}),g(u_n))\|^2+\la \rho  Tu_n, \eta (g(u_{n+1}),g(u))  \ra,\\
& \geq & \beta \|\eta(g(u_{n+1}),g(u_n)\|^2+\la \rho  Tu_n-\rho Tu, \eta (g(u_{n+1}),g(u))\ra,   \\
& \geq & ( \beta -\rho \alpha )\|\eta(g(u_{n+1}),g(u_n))\|^2 .
\end{eqnarray*}
If $g(u_{n+1}) = g(u_n), $ then clearly $g(u_n)$ is a solution of the general
variational-like inequality (\ref{eq9.3}). Otherwise, the assumption $0 < \rho <
\frac{\alpha }{\beta }, $ implies that the sequence $B(g(u),g(u_n))-B(g(u),
g(u_{n+1}))$  is nonnegative , and we must have
\begin{eqnarray*}
\lim _{n \longrightarrow \infty }\|\eta(g(u_{n+1}),g(u_n) \| = 0.
\end{eqnarray*}
Now by using the technique of Zhu and Marcotte [164], it can be shown that
the entire sequence $\{u_n\} $ converges to the cluster point  $\bar{u} $
satisfying the variational-like inequality (\ref{eq9.3}).
\end{proof}

We now show that the solution of the auxiliary general variational-like inequality (\ref{eq9.11})
 is the minimum of the functional $I[g(w)] $ on the general invex set $K_{g\eta},$ where
\begin{eqnarray}\label{eq9.14}
I[g(w)] &=& E(g(w))-E(g(u)- \la E^{\prime }(g(u))-\rho Tu, \eta (g(w),g(u)) \ra \nonumber \\
&=& B(g(w),g(u))-\rho \la Tu, \eta (g(w),g(u)) )\ra,
\end{eqnarray}
is known as the auxiliary energy functional associated with the auxiliary
general variational-like inequality (\ref{eq9.11}), where $B(g(w),g(u)) $ is a general Bregman
function. We now prove that the minimum of the functional $I[w], $
defined by (\ref{eq9.14}), can be characterized by the general  variational-like
inequality (\ref{eq9.11}). \\

\begin{theorem}\label{thm9.3} Let $ E $ be a differentiable general  preinvex
function. If Assumption \ref{assump9.1} holds and $\eta (.,.)$ is prelinear in the
first argument, then the
minimum of $I[w],$ defined by (\ref{eq9.14}), can be characterized by the
auxiliary general variational-like inequality (\ref{eq9.3}).
\end{theorem}

\begin{proof} Let $w \in H: g(w) \in K_{\eta}$ be the minimum of $I[w]$ on $K_{g\eta}. $ Then
\begin{eqnarray*}
I[g(w)] \leq I[g(v)], \quad \forall v \in H: g(v) \in K_{g\eta}.
\end{eqnarray*}
Since $K_{g\eta} $ is a general  invex set, so for all $g(w),g(u) \in K_{g\eta}, t \in [0,1],  g(v_t) =
g(w)+t\eta (g(v),g(w))\in K_{g\eta}.$  Replacing $g(v) $ by $g(v_t) $ in the above inequality, we
have
\begin{eqnarray}\label{eq9.15}
I[g(w)] \leq I[g(w)+t \eta (g(v),g(w))].
\end{eqnarray}
Since $\eta (.,.)$ is prelinear in the first argument, so,  from (\ref{eq9.12}) and
(\ref{eq9.15}), we have
\begin{eqnarray*}
E(g(w))-E(g(u))&-&\la E'(g(u)) - \rho Tu,\eta (g(w),g(u)) \ra \\ & \leq &E(g(v_t))-E(g(u)) -\la
E'(g(u))-\rho Tu,\eta (g(v_t),g(u))\ra  \\
& \leq & E(g(v_t))-(1-t)\la E'(g(u))-\rho Tu,\eta (g(w),g(u)) \ra   \\
& &-t\la E'(g(u))-\rho Tu, \eta (g(w),g(u)) \ra,
\end{eqnarray*}
which implies that
\begin{eqnarray}\label{eq9.16}
E(g(w)+t\eta (g(v),g(w)))-E(g(w)) &\geq & t\la E'(g(u))-\rho Tu,\eta(g(v),g(w))\ra \nonumber \\
&&-t\la E'(g(u))-\rho Tu,\eta (g(w),g(u))\ra .
\end{eqnarray}
Now using Assumption \ref{assump9.1}, we have
\begin{eqnarray}
 \la E'(g(u)),\eta (g(v),g(u))\ra &=& \la E'(g(u)),\eta (g(v),g(w))\ra \nonumber \\&&+\la E'(g(u)), \eta
(g(w),g(u))\ra  \label{eq9.17}\\
\la Tu, \eta  (g(v),g(u))\ra & = &\la Tu,\eta (g(v),g(w))\ra + \la Tu, \eta (g(w),g(u))\ra.\label{9.18}
\end{eqnarray}
From (\ref{eq9.15}), (\ref{eq9.16}), (\ref{eq9.17}) and (\ref{9.18}),  we obtain
\begin{eqnarray*}
E(g(w)+t\eta (g(v),g(w)))-E(g(w)) \geq  t\la E'(g(u))-\rho Tu, \eta (g(v),g(w))\ra.
\end{eqnarray*}
Dividing both sides by $t$ and letting $t \rightarrow 0,$ we have
\begin{eqnarray*}
\la E'(g(w)), \eta (g(v),g(w))\ra \geq \la E'(g(u))-\rho Tu, \eta (g(v),g(w))\ra,
\end{eqnarray*}
the required inequality (\ref{eq9.11}).\\

Conversely, let $u \in H: g(u) \in K_{\eta} $ be a solution of (\ref{eq9.11}). Then
\begin{eqnarray*}
I[g(w)]-I[g(v)] &  =  & E(g(w))-E(g(v))-\la E'(g(u))-\rho Tu,\eta (g(w),g(u))\ra  \\
& &  + \la E'(g(u))-\rho Tu,\eta (g(v),g(u))\ra \\
  & \leq & -\la E'(g(w)),\eta (g(v),g(w))\ra \\ &&+ \la E'(g(u)), \eta (g(v),g(u))-\eta (g(w),g(u))\ra \\
& & -\rho \la Tu,\eta (g(v),g(u))- \eta (g(w),g(u))\ra  \\
& \leq & \la E'(g(u)),\eta (g(v),g(w))\ra -\la E'(g(w))-\rho Tu, \eta (g(v),g(w))\ra \\
& &  + \la E'(g(w))+ \rho Tu,\eta (g(v),g(w))\ra -\la E'(g(u)), \eta (g(v),g(w))\ra \\
& \leq & 0.
\end{eqnarray*}
Thus it follows that $I[g(w)] \leq I[g(v)],$ showing that $g(v) \in K_{\eta} $ is the
minimum of the functional $I[g(w)]$ on $K_{\eta} $, the required result.
\end{proof}

\section{Higher order strongly general  convex functions}

We would like to point out that the strongly convex functions were introduced and studied by Polyak \cite{158}, which
play an important part in the optimization theory and related areas. For example, Karmardian\cite{57} used
the strongly convex functions to discuss the unique existence of a solution of the nonlinear
complementarity problems. Strongly convex functions also played important role in the convergence analysis of
the iterative methods for solving variational inequalities and equilibrium
problems, see Zu and Marcotte \cite{200}.  Lin and Fukushima \cite{65} introduced the concept of higher order strongly convex functions and used it in the study of mathematical program with equilibrium constraints. These mathematical programs with equilibrium constraints are defined by a parametric variational inequality or complementarity system and play an important role in many fields such as engineering
design, economic equilibrium and multilevel game. These facts and observations
inspired  Mohsen et al\cite{75}to  consider higher order strongly convex function involving an arbitrary bifunction.
Noor and Noor \cite {139,139a} have introduced the higher order strongly general convex functions, which include the higher order strongly convex functions \cite{65,75} as special cases. \\ \\
 In this chapter,  we introduce concept of higher order strongly general convex functions.
 Several new concepts of monotonicity are introduced. Our results represent the refinement and improvement of
 the results of Lin and Fukushima \cite{65}. Higher order strongly general convex functions are used to obtain new  characterizations of the uniformly reflex Banach spaces by the parallelogram laws.  It is worth mentioning that the parallelogram laws have been  discussed in \cite{21,22,23,24,189}. \\

We now define the concept of higher order strongly general convex functions, which have been investigated  in \cite {139,139a}.
\begin{definition}\label{def10.1}\quad    A function $F$ on the convex set $K$ is said to be higher order
strongly general  convex with respect to the  function $g,$ if there exists a constant $\mu>0 $   such that
\begin{eqnarray*}
&&F(g(u)+t(g(v)-g(u)))\leq (1-t)F(g(u()+tF(g(v)) \\ &&-\mu \{t^p(1-t)+t(1-t)^p\}\|g(v)-g(u)\|^p,
 \quad \forall g(u),g(v)\in K_g, t\in[0,1], p>1.
\end{eqnarray*}
\end{definition}

A function $F$ is said to higher order strongly general concave, if and
only if, $-F$ is higher order strongly general convex.  \\
If $t=\frac{1}{2}, $  then
\begin{eqnarray*}
F\bigg(\frac{g(u(+g(v)}{2}\bigg)\leq  \frac{F(g(u))+F(g(v))}{2}- \mu\frac{1}{2^p}\| g(v)-g(u) \|^p,
 \forall g(u),g(v)\in K_g, p>1.
\end{eqnarray*}
The function $F$ is said to be higher order strongly general  $J$-convex function.\\ \\
We now discuss some special cases. \\ \\
\noindent{\bf I. }If $ p=2,$  then the higher order
strongly convex function becomes strongly convex functions, that
is,
\begin{eqnarray*}
F(g(u)+t(g(v)-g(u)))&\leq & (1-t)F(g(u))+tF(g(v))-\mu t(1-t)\| g(v)-g(u)\|^2,\\ && \forall g(u),g(v)\in K_g,
t\in[0,1].
\end{eqnarray*}

For the properties of the  strongly convex functions in variational
inequalities and equilibrium problems, see Noor \cite{95,122,129}.\\

\noindent{II.} If $g =  I,$  then Definition \ref{def10.1} reduces to
\begin{definition}\label{def2.1}\quad    A function $F$ on the convex set $K$ is said to be higher order
strongly convex,  if there exists a constant $\mu>0, $ such that
\begin{eqnarray*}
F(u+t(v-u))\leq (1-t)F(u)+tF(v)-\mu \{t^p(1-t)+t(1-t)^p\}\|v-u\|^p, p>1,
\\ \forall u,v\in K, t\in[0,1],
\end{eqnarray*}
\end{definition}
which  appears to be a new one.\\

For appropriate and suitable choice of the unction $g$ and $p, $ one can obtain various new and known classes of strongly convex functions.
This show that the higher order strongly convex functions involving the function $g $ is quite general and unifying one. One can explore the applications of the higher order strongly general convex functions, this is another direction of further research.\\

\begin{definition}\label{def10.2}\quad A  function $F$ on the convex set $K$ is said to be a higher order
strongly affine  general convex with respect to the function $g, $ if there exists a constant $\mu>0, $
such that
\begin{eqnarray*}\label{eq3}
F(g(u)&+&t(g(v)-g(u)))\leq  (1-t)F(g(u))+tF(g(v) \\ &&-\mu \{t^p(1-t)+t(1-t)^p\}\|g(v)-g(u)\|^p,
\forall g(u),g(v)\in K_g, t\in[0,1], p>1.
\end{eqnarray*}
\end{definition}
Note that if a functions is both higher order strongly convex and higher order strongly concave, then it is higher order strongly affine convex function.\\
\begin{definition}\label{def10.3}  A function $F$ is  called higher order strongly quadratic equation with respect to the function $g,$ if there exists a constant $\mu>0, $ such that
\begin{eqnarray*}\label{eq4}
F\bigg(\frac{g(u)+g(v)}{2}\bigg)&=& \frac{F(g(u))+F(g(v))}{2}-  \\
&&\mu \frac{1}{2^p}\|g(v)-g(u)\|^p,  \forall g(u),g(v)\in K_g, t\in[0,1], p>1.
\end{eqnarray*}
\end{definition}
This  function $F$ is also called  higher order  strongly affine general  $J$-convex function. \\

\begin{definition}\label{def10.4}
\quad A function $F$ on the convex  set $K$ is said to be higher order
strongly quasi convex,  if there exists  a constant $\mu>0$ such that
\begin{eqnarray*}
F(g(u)&+&t(g(v)-g(u))\leq \max\{F(g(u)),F(g(v))\} \\
  &&-\mu \{t^p(1-t)+t(1-t)^p\}\|g(v)-g(u)\|^p, \forall g(u),g(v)\in K_g, t\in[0,1], p>1.
\end{eqnarray*}
\end{definition}

\begin{definition}\label{def10.5} \quad A function $F$ on the convex  set $K$ is said to be higher order
strongly log-convex,  if there exists a constant $\mu>0$ such that
\begin{eqnarray*}
F(g(u)&+&t(g(v)-g(u))\leq (F(g(u)))^{1-t}(F(g(v)))^t \\&-&\mu \{t^p(1-t)+t(1-t)^p\}\|g(v)-g(u)\|^p, \forall g(u),g(v)\in K_g, t\in[0,1], p>1,
\end{eqnarray*}
where $F(\cdot)>0.$
\end{definition}

From the above definitions, we have
\begin{eqnarray*}
F(g(u)&+&t(g(v)-g(u))\leq (F(g(u)))^{1-t}(F(g(v)))^t\\ &&-\mu \{t^p(1-t)+t(1-t)^p\}\|g(v-g(u)\|^p \\
&\leq& (1-t)F(g(u))+tF(g(v))-\mu  \{t^p(1-t)+t(1-t)^p\}\|g(v)-g(u)\|^p\\
&\leq& \max\{F(g(u)),F(g(v))\}-\mu \{t^p(1-t)+t(1-t)^p\}\|g(v)-g(u)\|^p, p>1.
\end{eqnarray*}
This shows that every higher order  strongly general log-convex function is a higher order  strongly general convex function
and every higher order strongly general convex function is a higher order strongly general quasi-convex function. However, the converse is not true. \\
For appropriate and suitable choice of the arbitrary bifunction $g, $  one can obtain several new and known classes of strongly convex functions and their variant forms as special cases of generalized strongly convex functions. This shows that the class of higher order strongly general convex functions is quite broad and unifying one.

\begin{definition}\label{def10.6} \quad An operator $T:K\rightarrow H$ is said to be:\\
  \noindent{\bf(i).} higher order   strongly  monotone, if and only if, there exists a
  constant $\alpha>0$ such that
  \begin{eqnarray*}
\langle Tu-Tv, g(u)-g(v)v\rangle\geq  \alpha\|g(u)-g(v)\|^p, \forall g(u),g(v)\in K_g.
\end{eqnarray*}

  \noindent{\bf (ii).} higher order   strongly pseudomonotone, if and only if, there exists  a
  constant $\nu> 0$ such that
  \begin{eqnarray*}
&&\langle Tu,g(v)-g(u)\rangle+\nu \|g(v)-g(u)\|^p \geq 0  \\
 && \Rightarrow \\
&&\langle Tv,g(v)-g(u)\rangle\geq 0,\forall g(u),g(v)\in K_g.
\end{eqnarray*}
  \noindent{\bf (iii). } higher order  strongly relaxed pseudomonotone, if and only if, there exists a
  constant $\mu> 0$ such that
  \begin{eqnarray*}
&&\langle Tu, g(v)-g(u)\rangle\geq 0 \\
&& \Rightarrow  \\
&&  -\langle  Tv, g(u)-g(v)\rangle+\mu \|\xi(v,u)\|^p\geq 0,\forall g(u),g(v)\in K_g.
\end{eqnarray*}
\end{definition}

\begin{definition}\label{eq10.7}  A differentiable function $F$ on the convex set $K_g$ is said to be
higher order  strongly pseudoconvex function, if and only if, if there exists
a constant $\mu>0$ such that
\begin{eqnarray*}
 \langle F'(u),g(v)-g(u)\rangle+\mu \|g(v)-g(u)\|^p\geq 0\Rightarrow F(v)\geq F(u),
  \forall g(u),g(v)\in K_g.
\end{eqnarray*}
\end{definition}

We now  consider some basic properties of higher order strongly general convex functions.\\

\begin{theorem}\label{theorem10.1} \quad Let $F$ be a differentiable function on the convex set $K_g .$
 Then the function $F$ is higher order  strongly general convex function,  if and only if,
\begin{eqnarray}\label{eq10.12}
F(g(v))-F(g(u))& \geq& \langle F^{\prime}(g(u)), g(v)-g(u) \rangle +\mu \|g(v)-g(u)\|^p ,  \\ && \forall g(v),g(u)\in K_g. \nonumber
\end{eqnarray}
\end{theorem}
\begin{proof} \quad Let $F$ be a higher order  strongly general convex function on the convex set
$K_g.$ Then
\begin{eqnarray*}
&&F(g(u)+t(g(v)-g(u))\leq (1-t)F(g(u))+tF(g(v))\\ &&-\mu \{t^p(1-t)+t(1-t)^p\}\|g(v)-g(u) \|^p,\quad\forall g(u),g(v) \in K_g
\end{eqnarray*}
which can be written as
\begin{eqnarray*}
F(g(v))-F(g(u))\geq \{\frac{F(g(u)+t(g(v)-g(u))-F(g(u))}{t}\}\\
+ \{t^{p-1}(1-t)+(1-t)^{p}\}\|g(v)-g(u)\|^p.
\end{eqnarray*}
Taking the limit in the above inequality as $t\rightarrow 0, $, we have
\begin{eqnarray*}
F(g(v))-F(g(u))\geq \langle F'(g(u)),g(v)-g(u))\rangle+ \mu \|g(v)-g(u)\|^p,\forall g(u),g(v)\in K_g.
\end{eqnarray*}
which is (\ref{eq10.12}), the required result. \\

Conversely, let (\ref{eq10.12}) hold.  Then,  $\forall g(u),g(v)\in K_g, t\in [0,1],$\\
$g(v_t)=g( )u+t(g(v)-g(u))\in K_g, $  we have
\begin{eqnarray}\label{eq10.13}
F(g(v))&-&F(g(v_t)) \geq \langle
F'(g(v_t)),g(v)-g(v_t))\rangle+\mu \|g(v)-g(v_t)\|^p\nonumber \\
&=&(1-t)\langle F'(g(v_t)),g(v)-g(u)\rangle+ \mu(1-t)^p\|g(v)-g(u)\|^p.
\end{eqnarray}
In a similar way, we have
\begin{eqnarray}\label{eq10.14}
F(g(u))-F(g(v_t))&\geq& \langle F'(g(v_t)),g(u)-g(v_t))\rangle+\mu \|g(u)-g(v_t)\|^p\nonumber \\
&=&-t\langle F'(g(v_t)),g(v)-g(u)\rangle+ \mu t^p\|g(v)-g(u)\|^p.
\end{eqnarray}
Multiplying (\ref{eq10.13}) by $t$ and (\ref{eq10.14}) by $(1-t)$ and adding
the resultants, we have
\begin{eqnarray*}
F(g(u)&+&t(g(v)-g(u))\leq(1-t)F(g(u))+tF(g(v))\\ &&- \mu \{t^p(1-t)+t(1-t)^p\}\|g(v)-g(u)\|^p, \forall  g(u),g(v) \in K_g,
\end{eqnarray*}
showing that $F$ is a higher order  strongly general convex function.
\end{proof}

\begin{theorem}\label{theorem10.2} \quad Let $F$ be a differentiable higher order strongly convex function on the convex set $K_g. $  Then
\begin{eqnarray}\label{eq10.15}
\langle F'(g(u))-F'(g(v)),g(u)-g(v)\rangle \geq 2\mu \{\|g(v)-g(u)\|^p, \forall g(u),g(v)\in K_g.
\end{eqnarray}
\end{theorem}
\small{
\begin{proof}
Let $F$ be a higher order strongly general convex function on the convex set $K_g.$ Then,
from Theorem \ref{theorem10.1}. we have
\begin{eqnarray}\label{eq10.16}
F(g(v))-F(g(u))\geq \langle F'(g(u)),g(v)-g(u)\rangle+ \mu \|g(v)-g(u)\|^p,  \forall g(u),g(v) \in K_g.
\end{eqnarray}
Changing the role of $u$ and $v$ in (\ref{eq10.16}), we have
\begin{eqnarray}\label{eq10.17}
F(g(u))-F(g(v))\geq \langle F'(g(v)),g(u)-g(v))\rangle+ \mu \|g(v)-g(u))\|^p, \forall g(u),g(v) \in K_g.
\end{eqnarray}
Adding (\ref{eq10.16}) and (\ref{eq10.17}), we have
\begin{eqnarray}\label{eq10.18}
\langle F'(g(u))-F'(g(v)),g(u)-g(v)\rangle \geq 2\mu \{\|g(v)-g(u)\|^p, \forall g(u),g(v)\in K_g.
\end{eqnarray}
which shows that $F'(.)$ is a higher order strongly  general monotone operator.\\
\end{proof}
}
We  remark that the converse of Theorem \ref{theorem10.2} is not true. In this direction, we have the following result. \\
\begin{theorem}\label{theorem10.3} If the differential operator $F^{\prime}(.) $ of a differentiable higher order strongly general convex function $F$ is higher order strongly monotone operator, then
\begin{eqnarray}\label{eq10.19}
F(g(v))-F(g(u)) \geq \langle F'(g(u)),g(v)-g(u)\rangle +2\mu \frac{1}{p}\|g)v)-g(u)\|^p, \forall g(u),g(v) \in K_g.
\end{eqnarray}
\end{theorem}
\begin{proof}
Let $F'$ be a higher order  strongly monotone operator. Then, from (\ref{eq10.18}), we have
\begin{eqnarray}\label{eq10.20}
\langle F'(g(v)),g(u)-g(v)\rangle \geq  \langle F'(g(u)),g(u)-g(v))\rangle +2\mu\|g(v)-g(u)\|^p, \forall g(u),g(v) \in K_g.
\end{eqnarray}
Since $K$ is an convex set, $\forall g(u),g(v) \in K_g$, $t\in [0,1], $
$g(v_t)=g(u)+t(g(v)-g(u))\in K_g.$ Taking $g(v)= g(v_t)$ in (\ref{eq10.20}),  we have
\begin{eqnarray*}
\langle F'(g(v_t)),g(u)-g(v_t)\rangle &\leq& \langle
F'(g(u)), g(u)-g(v_t)\rangle-2\mu \|g(v_t)=g(u)\|^p, \nonumber \\
&=&-t \langle F'(g(u)),g(v)-g(u)\rangle-2\mu t^p \|g(v)-g(u)\|^p,
\end{eqnarray*}
which implies that
\begin{eqnarray}\label{eq10.21}
\langle F'(g(v_t)),g(v)-g(u)\rangle\geq \langle F'(g(u)),g(v)-g(u)\rangle +2 \mu t^{p-1} \|g(v)-g(u)\|^p.
\end{eqnarray}
Consider the auxiliary function
\begin{eqnarray} \label{eq10.22}
\zeta (t)=F(g(u)+t(g(v)-g(u)), \forall g(u),g(v) \in K_g,
\end{eqnarray}
from which, we have
\begin{eqnarray*}
\zeta (1)= F(g(v)), \quad \zeta(0)= F(g(u)).
\end{eqnarray*}
Then, from (\ref{eq10.22}), we have
\begin{eqnarray}\label{eq10.23}
\zeta'(t)=\langle F'(g(v_t), g(v)-g(u)\rangle\geq
\langle F'(g(u)),g(v)-g(u)\rangle +2\mu t^{p-1} \|g(v)-g(u)\|^p.
\end{eqnarray}
Integrating (\ref{eq10.23}) between 0 and 1, we have
\begin{eqnarray*}
\zeta (1)-\zeta(0 )&=& \int^{1}_{0}\zeta^{\prime}(t)dt \\
 &&  \geq \langle F'(g(u)),g(v)-g(u)\rangle +2\mu \frac{1}{p}\|g(v)-g(u)\|^p.
\end{eqnarray*}
Thus it follows that
\begin{eqnarray*}
F(g(v))-F(g(u)) \geq \langle F'(g(u)),g(v)-g(u)\rangle +2\mu \frac{1}{p}\|g(v)-g(u)\|^p,\forall g(u),g(v)\in K_g,
\end{eqnarray*}
 which is the required (\ref{eq10.19}).
\end{proof}

We note that,  if $p=2, $ then Theorem \ref{theorem10.3} can be viewed as the converse of Theorem \ref{theorem10.2}.\\ \\
We now give a necessary condition for higher order  strongly general pseudoconvex function.

\begin{theorem}\label{theorem10.4} \quad Let $F'(.)$ be a higher order  strongly relaxed pseudomonotone operator. Then $F$ is a higher order  strongly
pseudo-convex function.
\end{theorem}
\begin{proof} \quad Let $F'(.)$ be a higher order  strongly relaxed general  pseudomonotone operator. Then, from(\ref{eq10.12}), we have
\begin{eqnarray*}
\langle F'(g(u)),g(v)-g(u)\rangle \geq0, \forall g(u),g(v)\in K_g,
\end{eqnarray*}
implies that
\begin{eqnarray} \label{eq10.24}
\langle F'(g(v)),g(v)-g(u) \rangle \geq  \mu \|g(v)-(u)\|^p, \forall g(u),g(v)\in K_g.
\end{eqnarray}
Since $K_g$ is an convex set, $\forall g(u),g(v)\in K_g, \quad t \in [0,1],$
$g(v_t)=g(u)+t(g(v)-g(u))\in K.$ \\ Taking $g(v)=g( v_t)$ in (\ref{eq10.24}), we have
\begin{eqnarray} \label{eq10.25}
\langle F'(g(v_t)), g(v)-g(u)\rangle \geq  \mu t^{p-1}\|g(v)-g(u)\|^p.
\end{eqnarray}
Consider the auxiliary function
\begin{eqnarray}\label{eq10.25}
\zeta(t)=F(g(u)+t(g(v)-g(u)))= F(v_t),\quad \forall g(u),g(v)\in K_g, t\in [0,1],
\end{eqnarray}
which is differentiable,  since $F$ is differentiable function.
Then, using (\ref{eq10.25}), we have
\begin{eqnarray*}
\zeta '(t)=\langle F'(g(v_t)),g(v)-g(u))\rangle \geq  \mu t^{p-1}\|g(v)-g(u)\|^p.
\end{eqnarray*}
Integrating the above relation between 0 to 1, we have
\begin{eqnarray*}
\zeta(1)-\zeta(0)= \int^{1}_{0} \zeta^{\prime}(t)dt \geq  \frac{\mu}{p}\|g(v)-g(u)\|^p,
\end{eqnarray*}
that is,
\begin{eqnarray*}
F(g(v))-F(g(u))\geq \frac{\mu}{p}\|g(v)-g(u)\|^p), \forall g(u),g(v)\in K_g,
\end{eqnarray*}
showing that $F$ is a higher order  strongly general  pseudoconvex function.
\end{proof}

\begin{definition}\quad A  function $F$ is said to be sharply higher order  strongly general pseudoconvex, if there exists a constant $\mu>0$ such that
\begin{eqnarray*}
 &&\langle F'(g(u)),g(v)-g(u)\rangle\geq 0 \\
 &&\Rightarrow    \\
 &&F(g(v))\geq F(g(v)+t(g(u)-g(v))) +\mu \{t^p(1-t)+t(1-t)^p\}\|g(v)-g(u)\|^p, \forall g(u),g(v)\in K_g.
\end{eqnarray*}
\end{definition}

\begin{theorem} \quad Let $F$ be a sharply higher order  strongly general pseudoconvex function on the general convex set
$K_g$ with a constant $\mu >0.$ Then
\begin{eqnarray*}
 \langle F'(g(v)),g(v)-g(u)\rangle\geq \mu \|g(v)-g(u)\|^p, \forall g(u),g(v)\in K_g.
\end{eqnarray*}
\end{theorem}
\begin{proof} \quad Let $F$ be a sharply higher order  strongly general  pseudoconvex function on the general convex set
$K_g$. Then
\begin{eqnarray*}
F(g(v))\geq F(g(v)+t(g(u)-g(v)))+ \mu \{t^p(1-t)+t(1-t)^p\}\|g(v)-g(u)\|^p,\\  \forall g(u),g(v)\in K_g, t\in [0,1],
\end{eqnarray*}
from which,  we have
\begin{eqnarray*}
\{\frac{F(g(v)+t(g(u)-g(v))-F(g(v))}{t}\}+\mu \{t^{p-1}(1-t)+(1-t)^{p}\}\|g(v)-g(u)\|^p\geq 0.
\end{eqnarray*}
Taking limit in the above inequality, as $t \rightarrow 0,$ we have
\begin{eqnarray*}
 \langle F'(g(v)),g(v)-g(u)\rangle\geq \mu \|g(v)-g(u)\|^p, \forall g(u),g(v)\in K_g,
\end{eqnarray*}
 the required result.
\end{proof}

\begin{definition}\quad A function $F$ is said to be a pseudoconvex function with respect to a strictly positive bifunction $B(.,.),$ such that
\begin{eqnarray*}
&&F(v) <  F(u)  \\
&&\Rightarrow \\
&&F(u + t~l(v, u)) < F(u) + t(t - 1)B(v, u),  \forall u,v\in K, t\in [0,1].
\end{eqnarray*}
\end{definition}

\begin{theorem} \quad If the function $F$  is higher order strongly convex function such that $F(g(v)) < F(g(u)),$
then the  function $F$  is higher order  strongly pseudoconvex.
\end{theorem}
\begin{proof}\quad   Since $F(g(v))<F(g(u))$  and $F$ is higher order  strongly convex function,  then \\$\forall g(u),g(v)\in K_g, \quad t\in [0,1],  $  we have
\begin{eqnarray*}
F(g(u) + t(g(v)- u))& \leq & F(g(u)) + t(F(g(v)) - F(g(u)))-\mu \{t^p(1-t)+t(1-t)^p\}\|g(v)-g(u)\|^p \\
& < & F(g(u)) + t(1 - t)(F(g(v)) - F(gu)))-\mu \{t^p(1-t)+t(1-t)^p\}\|g(v)-g(u)\|^p \\
&= & F(g(u)) + t(t - 1)(F(g(u)) - F(g(v)))-\mu \{t^p(1-t)+t(1-t)^p\}\|g(v)-g(u)\|^p \\
& < & F(g(u)) + t(t - 1)B(g(u), g(v)) -\mu \{t^p(1-t)+t(1-t)^p\}\|g(v)-g(u)\|^p,
\end{eqnarray*}
where $B(g(u), g(v)) = F(g(u)) - F(g(v)) > 0,$ the required result.
\end{proof}

We now discuss the optimality for the differentiable  strongly general convex functions, which is the main motivation of our next result.\\
\begin{theorem} \quad Let $F $ be a differentiable higher order strongly general convex function with modulus $\mu > 0. $ If $u \in H: g(u)\in K_g $ is the minimum of the function $F, $ then
\begin{eqnarray}\label{eq10.26}
F(g(v))-F(g(u)) \geq \mu \|g(v)-g(u)\|^p, \quad \forall g(u),g(v)\in K_g.
\end{eqnarray}
\end{theorem}
\begin{proof}
\quad Let $u \in H: g(u)\in K_g $ be a minimum of the function $F. $ Then
\begin{eqnarray}\label{eq10.27}
F(u)\leq F(v), \forall v \in H: g(v) \in K_g.
\end{eqnarray}
Since $K $ is a general convex set, so, $\forall g(u),g(v)\in K_g, \quad t\in [0,1], $
$$ g(v_t)= (1-t)g(u)+ tg(v) \in K_g.$$
Taking $ g(v)= g(v_t) $ in (\ref{eq10.27}), we have
\begin{eqnarray} \label{eq10.28}
0 \leq  \lim_{t \rightarrow 0}\{\frac{F(g(u)+t(g(v)-g(u)))-F(g(u))}{t}\} = \langle F^{\prime}(g(u)), g(v)- g(u)\rangle.
\end{eqnarray}
Since $F $ is differentiable higher order  strongly general convex function, so
\begin{eqnarray*}
F(g(u)+t(g(v)-g(u))) &\leq & F(g(u))+ t(F(g(v))-F(g(u))) \\ &&-\mu \{t^p(1-t)+t(1-t)^p\}\|g(v)-g(u)\|^p,
 \forall g(u),g(v)\in K_g,
\end{eqnarray*}
from which, using(\ref{eq10.28}), we have
\begin{eqnarray*}
F(g(v))-F(g(u)) &\geq & \lim_{t \rightarrow 0}\{\frac{F(g(u)+t(g(v)-g(u)))-F(g(u))}{t}\} + \mu \|g(v)-g(u)\|^p\\
&=& \langle F^{\prime}(g(u)), g(v)-g(u) \rangle + \mu \|g(v)-g(u)\|^p, \\
\end{eqnarray*}
the required result (\ref{eq10.26}).
\end{proof}
\noindent{\bf Remark: } We would like to mention that, if $u \in H: g(u) \in K_g $ satisfies the inequality
 \begin{eqnarray}\label{10.14}
 \langle F^{\prime}(g(u)), g(v)-g(u) \rangle + \mu \|g(v)-g(u)\|^p \geq 0, \quad \forall u,g(v) \in K_g,
 \end{eqnarray}
 then $u \in K_g $ is the minimum of the  function $F. $ The inequality of the type (\ref{10.14}) is called the higher order  general variational inequality.

 \begin{theorem} \quad Let $f$  be a higher order  strongly affine general convex  function.  Then $F$ is a  higher order  strongly general  convex function, if and only if,    $H= F-f $ is a general convex function.
 \end{theorem}
 \begin{proof}\quad Let $f $ be a higher order strongly affine general convex  function,  Then
 \small{
 \begin{eqnarray}\label{3.14}
 f((1-t)g(u)+tg(v)) &=&  (1-t)f(g(u))+tf(g(v)) \nonumber \\ &&- \mu \{t^{p}(1-t)+t(1-t)^p\}\|g(v)-g(u)\|^p,
 \forall g(u),g(v) \in K_g.
 \end{eqnarray}
 From the higher order  strongly general convexity of $F, $ we have
 \begin{eqnarray}\label{3.15}
 F((1-t)g(u)+tg(v)) &\leq & (1-t)F(g(u))+tF(g(v))\nonumber \\ &&-\mu \{t^{p}(1-t)+t(1-t)^p\}\|g(v)-g(u)\|^p,
 \forall g(u),g(v)\in K_g,
 \end{eqnarray}
 }
  From (\ref{3.14} ) and (\ref{3.15}), we have
  \begin{eqnarray}\label{3.16}
  F((1-t)g(u)+tg(v))-f((1-t)f(g(u))+tf(g(v)) \leq  (1-t)(F(g(u))-f(g(u)))\nonumber \\
  +t (F(g(v))-f(g(v))),
  \end{eqnarray}
  from which it follows that
  \begin{eqnarray*}
  H((1-t)g(u)+tg(v))&=&F((1-t)g(u)+tg())-f((1-t)g(u)+tg(v))\\
  &\leq & (1-t)F(g(u))+tF(g(v))-(1-t)f(g(u))-tf(g(v)) \\
   &= & (1-t)(F(g(u))-f(g(u)))+t (F(g(v))-f(g(v))),
  \end{eqnarray*}
  which show that $H= F-f $ is a convex function.The inverse implication is obvious.
  \end{proof}

   It is worth mentioning that the higher order strongly convex function  is also higher order strongly Wright  general convex   function. From the definition, we have
\begin{eqnarray*}\label{eq1}
F(g(u)+t(g(v)-g(u)))&+ & F(g(v)+t(g(u)-g(v)))\leq F(g(u))+F(g(v))\\
&-&2\mu \{t^p(1-t)+t(1-t)^p\}\|g(v)-g(u)\|^p,\forall g(u),g(v)\in K_g, t\in[0,1].
\end{eqnarray*}
which is called the higher order strongly Wright  general convex  function. One can studies the properties and applications of the Wright higher order strongly convex functions in optimization operations research.\\

 Bynum \cite{21} and Chen et al \cite{22,23,24,25} have studied the properties and applications of the parallelogram laws for the Banach spaces.
 Xi \cite{189} obtained  new characteristics of $p$-uniform convexity and $q$-uniform smoothness of a Banach space  via the functionals $\|.\|^p $ and $\|.\|^q,  $ respectively. These results can be obtained  from the concepts of higher order strongly general convex(concave) functions, which can be viewed as novel application.
 Setting  $F(u)= \|u\|^p$  in Definition \ref{def10.2}, we have
 \begin{eqnarray}\label{eq4.1}
\|g(u)&+&t(g(v)-g(u))\|^p \leq (1-t)\|g(u)\|^p+t\|g(v)\|^p \nonumber \\
&-&\mu \{t^p(1-t)+t(1-t)^p\}\|g(v)-g(u)\|^p, \forall g(u),g(v)\in K_g, t\in[0,1].
\end{eqnarray}
Taking $t=\frac{1}{2} $ in (\ref{eq4.1}), we have
\begin{eqnarray}\label{eq4.2}
\|\frac{g(u)+g(v)}{2}\|^p+\mu \frac{1}{2^p}\|g(v)-g(u)\|^p\leq \frac{1}{2}\|g(u)\|^p+\frac{1}{2}\|g(v)\|^p,\forall g(u),g(v)\in K_g,
\end{eqnarray}
which implies that
\begin{eqnarray}\label{eq4.2}
\|g(u)+g(v)\|^p+\mu \|g(v)-g(u)\|^p\leq 2^{p-1}\{\|g(u)\|^p+\|v\|^p\},\forall g(u),g(v)\in K_g,
\end{eqnarray}
which is known as the lower parallelogram for the $l^p$-spaces.
In a similar way, one can obtain the upper  parallelogram law as
\begin{eqnarray}\label{eq4.3}
\|g(u)+g(v)\|^p+\mu \|g(v)-g(u)\|^p\geq 2^{p-1}\{\|g(u)\|^p+\|g(v)\|^p\},\forall g(u),g(v)\in K_g.
\end{eqnarray}
Similarly from  Definition 2.3, we can have
\begin{eqnarray}\label{eq4.4}
\|g(u)+g(v)\|^p+\mu \|g(v)-g(u)\|^p = 2^{p-1}\{\|g(u)\|^p+\|g(v)\|^p\},\forall g(u),g(v)\in K_g,
\end{eqnarray}
which is known as the parallelogram for the $l^p$-spaces. For the applications of the parallelogram laws for  Banach spaces in prediction theory and applied sciences, see \cite{21,22,23,24} and the references therein.\\ \\
In this section, we have introduced and studied a new class of convex  functions, which is called higher order strongly convex function.
We have the improved the results of  Lin and Fukushima \cite{65}.
It is shown  that several new classes of strongly convex functions can be obtained as special cases
of these higher order strongly general  convex functions. We have studied the basic properties of these functions. We have shown that one can  derive the parallelogram  laws in Banach spaces, which have applications in prediction theory and stochastic analysis. These parallelogram laws  can be used to characterize the $p$-uniform convexity and $q$-uniform smoothness of a Banach spaces.   The interested readers may explore the applications and other properties for the higher order  strongly convex functions in various fields of pure and applied sciences. This is an interesting direction of future research.

\section{Higher order general variational inequalities}
In this section, we consider a more general variational inequality of which (\ref{10.14}) is a special case.\\

For given two operators $T,g, $  we consider the problem of finding $u\in K $ for a constant $\mu> 0, $ such that
\begin{eqnarray}\label{eq11.1}
\langle Tu, g(v)-g(u) \rangle + \mu \|g(v)-g(u)\|^p \geq 0, \forall g(v)\in K, p>1,
\end{eqnarray}
which is called the higher order  general variational inequality, see \cite {139a}.\\
We note that,  if $\mu=0,$ then (\ref{eq11.1})  is equivalent to finding $u \in K, $ such that
\begin{eqnarray}\label{eq11.2}
\langle Tu, g(v)-g(u) \rangle \geq 0, \forall g(v)\in K,
\end{eqnarray}
which is known as the general variational inequality (\ref{2.5}), which was introduced and studied by Noor \cite{87} in 2008.\\
For suitable and appropriate choice of the parameter $\mu $ and $p,$  one can obtain several new and known
classes of variational inequalities, see \cite{87,88,90,110,122,123} and the references therein.
We note that the projection method and its variant forms can be used to study the higher order
strongly general variational inequalities (\ref{eq11.1}) due to its inherent structure. These facts motivated us to
consider the auxiliary principle technique, which is mainly due to Glowinski et al\cite{47} and Lions and
Stampacchia \cite{66} as developed by Noor \cite{122}. We use this technique to suggest some iterative methods
for solving the higher order general variational inequalities (\ref{eq11.1}).\\
For given $u\in K $  satisfying (\ref{eq11.1}), consider the problem of finding $w\in K, $ such that
\begin{eqnarray}\label{eq11.3}
\langle \rho Tw, g(v) - g(w)\rangle  + \langle w - u, v - w\rangle  + \nu \|g(v) - g(w)\|^p \geq 0, \forall g(v)\in K,
\end{eqnarray}
where $\rho  > 0 $ is a parameter. The problem (\ref{eq11.3}) is called the auxiliary higher order strongly general
variational inequality.
It is clear that the relation (\ref{eq11.3}) defines a mapping connecting the problems (\ref{eq11.1}) and (\ref{eq11.3}). We not that, if
$w= u, $ then $w$ is a solution of problem (\ref{eq11.1}). This simple observation enables to suggest an iterative
method for solving (\ref{eq11.1}).
\begin{algorithm}\label{alg11.1}. For given $u\in K,$  find the approximate solution $u_{n+1} $ by the scheme
\begin{eqnarray}\label{eq11.4}
\langle \rho T u_{n+1}, g(v) - g(u_{n+1})\rangle  + \langle u_{n+1} - u_n, v- u_{n+1}\rangle \\
  + \nu \|g(v) - g(u_{n+1})\|^p \geq 0.  \quad \forall g(v)\in K. \nonumber
\end{eqnarray}
\end{algorithm}
The Algorithm \ref{alg11.1} is known as an implicit method. Such type of methods have been studied extensively
for various classes of variational inequalities. See \cite{11,18,19} and the reference therein.
If $\nu =0, $ then Algorithm \ref{alg11.1}  reduces to:
\begin{algorithm}\label{alg11.2}   For given $u_0\in K, $ find the approximate solution $u_{n+1} $ by the scheme
\begin{eqnarray*}
\langle \rho Tu_{n+1}, g(v) - g(u_{n+1})\rangle  + \langle u_{n+1}u_n,  v - u_{n+1} \rangle  \geq  0,\\ \forall  g(v) \in K,
\end{eqnarray*}
\end{algorithm}
which appears to new ones even for solving the general variational inequalities (\ref{2.5}).\\
To study the convergence analysis of Algorithm \ref{alg11.1}, we need the following concept.

\begin{definition}\label{def11.5} The operator $T$ is said to be pseudo $g$-monotone with respect to $\mu \|g(v)-u\|^p,  $
if
\begin{eqnarray*}
&&\langle \rho Tu, g(v) - g(u) \rangle  + \mu \|g(v) - g(u)\|^p \geq 0, \forall g(v) \in K, p>1,\nonumber \\
&&\Longrightarrow  \nonumber \\
&& \langle \rho Tv, g(v) - g(u) \rangle  - \mu \|g(u) - g(v)\|^p \geq 0, \forall g(v) \in K.
\end{eqnarray*}
\end{definition}
If  $\mu =0, $ then Definition \ref{def11.5} reduces to:

\begin{definition}\label{def11.6} The operator $T$ is said to be pseudo $g$-monotone,  if
\begin{eqnarray*}
&&\langle \rho Tu, g(v) - gu) \rangle   \geq 0, \forall g(v) \in K \\
&&\Longrightarrow\\
&& \langle \rho Tv, g(v) - g(u) \rangle  \geq 0,  \forall g(v) \in K,
\end{eqnarray*}
\end{definition}
which appears to be a new one.\\
We now study the convergence analysis of Algorithm \ref{alg11.1}.
\begin{theorem}\label{theorem11.1}  Let $u\in K$   be a solution of (\ref{eq11.1}) and $u_{n+1} $ be the approximate solution obtained from
Algorithm \ref{alg11.1}.  If  $T$ is a pseudo $g$- monotone operator, then
\begin{eqnarray}\label{eq11.5}
\| u_{n+1}-u\|^2  \leq \|u_n-u\|^2 - \|u_{n+1}-u_n\|^2.
\end{eqnarray}
\end{theorem}
\begin{proof} Let $u \in K $ be a solution of (\ref{eq11.1}), then
\begin{eqnarray*}
\langle \rho Tu, g(v)-g(u) \rangle + \mu \|g(v)-g(u)\|^p, \forall g(v)\in K,
\end{eqnarray*}
implies that
\begin{eqnarray}\label{eq11.6}
\langle \rho Tv, g(u)-g(v) \rangle - \mu \|g(u)-g(v)\|^p, \forall g(v)\in K,
\end{eqnarray}
Now taking $v = u_{n+1} $ in (\ref{eq11.6}), we have
\begin{eqnarray}\label{eq11.7}
\langle \rho Tu_{n+1}, u_{n+1}-g(u)\rangle -\mu \|u_{n+1}-g(u)\|^p \geq 0.
\end{eqnarray}
Taking $v = u$  in (\ref{eq11.1}), we have
\begin{eqnarray}\label{eq11.8}
\langle \rho T u_{n+1}, g(u) - g(u_{n+1})\rangle  + \langle u_{n+1} - u_n, v- u_{n+1}\rangle \nonumber \\
 + \nu \|g(u) - g(u_{n+1})\|^p \geq 0. \forall g(v)\in K.
\end{eqnarray}
Combining (\ref{eq11.7}) and (\ref{eq11.8}), we have
\begin{eqnarray*}
\langle u_{n+1}-u_n, u_{n+1}-u \rangle \geq 0.
\end{eqnarray*}
Using the inequality $$  2\langle a,b\rangle =\|a+b\|^2-\|a\|^2-\|b\|^2, \forall a,b \in H, $$ we obtain
\begin{eqnarray*}
\| u_{n+1}-u\|^2 \leq \|u_n- u\|^2- \|u_{n+1}-u_n\|^2,
\end{eqnarray*}
the required result (\ref{eq11.5}).
\end{proof}

\begin{theorem}\label{thjeorem11.2}  Let the operator $T$ be a pseudo $g$-monotone. If $u_{n+1}$  be the approximate solution

obtained from Algorithm \ref{alg11.1}  and $u \in K $ is  the exact solution  (\ref{eq11.1}), then $\lim_{n\rightarrow \infty}u_n=u. $
\end{theorem}
\begin{proof}Let $u\in K $ be a solution of (\ref{eq11.1}). Then, from (\ref{eq11.5}), it follows that the sequence $\{\|u - u_n\|\}$ is
nonincreasing and consequently $\{u_n\}$  is bounded. From (\ref{eq11.5}), we have
\begin{eqnarray*}
\sum^{\infty}_{n=0} \|u_{n+1}-u_n\|^2 \leq \|u_0-u\|^2,
\end{eqnarray*}
from which, it follows that
\begin{eqnarray}\label{eq11.9}
\lim_{ n \rightarrow \infty}\|u_{n+1}-u_n\|= 0.
\end{eqnarray}
Let $\hat{u} $ be a cluster point of $ \{u_n\}$  and the subsequence $\{u_{n_j}\} $ of the sequence $u_n$  converge to $\hat{u} \in H.$
Replacing $u_n$  by $u_{n_j}$ in (\ref{eq4.4}), taking the limit $ n_j \rightarrow 0 $  and from (\ref{eq11.9}), we have
\begin{eqnarray*}
\langle T\hat{u}, g(v)-g(\hat{u}) \rangle  +\mu \|g(v)-g(\hat{u}) \|^p, \quad \forall g(v) \in K.
\end{eqnarray*}
This implies that $\hat{u} \in K $ K satisfies and
\begin{eqnarray*}
\|u_{n+1}-u_n \|^2 \leq \|u_n- \hat{u}\|^2.
\end{eqnarray*}
Thus it follows from the above inequality that the sequence $u_n$   has exactly one cluster point $\hat{u} $  and
$$ \lim_{n \rightarrow \infty}(u_n )= \hat{u}.  $$
\end{proof}
\bigskip
In order to implement the implicit Algorithm \ref{alg11.1}, one uses the predictor-corrector technique. Consequently,
Algorithm \ref{alg11.1} is equivalent to the following iterative method for solving the higher order strongly general variational
inequality (\ref{eq11.1}).
\begin{algorithm}\label{alg11.3} For a given $u_0 \in K, $ find the approximate solution $u_{n+1} $ by the schemes
\begin{eqnarray*}
\langle \rho Tu_n, g(v) - g(y_n) \rangle &+& \langle y_n - u_n v - y_n\rangle + \mu \|g(v)-g(y_n)\|^p \geq   0, \forall g(v) \in K, \\
\langle \rho Ty_n,  g(v) - (u_n)\rangle & +& \langle u_n-y_n,  v - u_n\rangle  \mu\|g(v) - g(u_n)\|^p  \geq 0, \forall g(v)\in K.
\end{eqnarray*}
\end{algorithm}
Algorithm \ref{alg11.3} is called the two-step iterative method and appears to be a new one.\\ \\
Using the auxiliary principle technique, on can suggest several iterative methods for solving the
higher order strongly general variational inequalities and related optimization problems. We have only
given a glimpse of the higher order strongly general variational inequalities. It is an interesting problem
to explore the applications of such type variational inequalities in various fields of pure and applied
sciences.

\section{Strongly exponentially general convex functions}

Convexity theory describes a broad spectrum of very interesting developments involving a
link among various fields of mathematics, physics, economics and engineering sciences. The
development of convexity theory can be viewed as the simultaneous pursuit of two different
lines of research. On the one hand, it is related to integral inequalities. It has been shown
that a function is a convex function, if and only if, it satisfies the Hermite-Hadamard type
inequality. These inequalities help us to derive the upper and lower bounds of the integrals.
On the other hand, the minimum of the differentiable convex functions on the convex
set can be characterized by the variational inequalities. the origin of
which can be traced back to Bernoulli's brothers, Euler and Lagrange. Variational
inequalities provide us a powerful tool to discuss the behaviour of solutions (regarding its existence,
uniqueness and regularity) to important classes of problems. Variational inequality
theory also enables us to develop highly efficient  powerful new numerical methods to
solve nonlinear problems. Recently various extensions and generalizations of convex functions and convex sets
have been considered and studied using innovative ideas and techniques. It is known that
more accurate and inequalities can be obtained using the logarithmically convex functions
than the convex functions. Closely related to the log-convex functions, we have the concept of
exponentially convex(concave) functions, which have important applications in information
theory, big data analysis, machine learning and statistic. Exponentially convex functions
have appeared significant generalization of the convex functions, the origin of which can be
traced back to Bernstein \cite{16}. Avriel\cite{9,10} introduced the concept of $r$-convex functions, from
which one can deduce the exponentially convex functions. Antczak \cite{2} considered the $(r, p)$
convex functions and discussed their applications in mathematical programming and optimization theory.
It is worth mentioning that the exponentially
convex functions have important applications in information sciences, data mining and statistics, see,
for example,\cite{1,2,9,10,132,133,134,135,136,137,138,153} and the references therein.\\

We would like to point out that the general convex functions and exponentially general
convex functions are two distinct generalizations of the convex functions, which have played
a crucial and significant role in the development of various branches of pure and applied
sciences. It is natural to unify these concepts. Motivated by these facts and observations, we now
 introduce a new class of convex functions, which is called exponentially general
convex functions in involving an arbitrary function. We discuss the basic properties of the
exponentially general convex functions. It is has been shown that the exponentially general
convex(concave) have nice properties which convex functions enjoy. Several new
concepts have been introduced and investigated. We prove that the local minimum of the
exponentially general convex functions is also the global minimum. \\
Noor and Noor \cite{132,133,134,135,136,137,138,153}  studied and investigated some  classes of strongly exponentially convex functions.
Inspired by the work of  Noor and Noor\cite{138}, we introduce and consider some  new  classes of higher order
strongly exponentially convex functions.  We establish the relationship between these classes and
derive some new results. We have also investigated  the optimality conditions for the higher order
strongly exponentially convex functions.  It is shown that the difference of strongly exponentially convex functions
 and strongly exponentially affine functions is again an exponentially convex function.  The optimal conditions
of the differentiable exponentially convex functions can be characterized by a class of variational
inequalities, called the exponentially general variational inequality, which is itself an
interesting problem. \\

We now define  the exponentially convex functions, which are mainly due to Noor and Noor \cite{132,133,134,135,136,137,138,153}.
\begin{definition}\cite{132,133,134,138}\label{de12.1} A function  $F$ is said to be exponentially convex function, if
\begin{eqnarray*}
e^{F((1-t)u+tv)} \leq (1-t)e^{F(u)}+ te^{F(v)}, \quad \forall u,v \in K, \quad t\in [0,1].
\end{eqnarray*}
\end{definition}
It is worth mentioning that Avriel \cite{9,10} and Antczak \cite{2} introduced the following concept.
\begin{definition}\cite{3,4}\label{def12.2} A function  $F$ is said to be
exponentially convex function, if
\begin{eqnarray*}
F((1-t)a+tb)\leq \log[(1-t)e^{F(a)}+te^{F(b)}] ,\quad \forall
a,b\in K,\quad t\in[0,1],
\end{eqnarray*}
\end{definition}
Avriel \cite{9,10} and Antczak \cite{2} discussed the application of the $1$-convex functions in the
mathematical programming. We note that the Definitions \ref{2.1} and \ref{2.2} are equivalent. A function
is called the exponentially concave function $f$, if $-f$ is exponentially convex function. \\

For the applications of the exponentially  convex(concave) functions in the mathematical programming and information theory.
For application in communication theory and information theory, see  Alirezaei and  Mathar\cite{1}.

\begin{example}\cite{1} The error function$$ erf(x)= \frac{2}{\sqrt{\pi}} \int^{x}_{0} e^{-t^2}dt, $$  becomes an exponentially concave function in the form $erf(\sqrt{x}), \quad x \geq 0, $ which describes  the bit/symbol error probability of communication systems  depending on the square root of the underlying signal-to-noise ratio. This shows that the exponentially concave functions can play important part in communication theory and information theory.
\end{example}

For the properties, generalizations and applications of the  various classes of exponentially convex functions,
see \cite{1,2,9,10,132,133,134,135,136,137,138,153}\\

It is clear  that the exponentially convex functions and general convex functions are two distinct generalizations of the convex functions.
It is natural to unify these concepts.  Motivated by this fact, we Noor and Noor \cite{138} introduced
 some new concepts of exponentially general convex functions. We include these results for the sake of completeness and for the convenience of the interested readers.
\begin{definition}\label{def12.3} A function  $F$ is said to be
exponentially strongly general convex function with respect to an arbitrary non-negative function $g, $  if
\begin{eqnarray*}
e^{F((1-t)g(u)+tg(v))}\leq (1-t)e^{F(g(u))}+te^{F(g(v))}. \quad \forall  g(u),g(v) \in K_g,  t\in[0,1].
\end{eqnarray*}
\end{definition}
or equivalently
\begin{definition}\label{def12.4} A function  $F$ is said to be
exponentially general convex function with respect to an arbitrary non-negative function $g, $ if,
\begin{eqnarray*}
F((1-t)g(u)+tg(v))\leq \log[(1-t)e^{F(g(u))}+te^{F(g(v))}], \quad \forall  g(u), g(v) \in K_g, t\in[0,1].
\end{eqnarray*}
\end{definition}
A function is called the exponentially general concave function $f$,  if $-f$ is exponentially general convex function.  \\
\begin{definition}\label{def12.5} A function  $F$ is said to be
exponentially  affine  general convex function with respect to an arbitrary non-negative function $g, $  if
\begin{eqnarray*}
e^{F((1-t)g(u)+tg(v))}=  (1-t)e^{F(g(u))}+te^{F(g(v))}, \forall  g(u),g(v) \in K_g,  t\in[0,1].
\end{eqnarray*}
\end{definition}
\medskip \medskip
If $g= I, $ the identity operator, then exponentially general convex functions reduce to the exponentially convex functions.\\
\begin{definition}\label{def12.6}
The function $F$ on the general convex  set $K_g$ is said to be exponentially general quasi-convex, if
\begin{eqnarray*}
e^{F(g(u)+t(g(v)-g(u)))}\leq \max\{e^{F(g(u))},e^{F(g(v))}\}, \quad \forall g(u),g(v)\in K_g, t\in[0,1].
\end{eqnarray*}
\end{definition}

\begin{definition}\label{def12.7}
The function $F$ on the general convex  set $K_g$ is said to be exponentially general  log-convex, if
\begin{eqnarray*}
e^{F(g(u)+t(g(v)-g(u)))}\leq (e^{F(g(u))})^{1-t} (e^{F(g(v))})^{t}, \quad \forall g(u),g(v)\in K_g, t\in[0,1],
\end{eqnarray*}
where $F(\cdot)>0.$
\end{definition}\label{def12.7}
From the above definitions, we have
\begin{eqnarray*}
e^{F(g(u)+t(g(v)-g(u))}&\leq& (e^{F(g(u))})^{1-t} (e^{F(g(v))})^{t}\\
&\leq& (1-t)e^{F(g(u))}+te^{F(g(v))})\\
&\leq& \max\{e^{F(g(u))},e^{F((g(v)))}\}.
\end{eqnarray*}
This shows that every   exponentially general  log-convex function is a   exponentially general convex function
and every  exponentially general convex function is a exponentially general   quasi-convex function. However, the converse is not true. \\
\noindent Let $K_g  =I_g=[g(a),g(b)]$ be the
interval. We now define the exponentially general convex functions on  $I_g$.
\begin{definition}\label{def12.8}
Let $I_g =[g(a),g(b)].$ Then $F$ is exponentially general convex function, if and only if,
\begin{eqnarray*}
\left|
\begin{array}{ccc}
1&1&1\\ g(a)&  g(x)&  g(b)\\
e^{F(g(a))}&  e^{F(g(x))}&e^{F(g(b))}
\end{array} \right|\geq0;\quad g(a)\leq g(x)\leq g(b).
\end{eqnarray*}
\end{definition}
\noindent One can easily show that the following are equivalent:
\begin{enumerate}
\item $F$ is exponentially general convex function.
\item $e^{F(g(x))}\leq  e^{F(g(a))}+\frac{e^{F(g(b))}-e^{F(g(a))}}{g(b)-g(a)}(g(x)-g(a)).$
\item
$\frac{e^{F(g(x))}-e^{F(g(a)}}{g(x)-g(a)}\leq  \frac{e^{F(g(b))}-e^{F(g(a))}}{g(b)-g(a)}.$
\item
$(g(x)-g(b))e^{F(g(a))} +(g(b)-g(a))e^{F(g(x))}+(g(a)-g(x))e^{F(g(b))})\geq 0.$
\item
$\frac{e^{F(g(a))}}{(g(b)-g(a))(g(a)-g(x))}+\frac{e^{F(g(x))}}{(g(x)-g(b))(g(a)-g(x))}+\frac{e^{F(g(b)}}{(g(b)-g(a))(g(x)-g(b))}\geq0,$
\end{enumerate}
where $g(x)= (1-t)g(a)+tg(b) \in[g(a),g(b)].$

\begin{theorem}\label{thm12.1} Let $F$  be a strictly exponentially general convex function. Then, any local minimum of $F $ is a global minimum.
\end{theorem}
\begin{proof}  Let the strictly exponentially convex  function $F$  have a local minimum at $g(u) \in K_g. $ Assume the
contrary, that is, $ F(g(v))<F(g(u)) $ for some $g(v) \in K_g.$  Since $F$  is strictly exponentially  general convex function, so
\begin{eqnarray*}
e^{F(g(u) + t(g(v)-g(u)))} < te^{F(g(v))} + (1 - t)e^{F(g(u))}, \quad  \mbox{for } \quad  0 < t < 1.
\end{eqnarray*}
Thus
$$ e^{F(g(u) + t(g(v0-g(u)))} - e^{F(g(u))} <- t[e^{F(g(v))} - e^{F(g(u))}] < 0, $$
from which it follows that
$$ e^{F(g(u) + t(g(v)-g(u0))} < e^{F(g(u))}, $$
for arbitrary small $ t > 0, $ contradicting the local minimum.
\end{proof}

\begin{theorem}\label{thm12.2} If the function $F $ on the general convex set $ K_g $ is exponentially general convex, then the level set
$$ L_{\alpha}  = \{g(u) \in K_g : e^{F(g(u) )} \leq  \alpha,  \quad \alpha \in R \} $$ is a general convex  set.
\end{theorem}
\begin{proof} Let $ g(u), g(v) \in L_{\alpha}. $  Then
$e^{F(g(u))} \leq \alpha  $   and  $e^{F(g(v))} \leq \alpha. $\\
Now, $\forall t \in (0, 1), \quad  g(w) = g(v) + t(g(u)-g(v))\in K_g, $ since  $K_g$  is a general convex  set.
Thus, by  the  exponentially general convexity of $F,$  we have
\begin{eqnarray*}
F e^{(g(v) + t(g(u)-g(v ))} &\leq &  (1 - t)e^{F(g(v))} + t e^{F(g(u))}\\
 & \leq & (1-t) \alpha + t\alpha =\alpha,
 \end{eqnarray*}
from which it follows that $g(v) + t(g(u)-g(v)) \in L_{\alpha} $ Hence $ L_{\alpha } $ is a general convex set.
\end{proof}

\begin{theorem}\label{thm12.3} The   function $F$ is   exponentially general convex function,  if and only if,
$$ epi(F) = \{(g(u), \alpha ): g(u) \in  K_g :  e^{F(g(u))} \leq \alpha, \alpha \in R \} $$ is a  general convex set.
\end{theorem}

\begin{proof} Assume that$F$ is exponentially general  convex function. Let $$(g(u), \alpha), \quad (g(v),\beta) \in epi(F). $$ Then it
follows that $e^{F(g(u))} \leq \alpha $ and  $e^{F(g(v))} \leq \beta .$  Hence,
 we have
$$  e^{F(g(u) + t(g(v)-g(u)))} \leq (1 - t)e^{F(g(u))} + t e^{F(g(v))} \leq  (1 - t)\alpha  + t \beta, $$
which implies that $$((1-t)g(u) + tg(v)), (1 - t)\alpha  + t\beta ) \in  epi(F). $$ Thus $ epi(F)$  is a general convex set.\\
Conversely, let $epi(F)$  be a general convex set. Let $ g(u), g(v) \in  K_g. $ Then  $(g(u), e^{F(g(u)}) \in epi(F) $ and
$(g(v, e^{F(g(v))}) \in epi(F). $ Since $epi(F) $ is a general  convex set, we must have
$$(g(u) + t(g(v)-g(u), (1 - t)e^{F(g(u))} + te^{F(g(v))} \in epi(F),  $$
which implies that
$$  e^{F((1-t)g(u) + tg(v))} \leq  (1 - t)e^{F(g(u))} + te^{F(g(u))}. $$
This shows that F is an exponentially  general convex function.
\end{proof}

\begin{theorem}\label{thm12.4}  The   function $F$  is exponentially general quasi-convex,  if and only if, the
level set $$L_{\alpha }  = \{g(u)\in K_g, \alpha \in R: e^{F(g(u))} \leq \alpha \} $$ is a general convex set.
\end{theorem}
\begin{proof}
Let $g(u), g(v) \in L_{\alpha}. $  Then $g(u), g(v) \in  K_g  $ and
 $\max(e^{F(g(u))}, e^{F(g(v))}) \leq \alpha . $ \\
 Now for $ t \in  (0, 1), g(w) = g(u) + t(g(v)-g( )u) \in K_g. $
 We have to prove that $ g(u) + t(g(v)-g(u)) \in L_{\alpha }. $
By the exponentially general convexity of  $F,$ we have
$$ e^{F(g(u) + t(g(v)-g(u)))} \leq \max{(e^{F(g(u))}, e^{F(g(v) )})} \leq \alpha , $$
which implies that $g(u) + t(g(v)-g(u)) \in L_{\alpha}, $  showing that the level set $L_{\alpha} $  is indeed a general convex set.\\

Conversely, assume that $L_{\alpha } $  is a general convex set. Then,  $\forall \quad g(u), g(v)  \in L_{\alpha} , t\in [0,1] , $
$ g(u) + t(g(v)- g(u)) \in L_{\alpha} . $  Let $ g(u), g(v)  \in L_{\alpha} $ for $$\alpha = max(e^{F(g(u))}, e^{F(g(v))}) \quad \mbox{and}\quad e^{F(g(v))} \leq e^{F(g(u))}. $$
Then,  from the definition of  the level set  $L_{\alpha} $,  it follows that
$$  e^{F(g(u) + t(g(v)-g(u))} \leq  \max{(e^{F(g(u))}, e^{F(g(v))}}) \leq  \alpha .  $$
Thus  $F $  is an exponentially  general quasi convex function. This completes the proof.
\end{proof}
\begin{theorem}\label{thm12.5} Let $F$ be an exponentially general convex function. Let $\mu =\inf_{u\in K}F(u)$. Then the set
$$E = \{g(u) \in  K_g : e^{F(g(u))}= \mu \} $$ is a general  convex set of  $ K_g.$ If $F $ is strictly exponentially general convex function , then $E $ is a singleton.
\end{theorem}
\begin{proof}  Let $ g(u), g(v) \in E.$  For $0 < t < 1,$ let $ g(w) = g(u) + t(g(v)-g(u)).$ Since $F$ is a exponentially general convex function, then
\begin{eqnarray*}F(g(w)) &=& e^{F(g(u) + t(g(v)-g(u))}\\
  &\leq &  (1 - t)e^{F(g(u))} + te^{F(g(v))}  = t \mu  + (1 - t)\mu = \mu,
 \end{eqnarray*}
which implies   $g(w) \in E . $ and hence $E$  is a general convex  set. For the second part, assume to the contrary that
$F(g(u)) = F(g(v)) = \mu. $  Since $K$ is a general convex  set, then for  $ 0 < t < 1, g(u) + t(g(v)-g(u)) \in K_g. $ Since $F$ is strictly
exponentially general convex function, so
\begin{eqnarray*}
 e^{F(g(u) + t(g(v)-g(u)))} &<& (1 - t)e^{F(g(u))} + te^{F(g(v))} =  (1 - t)\mu  + t\mu  =\mu.
 \end{eqnarray*}
This contradicts the fact that $\mu  = \inf_{g(u)\in K_g }F(u)  $ and hence the result follows.
\end{proof}

We now introduce the concept of the strongly exponentially general convex functions, which is the main motivation of this chapter.
\begin{definition}\label{def12.9}
A positive  function $F$ on the general convex set $K_g$ is said to be a
strongly exponentially general  convex  with respect to an arbitrary non-negative function $g, $  if there exists a constant $\mu>0, $
such that
\begin{eqnarray*}
e^{F(g(u)+t(g(v)-g(u)))}&\leq &(1-t)e^F(g(u))+te^F(g(v))\nonumber \\
&&-\mu \{t(1-t)\}\|g(v)-g(u)\|^2,\forall g(u),g(v)\in K_g, t\in[0,1].
\end{eqnarray*}
\end{definition}
The function $F$ is said to be   strongly exponentially general concave with respect to an arbitrary non-negative function $g, $ if and
only if, $-F$ is   strongly exponentially general  convex function with respect to an arbitrary non-negative function $g. $ \\
If $t=\frac{1}{2} $ and $\mu =1, $ then
\begin{eqnarray*}
e^{F(\frac{g(u)+g(v)}{2})} \leq \frac{e^F(g(u))+e^F(g(v))}{2} -\frac{1}{4}\|g(v)-g(u)\|^2, \\ \quad \forall g(u),g(v)\in K_g.
\end{eqnarray*}
The function $F$ is called the  strongly exponentially general $J$-convex function with respect to an arbitrary non-negative function $g $.\\

\begin{definition}\label{def12.10}
A positive  function is said to be  strongly exponentially  affine general  convex with respect to an arbitrary non-negative function $g, $  if there exists a constant $\mu>0, $ such that
\begin{eqnarray*}
e^{F(g(u)+t(g(v)-g(u)))}&= &(1-t)e^F(g(u))+te^F(g(v))\nonumber \\
&&-\mu \{t(1-t)\}\|g(v)-g(u)\|^2,\forall g(u),g(v)\in K_g, t\in[0,1].
\end{eqnarray*}
\end{definition}
If $t= \frac{1}{2}, $ then
\begin{eqnarray*}
e^{F(\frac{g(u)+g(v)}{2})}=  \frac{e^{F(g(u)}+e^{F(g(v))}}{2}-\frac{1}{4}\mu \|g(v)-g(u)\|^2,\\ \quad \forall g(u),g(v)\in K_g,
\end{eqnarray*}
then we say that the function $F$ is strongly exponentially affine  general $J$-convex function with respect to an arbitrary non-negative function $g, $

For the properties of the strongly exponentially general   convex functions in optimization,
inequalities and equilibrium problems, see  \cite{4,5,6,7,8,10,11,12,13,14,15,16,17,18,19,20,21} and the references therein..

\begin{definition}\label{def12.11}
A positive  function $F$ on the convex  set $K_g$ is said to be
strongly exponentially general quasi-convex,  if there exists a constant $\mu>0$ such that
\begin{eqnarray*}
e^{F(g(u)+t(g(v)-g(u)))}\leq \max\{e^{F(g(u))},e^{F(g(v))}\}-\mu \{t(1-t)\}\|g(v)-g(u)\|^2,\\ \forall g(u),g(v)\in K_g, t\in[0,1].
\end{eqnarray*}
\end{definition}

\begin{definition}\label{def12.12}
A positive  function $F$ on the general convex  set $K_g$ is said to be
strongly exponentially general log-convex, if there exists a constant $\mu>0$ such that
\begin{eqnarray*}
e^{F(g(u)+t(g(v)-g(u)))}\leq (e^{(F(g(u)))})^{1-t}(e^{(F(g(v)))})^{t}-\mu \{t(1-t)\}\|g(v)-g(u)\|^2,\\ \forall g(u),g(v)\in K_g, t\in[0,1],
\end{eqnarray*}
where $F(\cdot)>0.$
\end{definition}
From this Definition, we have
\begin{eqnarray*}
e^{F(g(u)+t(g(v)-g(u)))}&\leq & e^{(F(g(u)))^{1-t}}e^{(F(g(v)))^{t}}-\mu \{t(1-t)\}\|g(v)-g(u)\|^2, \\
&\leq  & (1-t)e^{F(g(u))}+ te^{F(g(v))}-\mu   \{t(1-t)\}\|g(v)-g(u)\|^2,
\end{eqnarray*}
This shows that every strongly exponentially general log-convex function is a   strongly exponentially general  convex function, but the converse is not true.\\

From above concepts, we have

\begin{eqnarray*}
e^{F(g(u)+t(g(v)-g(u)))}&\leq& (e^{(F(g(u)))})^{1-t}(e^{(F(g(v)))})^t-\mu \{t(1-t)\}\|g(v)-g(u)\|^2\\
&\leq& (1-t)e^{F(g(u))}+te^{F(g(v))}-\mu \{t(1-t)\}\|g(v)-g(u)\|^2\\
&\leq& \max\{e^{F(g(u))},e^{F(g(v))}\}-\mu \{t(1-t)\}\|g(v)-g(u)\|^2.
\end{eqnarray*}
This shows that every   strongly exponentially  general log-convex function is a   strongly exponentially convex function
and every strongly exponentially general convex function is a  strongly exponentially general   quasi-convex function. However, the converse is not true. \\

\begin{definition}\label{def12.12}
A differentiable function $F$ on the convex set $K_g$ is said to be a
 strongly exponentially general pseudoconvex function  with respect to an arbitrary non-negative function $g, $ if and only if, if there exists
a constant $\mu>0$ such that
\begin{eqnarray*}
 \langle e^{F(g(u))}F'(g(u)),g(v)-g(u)\rangle+\mu \|g(v)-g(u)\|^2&\geq & 0\\
 & \Rightarrow &\\
  e^{F(g(v))}- e^{F(g(u))}&\geq & 0,\quad \forall g(u),g(v)\in K_g.
\end{eqnarray*}
\end{definition}

\begin{theorem}\label{theorem12.1}
Let $F$ be a differentiable function on the convex set $K.$ Then the function $F$ is strongly exponentially
general convex function, if and only if,
\begin{eqnarray}\label{eq12.1}
e^{F(g(v))}-e^{F(g(u))} \geq \langle e^{F(g(u))} F^{\prime}(g(u)), g(v)-g(u) \rangle +\mu \|g(v)-g(u)\|^2, \forall g(u),g(v)\in K_g.
\end{eqnarray}
\end{theorem}
\begin{proof}
Let $F$ be a  strongly exponentially general  convex function. Then
\begin{eqnarray*}
e^{F(g(u)+t(g(v)-g(u)))}\leq (1-t)e^{F(g(u))}+te^{F(g(v))}- \mu t(1-t)\|g(v)-g(u)\|^2,\quad \forall g(u),g(v)\in K_g,  t\in [0,1]
\end{eqnarray*}
which can be written as
\begin{eqnarray*}
e^{F(g(v))}-e^{F(g(u))}\geq \{\frac{e^{F(g(u)+t(g(v)-g(u)))}-e^{F(g(u))}}{t}\} + \mu(1-t)\|g(v)-g(u)\|^2.
\end{eqnarray*}
Taking the limit in the above inequality as $t\rightarrow 0$ , we have
\begin{eqnarray*}
e^{F(g(v))}-e^{F(g(u))}\geq \langle e^{F(g(u))} F'(g(u)),g(v)-g(u))\rangle+ \mu \|g(v)-g(u)\|^2,
\end{eqnarray*}
which is (\ref{eq12.1}), the required result. \\

Conversely, let (\ref{eq12.1}) hold.  Then $ \forall g(u),g(v)\in K_g, t\in [0,1],$
$g(v_t) =g(u)+t(g(v)-g(u))\in K_g, $  we have
\begin{eqnarray}\label{3.1}
e^{F(g(v))}-e^{F(g(v_t))} &\geq& \langle
e^{F(g(v_t))}F'(g(v_t)),g(v)-g(v_t))\rangle+\mu \|g(v)-g(u)\|^2\nonumber\\
&=&(1-t)\langle e^{F(g(v_t))}F'(g(v_t)),g(v)-g(u)\rangle+ \mu(1-t)^2 \|g(v)-g(u)\|^2.
\end{eqnarray}
In a similar way, we have
\begin{eqnarray}\label{eq12.2}
e^{F(g(u))}-e^{F(g(v_t))}&\geq& \langle e^{F(g(v_t))}F'(g(v_t)),g(u)-g(v_t))\rangle+\mu \|g(u)-g(v_t)\|^2\nonumber \\
&=&-t\langle e^{F(g(v_t))}F'(g(v_t)),g(v)-g(u)\rangle+ \mu t^2 \|g(v)-g(u)\|^2.
\end{eqnarray}
Multiplying (\ref{eq12.1}) by $t$ and (\ref{eq12.2}) by $(1-t)$ and adding
the resultant, we have
\begin{eqnarray*}
e^{F(g(u)+t(g(v)-g(u)))}\leq (1-t)e^{F(g(u))}+te^{F(g(v))}- \mu t(1-t)\|g(v)-g(u)\|^2,
\end{eqnarray*}
showing that $F$ is a  strongly exponentially general convex function.
\end{proof}

\begin{theorem}\label{theorem12.2}
Let $F$ be differentiable  strongly exponentially general convex on the convex set $K_g. $ Then
\begin{eqnarray}\label{eq12.3}
\langle e^{F(g(u))}F'(g(u))-e^{F(g(v))}F'(g(v)),g(u)-g(v)\rangle \geq 2\mu \|g(v)-g(u)\|^2, \quad \forall g(u),g(v)\in K_g.
\end{eqnarray}
\end{theorem}
\begin{proof}
Let $F$ be a strongly exponentially general convex function. Then, from Theorem \ref{theorem12.1}, we have
\begin{eqnarray}\label{eq12.4}
e^{F(g(v))}-e^{F(g(u))}\geq \langle e^{F(g(u))} F'(g(u)),g(v)-g(u)\rangle+ \mu \|g(v)-g(u)\|^2,\quad \forall g(u),g(v) \in K_g.
\end{eqnarray}
Changing the role of $u$ and $v$ in (\ref{eq12.4}), we have
\begin{eqnarray}\label{eq12.5}
e^{F(g(u))}-e^{F(g(v))} \geq \langle e^{F(g(v))}F'(g(v)),g(u)-g(v))\rangle+ \mu \|g(u)-g(v)\|^2,\forall g(u),g(v) \in K_g.
\end{eqnarray}
Adding (\ref{eq12.5}) and (\ref{eq12.4}), we have
\begin{eqnarray*}\label{3.5}
\langle e^{F(g(u))}F'(g(u))-e^{F(g(v))}F'(g(v)),g(u)-g(v)\rangle \geq 2\mu \|g(v)-g(u)\|^2,
\end{eqnarray*}
the required (\ref{eq12.3}). \\
\end{proof}
We point out the converse of Theorem \ref{theorem12.2} is not true expect for $p=2.$ In fact, we have the following result.\\
\begin{theorem}\label{theorem12.3} If the differential of a  strongly exponentially general convex function satisfies
\begin{eqnarray}\label{eq12.6}
\langle e^{F(g(u))}F'(g(u))-e^{F(g(v))}F'(g(v)),g(u)-g(v)\rangle \geq 2\mu \|g(v)-g(u)\|^2, \forall g(u),g(v) \in K_g.
\end{eqnarray}
then
\begin{eqnarray}\label{eq12.7}
e^{F(g(v))}-e^{F(g(u))} \geq \langle e^{F(g(u))}F'(g(u)),g(v)-g(u)\rangle +\mu \|g(v)-g(u)\|^2, \forall g(u),g(v) \in K_g.
\end{eqnarray}
\end{theorem}

\begin{proof}
Let $F'(.)$ satisfy (\ref{eq12.6}). Then
\begin{eqnarray}\label{eq12.8}
\langle e^{F(g(v))}F'(g(v)),g(u)-g(v)\rangle \leq \langle
e^{F(g(u))}F'(g(u)),g(u)-g(v))\rangle- 2 \mu \|g(v)-g(u)\|^2,
\end{eqnarray}
Since $K_g$ is a general  convex set, $\forall g(u),g(v) \in K_g, t\in [0,1]$
$ g(v_t)= g(u)+t(g(v)-g(u))\in K_g.$ \\Taking $ g(v)= g(v_t)$ in (\ref{eq12.8}), we have
\begin{eqnarray*}
\langle e^{F(g(v_t))}F'(g(v_t)), g(u)-g(v_t)\rangle &\leq& \langle
e^{F(g(u))}F'(g(u)), g(u)-g(v_t)\rangle-2\mu \|g(v_t)-g(u)\|^2\nonumber\\
&=&-t \langle e^{F(g(u))}F'(g(u)),g(v)-g(u)\rangle-2t^2 \mu\|g(v)-g(u)\|^2,
\end{eqnarray*}
which implies that
\begin{eqnarray}\label{eq12.9}
\langle e^{F(g(v_t))}F'(g(v_t)),g(v)-g(u)\rangle\geq \langle e^{F(g(u))} F'(g(u)),g(v)-g(u)\rangle +2t \mu \|g(v)-g(u)\|^2.
\end{eqnarray}
Consider the auxiliary function
\begin{eqnarray}\label{eq12.10}
\xi(t)=e^{F(g(u)+t(g(v)-g(u)))}, \forall g(u),g(v) \in K_g,
\end{eqnarray}
from which, we have
\begin{eqnarray*}
\xi (1)= e^{F(g(v))}, \quad \xi(0)= e^{F(g(u))}.
\end{eqnarray*}
Then, from (\ref{eq12.9}) and (\ref{eq12.10}), we have
\begin{eqnarray}\label{eq12.11}
\xi'(t)&=&\langle e^{F(g(v_t))}F'(g(v_t), g(v)-g(u)\rangle \nonumber \\
  &\geq & \langle e^{F(g(u))}F'(g(u)),g(v)-g(u)\rangle +2\mu t \|g(v)-g(u)\|^2.
\end{eqnarray}
Integrating (\ref{eq12.11}) between 0 and 1, we have
\begin{eqnarray*}
\xi(1)-\xi(0)= \int^{1}_{0}\xi^{\prime}(t)dt \geq \langle e^{F(g(u))}F'(g(u)),g(v)-g(u)\rangle +\mu \|g(v)-g(u)\|^2.
\end{eqnarray*}
Thus it follows that
\begin{eqnarray*}
e^{F(g(v))}-e^{Fg((u))} \geq \langle e^{F(g(u))}F'(g(u)),g(v)-g(u)\rangle +\mu \|g(v)-g(u)\|^2,
\end{eqnarray*}
which is the required (\ref{eq12.7}).
\end{proof}

Theorem \ref{theorem12.1} and Theorem \ref{theorem12.2} enable us to introduce the following new concepts.\\
\begin{definition}
The differential $F^{\prime}(.) $ of the strongly exponentially convex functions is said to be  strongly exponentially monotone,
if there exists a constant $\mu >0, $ such that
\begin{eqnarray*}\label{3.5}
\langle e^{F(g(u))}F'(g(u))-e^{F(g(v))}F'(g(v)),g(u)-g(v)\rangle \geq \mu \|g(v)-g(u)\|^2, \forall u,v \in H.
\end{eqnarray*}
\end{definition}
\begin{definition}
The differential $F^{\prime}(.) $ of the  exponentially convex functions is said to be   exponentially monotone, if
\begin{eqnarray*}\label{3.5}
\langle e^{F(g(u))}F'(g(u))-e^{F(g(v))}F'(g(v)),g(u)-g(v)\rangle \geq 0, \forall u,v \in H.
\end{eqnarray*}
\end{definition}
\begin{definition}
The differential $F^{\prime}(.) $ of the strongly exponentially convex functions is said to be   strongly exponentially pseudo monotone,  if
\begin{eqnarray*}
\langle e^{F(g(u))}F'(g(u)),g(v)-g(u)\rangle \geq 0.
\end{eqnarray*}
implies that
\begin{eqnarray*} \label{3.9}
\langle e^{F(g(v))} F'(g(v)),g(v)-g(u) \rangle \geq  \mu \|g(v)-g(u)\|^2, \quad \forall g(u),g(v) \in K_g.
\end{eqnarray*}
\end{definition}
We now give a necessary condition for strongly exponentially pseudo-convex
function.

\begin{theorem}\label{theorem12.4}
Let $F'$ be a strongly exponentially  pseudomonotone operator. Then $F$ is a   strongly exponentially  general
pseudo-invex function.
\end{theorem}
\begin{proof}
Let $F'$ be a  strongly exponentially  pseudomonotone operator. Then
\begin{eqnarray*}
\langle e^{F(g(u))}F'(g(u)),g(v)-g(u)\rangle \geq0, \forall g(u),g(v) \in K_g,
\end{eqnarray*}
implies that
\begin{eqnarray} \label{eq12.12}
\langle e^{F(g(v))} F'(g(v)),g(v)-g(u) \rangle \geq  \mu \|g(v)-g(u)\|^2.
\end{eqnarray}
Since $K_g$ is a general  convex set, $ \forall g(u),g(v) \in K_g, t \in [0,1],$
$ g(v_t)= g(u)+t(g(v)-g(u))\in K_g.$ \\Taking $ g(v)=g( v_t)$ in (\ref{eq12.12}), we have
\begin{eqnarray} \label{eq12.13}
\langle e^{F(g(v_t))} F'(g(v_t)),g(v)-g(u)\rangle \geq  t \mu \|g(v)-g(u)\|^2.
\end{eqnarray}
Consider the auxiliary function
\begin{eqnarray*}
\xi(t)=e^{F(g(u)+t(g(v)-g(u)))}= e^{F(g(v_t))},\quad \forall g(u),g(v) \in K_g,  t\in [0,1],
\end{eqnarray*}
which is differentiable,  since $F$ is differentiable function.
Thus,  we have
\begin{eqnarray*}
\xi'(t)=\langle e^{F(g(v_t))}F'(g(v_t)),g(v)-g(u))\rangle \geq  t\mu \|g(v)-g(u)\|^2.
\end{eqnarray*}
Integrating the above relation between 0 to 1, we have
\begin{eqnarray*}
\xi(1)-\xi(0)= \int^{1}_{0} xi^{\prime}(t)dt \geq \frac{\mu}{2}\|v-u\|^2,
\end{eqnarray*}
that is,
\begin{eqnarray*}
e^{F(g(v))}-e^{F(g(u))}  \geq \frac{\mu}{2}\|g(v)-g(u)\|^2,
\end{eqnarray*}
showing that $F$ is a   strongly exponentially general  pseudo-convex
function.
\end{proof}

\begin{definition}
The function $F$ is said to be sharply  strongly exponentially pseudoconvex, if there exists a constant $\mu>0, $ such that
\begin{eqnarray*}
 \langle e^{F(g(u))}F'(g(u)),g(v)-g(u)\rangle&\geq &0 \\
 &\Rightarrow & \\
 F(g(v))&\geq & e^{F(g(v)+t(g(u)-g(v)))}+\mu t(1-t)\|g(v)-g(u)\|^2\quad \forall g(u),g(v) \in K_g,  t \in [0,1].
\end{eqnarray*}
\end{definition}

\begin{theorem}
Let $F$ be a sharply   strongly exponentially pseudoconvex function with a constant $\mu >0.$ Then
\begin{eqnarray*}
 \langle e^{F(g(v))}F'(g(v)),g(v)-g(u)\rangle\geq \mu \|g(v)-g(u)\|^2, \quad \forall g(u),g(v) \in K_g.
\end{eqnarray*}
\end{theorem}
\begin{proof}
Let $F$ be a sharply   strongly  exponentially general pseudoconvex function. Then
\begin{eqnarray*}
e^{F(g(v))}\geq e^{F(g(v)+t(g(u)-g(v)))}+\mu t(1-t) \|g(v)-g(u)\|^2,\quad \forall g(u),g(v) \in K_g, t\in [0,1].
\end{eqnarray*}
from which we have
\begin{eqnarray*}
\{\frac{e^{F(g(v)+t(g(u)-g(v)))}-e^{F(g(v))}}{t}\}+\mu (1-t) \|g(v)-g(u)\|^2\leq 0.
\end{eqnarray*}
Taking limit in the above inequality, as $t \rightarrow 0,$ we have
\begin{eqnarray*}
 \langle e^{F(g(v))}F'(g(v)),g(v)-g(u)\rangle\geq \mu \|g(v)-g(u)\|^2,
\end{eqnarray*}
 the required result.
\end{proof}

We now discuss the optimality condition for the differentiable  strongly exponentially convex functions.\\
\begin{theorem} \quad Let $F $ be a differentiable  strongly exponentially convex function.  If $u \in K $ is the minimum of the function $F, $ then
\begin{eqnarray}\label{eq12.14}
e^{F(g(v))}-e^{F(g(u))} \geq \mu \|g(v)-g(u)\|^p, \quad \forall g(u),g(v) \in K_g.
\end{eqnarray}
\end{theorem}
\begin{proof} Let $u \in H: g(u)\in K_g $ be a minimum of the function $F. $ Then
\begin{eqnarray*}\label{eq12.15}
F(g(u))\leq F(g(v)), \forall g(u),g(v) \in K_g,
\end{eqnarray*}
from which, we have
\begin{eqnarray}\label{eq12.16}
e^{F(g(u))}\leq e^{F(gv))}, \forall g(u),g(v) \in K_g.
\end{eqnarray}

Since $K _g $ is a convex set, so, $ \forall g(u),g(v) \in K_g, \quad t\in [0,1], $
$$ g(v_t) = (1-t)g(u)+ tg(v) \in K_g.$$
Taking $g(v) = g(v_t) $ in (\ref{eq12.16}), we have
\begin{eqnarray} \label{eq12.17}
0 \leq  \lim_{t \rightarrow 0}\{\frac{e^{F(g(u)+t(g(v)-g(u)))}-e^{F(g(u))}}{t}\}  =  \langle e^{F(g(u))} F^{\prime}(g(u)), g(v)-g(u) \rangle.
\end{eqnarray}
Since $F $ is differentiable  strongly exponentially general convex function, so
\begin{eqnarray*}
e^{F(g(u)+t(g(v)-g(u)))} \leq  e^{F(g(u))}+ t(e^{F(g(v))}-e^{F(g(u))}) -\mu t(1-t)\|g(v)-g(u)\|^2, \\
\quad \forall g(u),g(v) \in K_g,  t \in [0,1],
\end{eqnarray*}
from which, using (\ref{eq12.17}), we have
\begin{eqnarray*}
e^{F(g(v))}-e^{F(g(u))} &\geq & \lim_{t \rightarrow 0}\{\frac{e^{F(g(u)+t(g(v)-g(u)))}-e^{F(g(u))}}{t} \}+ \mu \|g(v)-g(u)\|^2.\\
&=& \langle e^{F(g(u))}F^{\prime}(g(u)), g(v)- g(u)\rangle + \mu \|g(v)-g(u)\|^p \geq  \mu \|g(v)-g(u)\|^2,
\end{eqnarray*}
the required result (\ref{eq12.15}).
\end{proof}

\begin{remark} \quad We would like to mention that, if  $u \in H: g(u)\in K_g $ satisfies
 \begin{eqnarray}\label{eq12.19}
 \langle e^{F(g(u))}F^{\prime}(g(u)), g(v)-g(u) \rangle + \mu \|g(v)-g(u)\|^2 \geq 0, \quad \forall g(u),g(v) \in K_g,
 \end{eqnarray}
  then $u \in H: g(u)\in K_g $ is the minimum of the  function $F. $ The inequality of the type (\ref{eq12.19}) is called the strongly exponentially variational inequality. It is an interesting problem to study the existence of the inequality (\ref{eq12.19}) and to develop numerical methods for solving the strongly exponentially variational inequalities.
\end{remark}

 We would like to point that the strongly exponentially convex is also a strongly Wright   general  convex functions. From the definition \ref{def12.9}, we have
 \begin{eqnarray*}\label{ea1}
e^{F((1-t)g(u)+tg(v))}+ e^{F(tg(u)+(1-t)g(v))}&\leq &e^F(g(u))+e^F(g(v))\\
&&-2\mu t(1-t) \|g(v)-g(u)\|^2, \forall g(u),g(v) \in K_g, t\in[0,1],
\end{eqnarray*}
which is called  the  strongly Wright  exponentially convex function.  It is an interesting problem to study the properties and applications of the
strongly Wright   exponentially general convex functions.

\section{Generalizations and extensions}
We would like to mention that some of the results obtained and presented in
this paper can be extended for more strongly general  variational inequalities.
To be more precise, for a given nonlinear operator $T,A,g ,$ consider the problem of finding  $ u\in H: g(u) \in K $ such  that
\begin{eqnarray}\label{eq13.1}
\langle Tu, g(v)-g(u) \rangle \geq \langle A(u), g(v)-g(u) \rangle + \mu \|g(v)-g(u)\|^2,\quad  \forall v\in H: g(v) \in K,
\end{eqnarray}
which is called the strongly  general variational inequality.\\
If $\mu = 0, $ then problem (\ref{eq13.1}) reduces to
\begin{eqnarray}\label{eq13.1a}
\langle Tu, g(v)-g(u) \rangle \geq \langle A(u), g(v)-g(u) \rangle ,\quad  \forall v\in H: g(v) \in K,
\end{eqnarray}
is called the general strongly  variational inequalities. \\
 We would like to mention that one can obtain various classes of
general  variational inequalities for appropriate and suitable
choices of the operators $T, A, g. $ \\ \\
 Using Lemma \ref{lemma1}, one can show that the problem (\ref{eq13.1a}) is equivalent to finding $ u\in H: g(u) \in K $ such  that
\begin{eqnarray}\label{eq13.2}
g(u) = P_K[g(u)- \rho( Tu-A(u))].
\end{eqnarray}
These alternative formulations can be used to suggest and analyze similar
techniques for solving general strongly  variational inequalities
(\ref{eq13.1a}) as considered in this paper under certain extra conditions. A complete
study of these algorithms for problem (\ref{eq13.1a}) is the subject of subsequent research.
 Development of efficient and implementable algorithms for problems (\ref{eq13.1a})  need further research efforts.\\
 \noindent{\bf (I).} For  given nonlinear operators $T,A,g ,$ consider the problem of finding  $ u\in H: g(u) \in K $ such  that
\begin{eqnarray}\label{eq13.3}
\langle Tu, v-g(u) \rangle \geq \langle A(u), v-g(u) \rangle ,\quad  \forall v\in K,
\end{eqnarray}
which is also called the strongly general variational inequality.\\

\noindent{\bf (III).} For  given nonlinear operators $T,A,g ,$ consider the problem of finding  $ u\in H: g(u) \in K $ such  that
\begin{eqnarray}\label{eq13.4}
\langle Tu, g(v)-u \rangle \geq \langle A(u), g(v)-u \rangle ,\quad  \forall v\in H:  g(v) \in K,
\end{eqnarray}
is also known as the strongly general variational inequality.\\
\begin{remark}
We would like to point out that
the problems (\ref{eq13.1a}), (\ref{eq13.3}) and (\ref{eq13.4}) are quite distinct and different from each other and have significant
 applications in various branches of pure and applied sciences. They are open and interesting problems for future research.
We would like to emphasize that problems (\ref{eq13.1a}), (\ref{eq13.3}) and (\ref{eq13.4}) are equivalent in many respect and
share the basic and fundamental properties. In particular, they have the same
equivalent fixed-point formulations. Consequently, most of the results obtained in this paper continue to hold for these problems  with minor
modifications.\\
\noindent{\bf(IV).} If $ K= H, $ then the  problem (\ref{eq13.1}) is equivalent to finding $u \in H: g(u) \in H, $ such that
\begin{eqnarray} \label{eq13.5}
\langle Tu, g(v) \rangle = \langle A(u), g(v) \rangle, \quad \forall v\in H:g(v) \in H,
\end{eqnarray}
which can be viewed as the representation theorem for the nonlinear functions involving an arbitrary function $g.$
For more details, see Noor and Noor\cite{140}. \\
\noindent{\b(V).} If $A(u)= |u|, $ then the  problem (\ref{eq13.5}) is equivalent to finding $u \in H: g(u) \in H, $ such that
\begin{eqnarray} \label{eq13.6}
\langle Tu, g(v) \rangle = \langle A|u|, g(v) \rangle, \quad \forall v\in H:g(v) \in H,
\end{eqnarray}
is known as the generalized absolute value equation. See Batool et al \cite{13}  more details.  \\ \\
The theory of  general variational inequalities does not appear to have developed to an extent that
it provides a complete framework for studying these problems.  Much
more research is needed in all these areas to develop a sound basis for applications. We have
not treated variational inequalities for time-dependent problems and
the spectrum analysis of the variational inequalities.
In fact, this field has been continuing and will continue to foster new, innovative and novel
applications in various branches of pure and applied sciences. We have given only a brief
introduction of this fast growing field. The interested reader is advised to explore this field
further. It is our hope that this brief introduction may inspire and motivate the reader to
discover new and interesting applications of the general variational inequalities and related optimization problems in other areas of sciences.

\end{remark}

\section*{Acknowledgements}
We wish to express our  deepest gratitude to our
colleagues, collaborators and friends, who have direct or indirect contributions
in the process of this paper. We are also grateful to Rector, COMSATS University Islamabad, Pakistanm
 for the research facilities and support in our research endeavors. Authors would like to express their sincere thanks to the referees for their valuable suggestions and comments.
\bibliographystyle{amsplain}

\end{document}